\documentclass[12pt,reqno]{amsart}

\usepackage[colorlinks=true,urlcolor=blue,bookmarks=true,bookmarksopen=true,citecolor=blue]{hyperref}

\usepackage{graphicx}
\usepackage[usenames,dvipsnames]{color}
\usepackage[all]{xy}
\usepackage{amsmath}
\usepackage{amsfonts,amssymb}
\usepackage{amsthm}
\usepackage{bbm}
\usepackage{mathrsfs}
\usepackage{marginnote}
\usepackage{pdfsync}
\usepackage{xfrac}
\usepackage{varioref}
\usepackage{a4wide}
\usepackage{url}
\usepackage{tikz, ifthen}

\usepackage[utf8]{inputenc}

\usepackage{parskip}

\usepackage{bbm}
\usepackage{tikz, ifthen, pgfplots, caption}
\pgfplotsset{compat=1.13}
\usepgfplotslibrary{fillbetween}
\usetikzlibrary{decorations.markings, intersections, calc, patterns, arrows}

\makeatletter
\tikzset{
  use path for main/.code={%
    \tikz@addmode{%
      \expandafter\pgfsyssoftpath@setcurrentpath\csname tikz@intersect@path@name@#1\endcsname
    }%
  },
  use path for actions/.code={%
    \expandafter\def\expandafter\tikz@preactions\expandafter{\tikz@preactions\expandafter\let\expandafter\tikz@actions@path\csname tikz@intersect@path@name@#1\endcsname}%
  },
  use path/.style={%
    use path for main=#1,
    use path for actions=#1,
  }
}
\makeatother

\tikzset{lift/.style={postaction={decorate,decoration={
         markings,mark=at position .8 with {\arrow[scale=3]{>}}
         }}},
         leaf/.style={dashed,thin,postaction={decorate,decoration={
         markings,mark=at position 0.9999 with {\arrow[scale=2]{>}}
         }}}}

\pgfdeclarepatternformonly{mydots}{\pgfqpoint{-1pt}{-1pt}}{\pgfqpoint{1pt}{1pt}}{\pgfqpoint{5pt}{5pt}}
{
	\pgfpathcircle{\pgfqpoint{0pt}{0pt}}{.5pt}
	\pgfusepath{fill}
}

\usetikzlibrary{calc, decorations.pathmorphing,  intersections, decorations.markings}

\makeatletter
\tikzset{
  use path for main/.code={%
    \tikz@addmode{%
      \expandafter\pgfsyssoftpath@setcurrentpath\csname tikz@intersect@path@name@#1\endcsname
    }%
  },
  use path for actions/.code={%
    \expandafter\def\expandafter\tikz@preactions\expandafter{\tikz@preactions\expandafter\let\expandafter\tikz@actions@path\csname tikz@intersect@path@name@#1\endcsname}%
  },
  use path/.style={%
    use path for main=#1,
    use path for actions=#1,
  }
}
\makeatother

\tikzset{lift/.style={postaction={decorate,decoration={
         markings,mark=at position .8 with {\arrow[scale=3]{>}}
         }}},
         leaf/.style={dashed,ultra thin,postaction={decorate,decoration={
         markings,mark=at position 0.9999 with {\arrow[scale=2]{>}}
         }}}}

\newcommand{\id}{id}

\makeatletter
\DeclareFontFamily{OMX}{MnSymbolE}{}
\DeclareSymbolFont{MnLargeSymbols}{OMX}{MnSymbolE}{m}{n}
\SetSymbolFont{MnLargeSymbols}{bold}{OMX}{MnSymbolE}{b}{n}
\DeclareFontShape{OMX}{MnSymbolE}{m}{n}{
    <-6>  MnSymbolE5
   <6-7>  MnSymbolE6
   <7-8>  MnSymbolE7
   <8-9>  MnSymbolE8
   <9-10> MnSymbolE9
  <10-12> MnSymbolE10
  <12->   MnSymbolE12
}{}
\DeclareFontShape{OMX}{MnSymbolE}{b}{n}{
    <-6>  MnSymbolE-Bold5
   <6-7>  MnSymbolE-Bold6
   <7-8>  MnSymbolE-Bold7
   <8-9>  MnSymbolE-Bold8
   <9-10> MnSymbolE-Bold9
  <10-12> MnSymbolE-Bold10
  <12->   MnSymbolE-Bold12
}{}

\let\llangle\@undefined
\let\rrangle\@undefined
\DeclareMathDelimiter{\llangle}{\mathopen}%
                     {MnLargeSymbols}{'164}{MnLargeSymbols}{'164}
\DeclareMathDelimiter{\rrangle}{\mathclose}%
                     {MnLargeSymbols}{'171}{MnLargeSymbols}{'171}
\makeatother
\renewenvironment{proof}[1][\proofname]{\noindent{\bfseries #1.  }}{\qed}

\DeclareMathOperator{\Ima}{Im}

\newcommand{\N}{\mathbbm{N}}                     
\newcommand{\Z}{\mathbbm{Z}}                     
\newcommand{\R}{\mathbbm{R}}                     
 
\newcommand{\fix}{\mathrm{Fix}}               
\newcommand{\Ham}{\mathrm{Ham}}    
\newcommand{\HF}{\mathrm{HF}}    

\newcommand{\Comp}{\mathrm{Comp}}
\newcommand{\Tu}{\mathrm{T}}

\newcommand{\overbar}[1]{\mkern 1.5mu\overline{\mkern-1.5mu#1\mkern-1.5mu}\mkern 1.5mu}

\newcommand{\topo}{\mathrm{top}}

\newcommand{\hofer}{\mathrm{Hofer}}

\newcommand{\dom}{\mathrm{dom}}
\newcommand{\sing}{\mathrm{sing}}
\newcommand{\selfi}{\mathrm{si}}

\newcommand{\per}{\mathrm{per}}
\newcommand{\Per}{\mathrm{Per}}
\newcommand{\hp}{\mathrm{h}}
\newcommand{\HP}{\mathrm{H}}
\newcommand{\hyp}{\mathrm{hyp}}

\newtheorem{thm}{Theorem}[section]               
\newtheorem*{thm*}{Theorem}               
\newtheorem{cor}[thm]{Corollary}        
\newtheorem*{cor*}{Corollary}        
\newtheorem{lem}[thm]{Lemma}            
\newtheorem{prop}[thm]{Proposition}     
\theoremstyle{definition}
\newtheorem{defn}[thm]{Definition}      
\newtheorem{rem}[thm]{Remark}           
 
\newtheorem{fact}[thm]{Fact}          
\newcounter{claim}
\newenvironment{claim}[1][]{\refstepcounter{claim}\par\noindent\underline{Claim~\theclaim:}\space#1}{}

\newenvironment{myequation}{\begin{equation*} \begin{gathered}}{\end{gathered} \end{equation*}}

\author{Arnon Chor}
\thanks{A. Chor was partially supported by the Israel Science Foundation grant 667/18}
\address{Arnon Chor, School of Mathematical Sciences, Tel Aviv University, Ramat Aviv, Tel Aviv 69978, Israel.}
\email{\texttt{arnonchor@gmail.com}}

\author{Matthias Meiwes}
\thanks{M. Meiwes was partially supported by the Israel Science Foundation grant 2026/17}
\address{Matthias Meiwes,
Chair for Geometry and Analysis, RWTH Aachen University, Jakobstrasse 2,
DE-52064 Aachen, Germany.}
\email{\texttt{meiwes@mathga.rwth-aachen.de}}

\title{Hofer's geometry and topological entropy}

\begin{document}
\subjclass[2020]{37E30, 37J46}

\begin{abstract}
In this article we study persistence features of topological entropy and periodic orbit growth of Hamiltonian diffeomorphisms on surfaces with respect to Hofer's metric. We exhibit stability of these dynamical quantities in a rather strong sense for a specific family of maps introduced by Polterovich and Shelukhin. 
A crucial ingredient comes from some enhancement of lower bounds for the topological entropy and orbit growth forced by a periodic point, formulated in terms of the geometric self-intersection number and a variant of Turaev's cobracket of the free homotopy class that it induces. 
Those bounds are obtained within the framework of Le Calvez and Tal's forcing theory.

\end{abstract}

\maketitle

\section{Introduction and Results}

Prototypical examples of maps that define a dynamical system of chaotic nature are horseshoe maps, introduced by Smale  \cite{smale}.
In dimension $2$, a horseshoe model is given by stretching a square horizontally and squeezing it vertically before folding it back in the shape of a horseshoe. 
Iterating a horseshoe $T$ has the features of a topologically chaotic system: There is a rich symbolic dynamics on a compact invariant set and in particular there is  exponential growth  of periodic orbits and positive  topological entropy.  
A remarkable property of horseshoes is their local stability. Structural stability asserts that a diffeomorphism $T'$ that is sufficiently $C^1$-close to $T$ contains the same symbolic dynamics as that of $T$, and has at least the topological entropy as $T$.\footnote{In fact a similar result holds for $C^0$-perturbations of $T$, see \cite{Nitecki1971}.}
Horseshoes are prevalent in dynamical systems of complex orbit structure. This is in particular the case for surface diffeomorphisms: By a celebrated result of Katok \cite{Katok1980-2}, any diffeomorphism of positive topological entropy has a hyperbolic fixed point with a transverse homoclinic point and has a horseshoe in some iterate. 
Besides these local stability features of surface maps, there exists "global" analogues in surface dynamics, by restricting to certain isotopy classes of homeomorphisms: A pseudo-Anosov homeomorphism is a factor of a subsystem of any homeomorphism isotopic to it \cite{Handel1985}. In particular it minimizes the entropy in its isotopy class. 


In this paper we exhibit stability phenomena for Hamiltonian diffeomorphisms on surfaces which have flavours of local as well as global nature. We motivate and explain here shortly some results and refer to \ref{subsec:Ham} below for definitions and precise  statements. The group of Hamiltonian diffeomorphisms carries a remarkable conjugation invariant norm $\|\cdot\|_{\hofer}$ inducing a bi-invariant metric $d_{\hofer}$, the Hofer metric. From a point of view that is concerned with $C^k$ topology, $k\geq 0$, this metric is rather flexible, in particular diffeomorphisms close with respect to $d_{\hofer}$ are in general not $C^0$ close, and at least part of the dynamics of some horseshoe might not survive under perturbation. Nevertheless, in terms of lower bounds of some dynamical quantities, remarkable stability features are displayed, even under relatively large perturbations (in terms of Hofer's metric). We exhibit such phenomena in the case of surfaces of genus $g\geq 2$ by considering a specific construction from \cite{polterovich2016autonomous} of the so-called eggbeater maps or eggbeaters, and its generalization by the first author \cite{ArnonThesis}.    
Eggbeaters are Hamiltonian variants of linked twist maps and the latter were studied by various authors, e.g by Thurston \cite{Thurston} as a family of examples of pseudo-Anosov homeomorphisms. Eggbeaters have positive topological entropy and  exponential homotopical orbit growth (cf.  \cite{Devaney1978} for the detection of horseshoes in linked-twist maps).
We are interested to what extent dynamical properties of eggbeaters persists under perturbation. For that fix some $0< \delta< 1$. We say that some dynamical property of $\phi$  $\delta$-persists if it holds for all $\psi$ with $d_{\hofer}(\psi, \phi) < \delta\|\phi\|_{\hofer}$. 
A main result in this article is, cf. Theorem \ref{thm:stable},  that on closed surfaces of genus $g\geq 2$ and for some fixed  $\delta>0$, any given lower bound on the topological entropy ("$h_{\topo} \geq E$")  $\delta$-persists for an unbounded family of eggbeater diffeomorphisms.  An analogous result holds for the exponential homotopical orbit growth. 
 Moreover we exhibit minimality of $h_{\topo}$ up to a fixed finite error on a family of eggbeaters, cf. Corollary \ref{cor1}, which is reminiscent of the minimality property of $h_{\topo}$ for pseudo-Anosov homeomorphisms.

To obtain our persistence results we will apply some lower bounds on the topological entropy for surface homeomorphisms in the presence of orbit types with specific properties, cf. Theorems \ref{intro:thm1}, \ref{intro:thm2} and \ref{HT}, and the first part of the paper is devoted to their proofs. 
We will now motivate and state these results as well as define relevant dynamical notions, and then, below in section \ref{subsec:Ham}, turn to the persistence results mentioned above.

 
\subsection{Free homotopy classes of periodic orbits and forcing}\label{subsec:forcing}

Let $M$ be a closed oriented surface and $f:M \to M$ a homeomorphism isotopic to the identity.  The topological entropy $h_{\topo}(f)$ is one of the most widely used measures for the complexity of the dynamics of $f$ and can be defined as follows:
for a fixed metric $d$ on $M$ denote by $N_{\topo}(f,n,\epsilon)$ the maximal cardinality of a set $X$ in $M$ such that
 $\sup_{0\leq k\leq n} d(f^k(x), f^k(y)) \geq \epsilon$ for all $x,y \in X$ with $x\neq y$. 
The \textit{topological entropy} of $f$ is then
\[
	h_{\topo}(f) := \lim_{\epsilon \to 0} \limsup_{n\to \infty} \frac{\log(N_{\topo}(f,n,\epsilon))}{n},
\]
which is well-defined and  independent of the metric $d$. 
Related measures for dynamical complexity is the growth of periodic points: If $N_{\per}(f,n)$ denotes the number of periodic points $x$ of $f$ of period $\leq n$, then the \text{(exponential) growth rate of periodic points} is defined as  $\Per^{\infty}(f) :=  \limsup_{n \to \infty} \frac{\log (N_{\per}(f,n))}{n}$. 

Although it is in general very hard to calculate or estimate the topological entropy, remarkably,  the existence of certain types of  periodic points of $f$ imply lower bounds on $h_{\topo}(f)$. The following result in dimension one is a typical example of this forcing phenomenon: If a continuous mapping $g$ on $[0,1]$ admits a periodic point of period $n=2^dm$ with $m>1$ odd, then $h_{\topo}(g)>\frac{1}{n}\log(2)$, see \cite{BowenFranks1976}. 
Passing to dimension two, in order to deduce some interesting dynamical information of  $f$, the period alone is not sufficient and one has to take into account further characteristics of a periodic point. Naturally one requires that the desired bounds are invariant under any isotopy of $f$ relative to the periodic orbit, which amounts to consider the braiding information, or braid type, of the periodic orbit in the suspension flow. Thurston-Nielsen theory provides a natural framework for understanding the dynamics that is forced by a given braid type, namely  dynamically minimal maps for a braid type are given by the Thurston-Nielsen canonical form. There are algorithmic methods to obtain the latter, although it becomes sometimes difficult to apply them practically, especially when passing to families of braid types. Entropy bounds for specific (families of) braid types are investigated by many researchers, see e.g. \cite{Hall1994, Song2005, HironakaKin2006} and references therein. The complete braiding information of an orbit, in particular if one considers orbits of higher complexity, is somewhat hard to overlook, and one way to simplify is to project braids back to the surface, in other words, to look for braid type invariants that can be purely presented in terms of the homotopy class of the curve that the orbit traces on $M$ through  an identity isotopy of $f$. It is those kind of invariants that we will consider in this article. 
Especially for dealing with bounds related to those kind of invariants the recently developed forcing theory of Le Calvez and Tal \cite{LeCalvezTal1} building upon Le Calvez' theory of transverse foliations \cite{LeCalvez2015} is perfectly suited. Lower bounds for $h_{\topo}(f)$ and $\Per^{\infty}(f)$ of that kind are contained in \cite{LeCalvezTal1, LeCalvezTal2, Silva}, as applications of their methods, and in \cite{Dowdall2011}, using Thurston-Nielsen theory. Our results improve and extend part of these bounds, and we work closely within the theory in \cite{LeCalvezTal1}.



 Let us introduce some essential notions. 
Let $I= (I_t)_{t\in[0,1]}$ be an identity isotopy for $f$, i.e. a continuous path from the identity to $f$, and consider the set $\dom(I) = M\setminus \fix(I)$, where  $\fix(I)=\{x \in M \, |\, I_t(x) \equiv x \text{ for all } t\in[0,1]\} \subset \fix(f)$ are the points that are fixed throughout the isotopy $I$. We denote by $\widehat{\pi}(X)$ the set of free homotopy classes of loops in a space $X$, and for any loop $\Gamma:S^{1} \to X$, by $[\Gamma] = [\Gamma]_{\widehat{\pi}(X)}$ its free homotopy class in $\widehat{\pi}(X)$. We say that $\alpha \in \widehat{\pi}(\dom(I))$   is \textit{primitive} if there is no representative of $\alpha$ that multiply covers another loop. For $m\in \N$ denote by $m\alpha$ the free homotopy class that is represented by the $m$-fold iteration of a loop that respresents $\alpha$.
For any periodic point $x\in\dom(I)$ of $f$, say of period $q \in \N$, consider the loop $I^q(x)$ given by the concatenation of the paths  $I_t(x)_{t\in [0,1]}$, $I_t(f(x))_{t\in [0,1]}, \cdots$, and $I_t(f^{q-1}(x))_{t\in [0,1]}$. The loop $I^q(x)$ defines a free homotopy class $[I^q(x)]= \alpha\in \widehat{\pi}(\dom(I))$ of loops in $\dom(I)$. An identity isotopy $I$ is \textit{maximal}, if for any fixed point $y\in \dom(I)$ of $f$,  the loop $I(y)= I^1(y)$ is not contractible in $\dom(I)$, cf. section \ref{subsec:identityisotopy}.
Let $\Gamma: S^1 \to \dom(I)$ be a loop and denote by $\mathcal{S}(\Gamma):= \{y \in \dom(I)\, | \, y = \Gamma(t) = \Gamma(t'), t\neq t'\}$ the set of self-intersections of $\Gamma$, and by $\selfi(\Gamma) = \#\mathcal{S}$  its cardinality. The \textit{geometric self-intersection number of $\alpha$} is defined as $\selfi_{\dom(I)}(\alpha) := \min \selfi(\Gamma)$,  where the minimum is taken over all smooth loops $\Gamma$ with $[\Gamma] = \alpha$ in $\widehat{\pi}(\dom(I))$ that are in \textit{general position}, i.e. each self-intersection $y\in \mathcal{S}(\Gamma)$ is transverse and $\# \Gamma^{-1}(y) = 2$.  
   
\begin{thm}\label{intro:thm1}
Let $M$ be a closed oriented surface, $f:M \to M$ a homeomorphism isotopic to the identity, and $I$ a maximal identity isotopy for $f$. Let $\alpha \in \widehat{\pi}(\dom(I))$ be primitive with $\selfi_{\dom(I)}(\alpha) \neq 0$. If there is a $q$-periodic point $x$ of $f$ in $\dom(I)$ with $[I^q(x)] = m\alpha \in \widehat{\pi}(\dom(I))$ for some $m\in \N$,  then both $\Per^{\infty}(f)$ and $h_{\topo}(f)$ are at least equal to  
$\frac{m}{q}\max\{\frac{\log(\selfi_{\dom(I)}(\alpha) + 1)}{16}, \frac{\log(2)}{2}\}$.
\end{thm}
If $M = S^2$ and $\selfi_{\dom(I)}(\alpha) \neq 0$, a fixed positive lower bound for $h_{\topo}(f)$ and $\Per^{\infty}(f)$ was obtained in \cite[Theorem 41]{LeCalvezTal1} and improved in \cite{LeCalvezTal2} and \cite{Silva}.
Theorem \ref{intro:thm1} improves these lower bounds, the main addition here is that the bounds grow with the complexity of $\alpha$.
We note that for many choices of $M$ and $\alpha$, similar lower bounds were obtained by Dowdall in \cite{Dowdall2011} using Thurston-Nielsen theory.

Theorem \ref{intro:thm1} implies similar lower bounds that are independent of the isotopy $I$. 
If $M$ has genus $\geq 2$, then, since the space of homeomorphisms isotopic to the identity is contractible, the free homotopy class $\alpha := [I^q(x)]_{\widehat{\pi}(M)}$ in $\widehat{\pi}(M)$ for a $q$-periodic point $x \in M$ of $f$ does not depend on the choice of identity isotopy $I$, and we may say that $x$ is a \textit{$q$-periodic point of class $\alpha$}.\footnote{This partitions $q$-periodic points into  equivalence classes which coincide with $f^q$-Nielsen classes of those points, see e.g. \cite{Jiang1996}.}
Theorem \ref{intro:thm1} and a result in \cite{Beguin2016} about existence of maximal isotopies, cf. section \ref{subsec:identityisotopy}, yields the following, see section \ref{sec:proof_entropy}:

\begin{thm}\label{intro:thm2}
Let $M$ be a closed oriented surface of genus $g\geq2$ and $\alpha$ a primitive class in $\widehat{\pi}(M)$ with $\selfi_M(\alpha) \neq 0$. Let $f:M \to M$ be a homeomorphism that is isotopic to the identity. If $f$ has a $q$-periodic point $x$ of class $\alpha$,  then $\Per^{\infty}(f)$ and $h_{\topo}(f)$ are at least equal to  $\frac{1}{q}\max\{\frac{\log(\selfi_M(\alpha) + 1)}{16}, \frac{\log(2)}{2}\}$. 
\end{thm}

In other words, a periodic point $x$ of class $\alpha$ that admits geometric self-intersections forces lower bounds for topological entropy and exponential orbit growth, and the more self-intersections the higher the bounds. 

Theorem \ref{intro:thm2} raises the following question: What can be said about the classes in $\widehat{\pi}(M)$ of those periodic points that are forced by the periodic point $x$ of class $\alpha$?
We address this question in connection with growth and consider the \textit{exponential homotopical orbit growth rate} given by  $$ \HP^{\infty}(f):=\limsup_{n \to \infty} \frac{\log(N_{\hp}(f,n))}{n}, $$
where $N_{\hp}(f,n)$ is the number of distinct classes of periodic points of period less then $n$. Does $\HP^{\infty}(f)$ satisfy similar bounds as $\Per^{\infty}(f)$? 
While the proof of Theorem \ref{intro:thm2} doesn't give a positive lower bound on $\HP^{\infty}(f)$ in terms of $\selfi_M(\alpha)$, we can bound $\HP^{\infty}(f)$ in terms of another invariant of $\alpha$. We adapt a construction of Turaev \cite{Turaev1978, Turaev1991} and then use it to define a growth rate $\Tu^{\infty}(\alpha)$. We explain this construction shortly here, see section \ref{sec:Turaev} for details. 

A loop $\Gamma:[0,1] \to M$ in general position representing a nontrivial class $\alpha \in \widehat{\pi}(M)$ splits at each self-intersection point $y \in \mathcal{S}(\Gamma)$ into two (oriented) closed loops $u^y_1$ and $u^y_2$ based at $y$, representing two classes $\alpha^y_1, \alpha^y_2 \in \widehat{\pi}(M)$, where we choose the labeling such that the initial tangent points of $u^y_1$ and $u^y_2$ define the orientation of $\Sigma$. Denoting by $\mathcal{Y}$ the intersection points with nontrivial $\alpha^y_1$ and $\alpha^y_2$,
\[
	v(\alpha) = \sum_{y \in \mathcal{Y}} \alpha^y_1 \otimes \alpha^y_2 - \alpha^y_2 \otimes \alpha^y_1 \in \Z[\widehat{\pi}(M)^*] \otimes \Z[\widehat{\pi}(M)^*]
\]
defines (by linear extension) \textit{Turaev's cobracket} on the free $\Z$-module over the set $\widehat{\pi}(M)^*$ of non-trivial free homotopy classes of loops in $M$. 
We consider a variant of this construction, and assume additionally that $ x_0 = \Gamma(0) = \Gamma(1)$ is not an intersection point of $\Gamma$. For each self-intersection point $y= \Gamma(t) = \Gamma(t'), t< t'$, by composing with the paths $\Gamma_{[0,t]}$ and its inverse $\overline{\Gamma_{[0,t]}}$ we obtain loops $\Gamma_{[0,t]}u ^y_i \overline{\Gamma_{[0,t]}}$, $i=1,2$, based at $x_0 = \Gamma(0) = \Gamma(1)$. These loops define elements $a^y_i \in \pi_1(M,x_0)$, $i=1,2$, see figures \ref{fig:a1} and \ref{fig:a2}.

Let $g \in \pi_1(M,x_0)$ be the element represented by $\Gamma$. Denote by $[\pi_1(M,x_0)^*]_g$ the set of $g$-equivalence classes of the nontrivial elements in $\pi_1(M,x_0)$, where we say that two elements are \textit{$g$-equivalent} if one is a conjugation of the other by a multiple of $g$, and denote the equivalence class with $[\cdot]_g$. The element
\[
	\mu(g) = \sum_{y \in \mathcal{Y}} [a^y_1 ]_g \otimes [ a^y_2]_g - [a^y_2]_g \otimes [a^y_1]_g\in \Z[\pi_1(M,x_0)^*]_g \otimes \Z[\pi_1(M,x_0)^*]_g
\]
is independent of the choice of the loop $\Gamma$ based at $x_0$ that represents $g$,  moreover $\mu(g)$ behaves well with respect to conjugation, see Lemma \ref{lem:mu_welldefined}. One actually obtains via $\mu$ an invariant of $\alpha \in \widehat{\pi}(M)$ that is a (strict) refinement of the invariant $v$.
We continue to define a growth rate as follows. First, for a subset $S=\{s_1, \cdots, s_m\} \subset \pi_1(M,x_0)$, let $\widehat{N}(n,S)$ be the number of distinct conjugacy classes of elements in $\pi_1(M,x_0)$ that can be written as a product of $\leq n$ factors from $S$, and define  $\Gamma(S):= \limsup_{n\to \infty} \frac{\log(\widehat{N}(n,S))}{n}$. Given $g\in \pi_1(M,x_0)$ and a set $\mathfrak{S}$ of $g$-equivalence classes of elements in $\pi_1(M,x_0)$, we moreover define 
$\Gamma(\mathfrak{S},g) = \inf \, \Gamma(S)$, where the infimum is taken over all sets $S\subset \pi_1(M,x_0)$ such that each element represents exactly one $g$-equvivalence class in $\mathfrak{S}$.
Now, writing $\mu(g) = \sum_{\mathfrak{a}, \mathfrak{b} \in \pi_1(M,x_0)^*_g} k_{\mathfrak{a}, \mathfrak{b}}\left( \mathfrak{a} \otimes \mathfrak{b}\right)$, 
 with $k_{\mathfrak{a},\mathfrak{b}} \in \Z$, we let  $\Comp(\mu(g))$ be the collection of terms $\mathfrak{a} \otimes \mathfrak{b}$ with $k_{\mathfrak{a},\mathfrak{b}} >0$.
For each $\mathcal{T} \subset \Comp(\mu(g))$, let $\mathcal{T}_+ = \{\mathfrak{a} \, | \, \mathfrak{a} \otimes \mathfrak{b} \in \mathcal{T}\}$ and   $\mathcal{T}_- = \{\mathfrak{b} \, | \, \mathfrak{a} \otimes \mathfrak{b} \in \mathcal{T}\}$, and define 
\[
\Gamma_{\mathcal{T}}^{g} := \min_{\pm} \min_{\mathfrak{S}}\left\{\Gamma(\mathfrak{S} \cup [g]_g, g)\, \middle| \, \mathfrak{S} \subset \mathcal{T}_{\pm}, \, \#\mathfrak{S}= \left\lceil \frac{1}{2}\#\mathcal{T}\right\rceil \right\}. \]
One shows that $T^{\infty} = \max \{\Gamma_{\mathcal{T}}^{g}\, | \, \mathcal{T} \subset \Comp(\mu(g))\}$ is actually invariant under conjugation of $g$, hence defines a growth rate $T^{\infty}(\alpha) \in [0,+\infty)$ associated to each free homotopy class $\alpha$ of loops in $M$.  
\color{black}

We obtain the following result. 

  
\begin{thm}\label{HT}
Let $M$ be a closed oriented surface of genus $g\geq 2$ and $\alpha \in \widehat{\pi}(M)$ be a primitive class.  Let $f:M\to M$ be a homeomorphism isotopic to the identity. If $f$ has a $q$-periodic point $x$ of class $\alpha$, then $\HP^{\infty}(f) \geq \frac{1}{q}\Tu^{\infty}(\alpha)$. 
\end{thm}
By a version of Ivanov's inequality it holds that $h_{\topo}(f) \geq \HP^{\infty}(f)$ \cite{Ivanov1982, Jiang1996}, see also \cite{Alves-Cylindrical}. Hence Theorem $\ref{HT}$ provides another lower bound on $h_{\topo}(f)$ in terms of the complexity of $\alpha$. 
In section \ref{sec:self_intersections} we exhibit an infinite family of classes in $\widehat{\pi}(M)$ such that for all of those classes $\alpha$,  $\Tu^{\infty}(\alpha) \geq \log({\selfi_M(\alpha)/4})$, see Lemmata \ref{cl:self_intersections_surface} and \ref{lem:T_comp}. In particular, the bound in Theorem \ref{HT} turns out to be sometimes significantly better than that in Theorem \ref{intro:thm2}.  
 
 We outline the arguments for the proof of Theorem \ref{HT}, and refer to section \ref{sec:Turaev} for details. Assume there is a $q$-periodic point $x$ of $f$ that is in class $\alpha$ with $T^{\infty}(\alpha)>0$. By the results in \cite{Beguin2016} one can choose a maximal identity isotopy $I$ for $f$. Le Calvez' theory of transverse foliation  yields a singular foliation $\mathcal{F}$ on $M$ that is transverse to $I$, see section \ref{sec:foliation}. This means in particular that there is loop $\Gamma$ transverse to $\mathcal{F}$ which is freely homotopic to $I^q(x)$ in $\dom (I)$, and we say that $\Gamma$ is associated to $x$.  Le Calvez and Tal's results provide methods to manipulate $\Gamma$ in order to obtain other transverse loops associated to periodic points of $f$. Roughly speaking, in good situations such loops might be created by starting at $\Gamma(0)$, following $\Gamma$ positively transverse to $\mathcal{F}$, stopping at some intersection point $y$ which is reached for the first time, and finally continuing along $\Gamma$ in positive direction after a turn at $y$. The turn is to the left or to the right, depending on which direction is positively transverse to the foliation. In other words, one creates a shortcut of $\Gamma$ at $y$. Moreover, one might iterate this procedure and create shortcuts at several self-intersection points of $\Gamma$.  
It is a bit subtle, especially for the iterated shortcuts, when the created loops are associated to periodic points, and one crucial assumption is that the intersections where the shortcuts are created are transverse with respect to the foliation, see section \ref{sec:foliation}.  
If $g$ denotes the element represented by $\Gamma$ in $\pi_1(M, \Gamma(0))$,  each element in $\Comp(\mu(g))$ gives rise to a self-intersection point $y$ that is actually $\mathcal{F}$-transverse. Moreover one can create several shortcuts to an iterate of $\Gamma$. We show that for (sufficiently large) $n>0$  the loops $\Gamma'= \gamma_{\rho_1}\gamma_{\rho_2} \cdots \gamma_{\rho_n}$ are associated to periodic points of period $nq$, where $\Gamma'$ is a concatenation of $n$ loops $\gamma_{\rho_i}$, each is either equal to $\Gamma$ or obtained from $\Gamma$ by a shortcut at some of the above  intersection points $y$  with one additional assumption: the turns at the chosen intersection points have to be either all to the left or all to the right.
This observations allows us to conclude the inequality $H^{\infty}(f)\geq \frac{1}{q}T^{\infty}(\alpha)$.



\subsection{Hamiltonian diffeomorphisms and persistence of topological entropy}\label{subsec:Ham}
We now turn to applications of Theorem \ref{intro:thm1} in the context of Hamiltonian diffeomorphisms.


Let us consider a closed symplectic manifold $(M, \omega)$ and denote by $\Ham(M, \omega)$ the group of Hamiltonian diffeomorphisms on $M$, that is the group of those diffeomorphisms that are the time-$1$-map of the (time-dependent) flow generated by the Hamiltonian vector field $X_{H_t}$ of some $H:S^1 \times M \to \R$, cf. section \ref{sec:floer_hom}. 
The group $\Ham(M, \omega)$ carries a distinctive bi-invariant metric $d_{\hofer}$, the Hofer metric, which plays a central role in the study of rigidity phenomena of Hamiltonian diffeomorphisms. $d_{\hofer}(\varphi,\psi)$ for any $\varphi,\psi \in \Ham(M, \omega)$ may be defined to be $\|\varphi \circ \psi^{-1}\|_{\hofer}$, where
\begin{equation*}
\|\theta\|_{\hofer} := \inf_{H} \int_{0}^{1}(\max_M H_t - \min_M H_t) \, dt,
\end{equation*}
with the infimum taken over all Hamiltonian functions  $H : [0,1] \times M \to \R$ whose associated Hamiltonian flow has $\theta$ as time-$1$-map. It follows directly from the definition that small perturbations in the sense of Hofer's metric are not necessarily small in the $C^0$ metric.\footnote{More subtle is the fact that also the converse fails in general, see e.g. \cite{EntovPolterovichPy}.}   			
The geometry of $(\Ham(M, \omega), d_{\hofer})$ and its interplay with dynamics has been thoroughly studied since its discovery by Hofer and his work in the early 90's, see \cite{PolterovichHam} for an extensive account with many results and references.

Hofer's metric plays an important role for stability features of Floer homology. 
Let us explain this shortly in our context, more details can be found in section \ref{sec:floer_hom}. Given that $(M, \omega)$ is symplectically aspherical and atoroidal, then for a free homotopy class $\alpha$ the action functional of a Hamiltonian $H$ on the smooth representatives $\mathcal{L}^{\alpha}(M)$ of $\alpha$ is given by $\mathcal{A}_H(x) = \int_0^1 H(t,x(t)) dt - \int_{\bar{x}} \omega$, where $\bar{x}:S^1 \times [0,1] \to M$ is a smooth map from an annulus to $M$, where one boundary component $S^1\times \{1\}$ maps to $x$ and the other $S^1 \times \{0\}$ to a fixed representative of $\alpha$. In case $H$ is nondegenerate one defines the Floer homology as the homology of a chain complex generated by $1$-periodic orbits for $X_{H_t}$, and chain maps given by counting $0$-dimensional moduli spaces of solutions $u:\R \times S^1 \to M$ of a certain perturbed non-linear Cauchy-Riemann equation $\bar{\partial}_{H,J}(u)=0$ whose asymptotes $\lim_{s\to \pm\infty} u(s,t) = x_{\pm}(t)$ are $1$-periodic orbits $x_{\pm}(t)$ for $X_{H_t}$. 
Since action decreases along such solutions, the Floer homology can be filtered by action, taking into account only periodic orbits of action $< a$, for varying $a \in \R$.  The full Floer homology for a non-trivial free homotopy class $\alpha$ vanishes, but one can take the filtration into account in order to understand structure and existence of orbits in $\alpha$. An elegant and fruitful way to keep track of the filtered Floer homology is to use the theory of persistent modules and barcodes \cite{polterovich2019topological}. With this terminology any Hamiltonian gives rise to a barcode, i.e. a multiset of intervals in $\R$. The barcode, in this setting, only depends on the Hamiltonian diffeomorphism that is  the time-$1$-map of $X_{H_t}$. Action estimates on continuation maps in Floer homology lead to stability properties of barcodes with respect to Hofer's metric, and hence in particular to persistence of certain fixed points.

A special family $\mathcal{E}_g$ of Hamiltonian diffeomorphisms on surfaces $M=\Sigma_g$ of genus $g$, called eggbeaters,  were introduced by Polterovich and Shelukhin in \cite{polterovich2016autonomous}. The construction was carried out for surfaces of genus $g\geq 4$ and later extended by the first author in \cite{ArnonThesis} to surfaces of genus $g\geq 2$. Computations of certain barcode-invariants were carried out and it was proved that $\mathcal{E}_g$ provide a class of Hamiltonian diffeomorphisms whose Hofer-distance to the space of autonomous Hamiltonian diffemorphisms can be arbitrary large. The same results pass to some products $\Sigma_g\times N$ \cite{polterovich2016autonomous,zhang2019}, and the constructions allow to embed the free group of two generators into the asymptotic cone of the group of Hamiltonian diffeomorphisms equipped with Hofer's metric \cite{10authors,ArnonThesis}.

   One may formulate the above results as (a consequence of) persistence of certain dynamical properties of eggbeaters. Eggbeaters are prototypical for a chaotic dynamical system. While this chaotic behaviour, as opposed to integrable behaviour, served as motivation to investigate these maps in  \cite{polterovich2016autonomous}, the question of persistence of chaos, that is persistence of entropy, exponential orbit complexity etc., has not been addressed yet. The following result gives some answers in the setting of surfaces of genus $g\geq 2$:

\begin{thm}\label{thm:stable}
Let $(\Sigma_g,\sigma_g)$ be a surface of genus $g\geq 2$ with an area form $\sigma_g$. There is a sequence $\phi_l \in \mathcal{E}_g$ of eggbeaters on $\Sigma_g$ with $M_l := \|\phi_l\|_{\hofer} \to \infty$ and constants $\delta, c>0$, such that for all $l \in \N$ and all $\psi\in \Ham(M,\omega)$ with $d_{\hofer}(\psi, \phi_l) < \delta M_l$, 
 \begin{equation*}
      \HP^{\infty}(\psi) \geq \log(C M_l^2),
     \end{equation*}
and in particular 
 \begin{equation}\label{bound_log}
 h_{\topo}(\psi) \geq \log(C M_l^2).
  \end{equation}
\end{thm} 

In particular, using the terminology at the beginning of the introduction, for any constant $E$ there is an unbounded family  in $(\Ham(\Sigma_g,\sigma_g), d_{\hofer})$ on which ("$h_{\topo}\geq E$") $\delta$-persists. 


The sequence $\phi_l$ satisfies $h_{\topo}(\phi_l) \leq \log(C'M_l^2)$ for some constant $C' > C$, see remark \ref{rem:entropy_eggb}, and hence the lower bounds of \eqref{bound_log} are "optimal" from an asymptotic viewpoint, or, in other words, eggbeaters are "almost" minimal points for the topological entropy on $\Ham(\Sigma_g,\sigma_g)$. More precisely we obtain 
\begin{cor}\label{cor1}
Let $(\Sigma_g,\sigma_g)$ be as above. Then there is  constants $K>0$ and $\delta>0$, and a sequence of $\phi_l \in \mathcal{E}_g$ with  $M_l := \|\phi_l\|_{\hofer} \to \infty$ and $h_{\topo}(\phi_l) \to \infty$, such that all $\psi \in \Ham(\Sigma_g, \sigma_g)$ with $d(\psi, \phi_l) <  \delta M_l$ satisfy  
$$h_{\topo}(\psi) \geq h_{\topo}(\phi_l) - K.$$
\end{cor}

The sequence $\phi_l$ defines an element in the asymptotic cone of $(\Ham(\Sigma_g, \sigma_g), d_{\hofer})$. An interesting question is whether similar minimality properties hold also for defining sequences of most elements in  the image of the embedding of $F_2$ into the asymptotic cone of $(\Ham(\Sigma_g, \sigma_g), d_{\hofer})$ obtained in \cite{10authors}, and it motivates efforts to further strengthen the bounds obtained in  Theorem \ref{HT}. 



Finally, we will also observe that the stability of lower bounds obtained in Theorem \ref{thm:stable} are Hofer-generic and we show

\begin{thm}\label{thm:generic}
Let $(\Sigma_g,\sigma_g)$ be a surface of genus $g\geq 2$, and let $M\geq 0$. There is an open and dense set $U \subset \Ham(\Sigma_g,\sigma_g)$ with respect to the topology induced by Hofer's metric $d_{\hofer}$ such that $h_{\topo}(\psi) \geq  M$ for all $\psi \in U$.  
\end{thm}

The proofs of Theorems \ref{thm:stable} and Theorems \ref{thm:generic} are given in Section~\ref{sec:egg_beaters}.

A recent result in \cite{Khanevsky2019}, building on different methods, cf. \cite{Brandenbursky}, asserts, that for surfaces $\Sigma_g$ of genus $g\geq 1$ there are for any $C\geq 0$ a class of Lagrangian pairs $(L,L')$, $L,L'\subset \Sigma_g$, such that $h_{\topo}(\varphi) >  C$ for all $\varphi$ with $\varphi(L) = L'$. It would be interesting to understand whether those sets of pairs may satisfy some stability properties similar to those of Theorems \ref{thm:stable} and \ref{thm:generic}.
 
Very recently the authors in \cite{Ginzburg2021} show that the topological entropy of a Hamiltonian diffeomorphism $\varphi$ on a closed surface coincides with its barcode entropy $\hbar(\varphi)$ which they introduce, and which measures the growth of the number of certain bars in the barcode of the iterates of $\varphi$.  
Hence, together with the results in this paper, this shows that the results obtained in Theorems \ref{thm:stable}, \ref{thm:generic} and Corollary \ref{cor1} hold additionally for  $\hbar$.  This is noteworthy, since it is a priori not clear which stability properties hold for $\hbar$ with respect to Hofer's metric. 

Refined stability properties with respect to $d_{\hofer}$ (for small perturbations) on the isotopy classes of braids that periodic orbits of Hamiltonian surface diffeomorphism induce in the suspension, are studied in a current project of M. Alves and the second author \cite{AlvesMeiwesBraids}. One dynamical consequence of their study is that $h_{\topo}$ is lower semi-continuous on $(\Ham(\Sigma, \omega), d_{\hofer})$  for closed surfaces $\Sigma$. 
 
Related questions of global robustness of positive entropy for families of contactomorphisms on contact manifolds were studied extensively and fruitfully in recent years by various methods.  A large class of contactomorphisms are those that arise via Reeb flows and there is an abundance of contact manifolds for which the topological entropy or the exponential orbit growth rate is positive for all Reeb flows. Examples and dynamical properties of those manifolds are investigated in 
\cite{AlvesSchlenkAbbon, Alves-Cylindrical,  Alves-Anosov, Alves-Legendrian, AlvesColinHonda, AlvesMeiwes, Cote2020, FrauenfelderSchlenk2006, MacariniSchlenk2011, Meiwesthesis}. Some of these results generalize to positive contactomorphisms \cite{Dahinden2018}, and results on the dependence of some lower bounds on topological entropy with respect to their positive contact Hamiltonians has been obtained in \cite{Dahinden2020}. Forcing results for Reeb flows are obtained in \cite{AlvesPirnapasov}. A related discussion and results on questions of $C^0$-stability of the topological entropy of geodesic flows can be found in \cite{LLMM}.  
While the approach in our paper is suited to dimension $2$, different methods yield higher dimensional symplectic manifolds for which conclusions similar to that of Theorems \ref{thm:stable} and \ref{thm:generic} hold. This will be discussed elsewhere.

\subsection{Structure of the paper}

In Section~\ref{sec:foliation} we recall the theory of transverse foliations on surfaces, which is the setting in which Theorems~\ref{intro:thm1}, \ref{intro:thm2} and \ref{HT} is proved. In Section~\ref{sec:proof_foliation} we restrict our attention to loops with so-called $\mathcal{F}$-transverse self-intersections, and prove a few key claims.

In Section~\ref{sec:proof_entropy} we prove Theorems~\ref{intro:thm1} and \ref{intro:thm2}, using the tools developed in the previous sections. In Section~\ref{sec:Turaev} we define the growth rate $\Tu^{\infty}$ of a free homotopy class and prove Theorem~\ref{HT}.

Section~\ref{sec:floer} gives a short background on Floer theory and persistence modules, which is then used in Section~\ref{sec:egg_beaters} to derive bounds on entropy and periodic orbit growth of large perturbations with respect to Hofer's metric of eggbeater maps. This will prove in particular Theorems~\ref{thm:stable} and~\ref{thm:generic}. 

\subsection{Acknowledgement}
Most of the work was completed while Matthias Meiwes was a postdoctoral researcher at the School of Mathematical Sciences of Tel Aviv University. He would like to thank this institute and especially Lev Buhovski, Yaron Ostrover, and Leonid Polterovich for their hospitality. He was also partially supported by the Chair for Geometry and Analysis of the RWTH Aachen. He would like to thank Umberto Hryniewicz for his support.

We would like to especially thank Leonid Polterovich for his important suggestions and for fruitful discussions. We would also like to thank Marcelo Alves, Umberto Hryniewicz and Felix Schlenk for their helpful suggestions and discussions.   
  
\section{Transverse foliations and transverse intersections}\label{sec:foliation}

In this section we will give the definitions and results from the theory of transverse foliations that are relevant for this paper.  We follow mainly \cite{LeCalvezTal1}. 

\subsection{Surface foliations and transverse paths}
In the following let $M$ be an oriented surface. The plane $\R^2$ will be endowed with the usual orientation. 
A \textit{path} on $M$ is a continuous map $\gamma:J \to M$, defined on an interval $J\subset \R$. A path $\gamma$ is \textit{proper} if $J$ is open and the preimage of every compact subset of $M$ is compact.
 A \textit{line} is an injective and proper path $\lambda: J \to M$, it inherits a natural orientation induced by the usual orientation of $\R$. If $M=\R^2$, the complement of $\lambda$ has two connected components, one on the right, $R(\lambda)$, and one on the left, $L(\lambda)$. 
A \textit{loop} is a continuous map $\Gamma:S^1 = \R/ \Z \to M$. It lifts to a path $\gamma: \R \to M$ with $\gamma(t+1) = \gamma(t)$ for all $t\in \R$, the \textit{natural lift} of $\Gamma$. For two closed finite intervals $J,J'$ and paths $\gamma:J \to M$, $\gamma':J \to M$ we denote by $\gamma \gamma'$ the usual concatenation of paths. In particular $\Gamma^m$ for $m\in \N$ is the $m$-fold iteration of a loop $\Gamma$, i.e. $\Gamma^m(t) := \gamma(mt)$. The path obtained by reverse parametrization of $\gamma$ is denoted by $\overline{\gamma}$.

A \textit{singular oriented foliation} on  $M$ is an oriented topological foliation $\mathcal{F}$ defined on an open set of $M$. This set is called the domain of $\mathcal{F}$ and is denoted by $\dom (\mathcal{F})$. The complement $M \setminus \dom (\mathcal{F})$ is the \textit{singular set}, denoted by $\sing (\mathcal{F})$. 
We denote by $\phi_z$ the leaf passing through $z\in \dom (\mathcal{F})$, $\phi_z^+$ the positive, and by $\phi_z^-$ the negative half-leaf. 
A path $\gamma : J \to M$ is \textit{(positively) transverse to $\mathcal{F}$} or \textit{$\mathcal{F}$-transverse} if its image does not meet the singular set and if, for every $t_0 \in J$, there is a continuous chart $h: W \to (0,1)^2$ at $\gamma(t_0)$ compatible with the orientation and sending the foliation $\mathcal{F}|_{W}$ onto the vertical foliation oriented downwards such that the map $\pi_1 \circ h \circ \gamma$ is increasing in a neighborhood of $t_0$, where $\pi_1$ is the vertical projection. 
Let $\widetilde{\dom(\mathcal{F})}$ be the universal covering space of $\dom(\mathcal{F})$. $\mathcal{F}|_{\dom (\mathcal{F})}$ lifts to a (non-singular) foliation $\widetilde{\mathcal{F}}$ on $\widetilde{\dom(\mathcal{F})}$. 
Note that since there is no non-singular foliation on $S^2$, $\widetilde{\dom(\mathcal{F})}$ is always homeomorphic to $\R^2$.
Every lift of an $\mathcal{F}$-transverse path $\gamma$ is an $\widetilde{\mathcal{F}}$-transverse path on $\widetilde{\dom(\mathcal{F})}$. 
We say that $\widetilde\gamma: \R \to \widetilde{\dom(\mathcal{F})}$ is a \textit{lift of the loop $\Gamma$ }(\textit{to $\widetilde{\dom(\mathcal{F})}$}) if it is a lift of its natural lift $\gamma: \R \to \dom(\mathcal{F})$. 

If $\mathcal{F}$ is a \textit{non-singular} foliation on $\R^2$, then two $\mathcal{F}$-transverse paths $\gamma: J \to \R^2$ and $\gamma': J'\to \R^2$ are  \textit{equivalent for $\mathcal{F}$} if there exists an increasing homeomorphism $h:J \to J'$ such that $\phi_{\gamma'(h(t))} = \phi_{\gamma(t)}$ for every $t\in \R$. 
In general, if $\mathcal{F}$ is  a (possibly \textit{singular}) foliation on an oriented surface $M$, two transverse paths $\gamma: J \to M$ and $\gamma': J'\to M$ are called \textit{equivalent for $\mathcal{F}$} if they can be lifted to the universal covering space $\widetilde{\dom{(\mathcal{F})}}$ of $\dom(\mathcal{F})$ as paths that are equivalent for the lifted foliation $\widetilde{\mathcal{F}}$. 
A loop $\Gamma: S^1 \to \dom(\mathcal{F})$ is called \textit{positively transverse} to $\mathcal{F}$ if this holds for the natural lift $\gamma: \R \to \dom(\mathcal{F})$. Two $\mathcal{F}$-transverse loops $\Gamma$ and $\Gamma'$ are \textit{equivalent} if there exists two lifts $\widetilde\gamma: \R \to \widetilde{\dom(\mathcal{F})}$ and $\widetilde\gamma': \R \to \widetilde{\dom(\mathcal{F})}$ of $\Gamma$ and $\Gamma'$, respectively, a deck transformation $T$ and an orientation preserving homeomorphism $h:\R \to \R$ invariant by $t\mapsto t+1$ and such that for all $t\in \R$, 
\begin{align*}
\widetilde{\gamma}(t+1) = T(\widetilde{\gamma}(t)), \, \widetilde\gamma'(t+1) = T(\widetilde\gamma'(t)), \, \phi_{\widetilde\gamma'(h(t))} = \phi_{\widetilde\gamma(t)}.
\end{align*}

\subsection{$\mathcal{F}$-transverse intersection}
We will now recall the definition of $\mathcal{F}$-transverse intersection, which is a central notion in \cite{LeCalvezTal1}. 

Let $\lambda_0, \lambda_1$ and $\lambda_2$ be three lines in $\R^2$.  $\lambda_2$ is \textit{above} $\lambda_1$ \textit{relative to} $\lambda_0$ (or $\lambda_1$ is \textit{below} $\lambda_2$ \textit{relative to} $\lambda_0$) if:
\begin{itemize}
\item the three lines are pairwise disjoint
\item none of the lines separates the two others
\item if $\gamma_1, \gamma_2$ are two disjoint paths that join $z_1=\lambda_0(t_1)$ resp. $z_2 = \lambda_0(t_2)$ to $z'_1\in \lambda_1$ resp.  $z'_2 \in \lambda_2$ and do not meet the three lines but at the ends, then $t_2 >t_1$. 
\end{itemize}


Assume $\mathcal{F}$ is a non-singular foliation on $\R^2$. 
Let $\gamma_1: J_1 \to \R^2$ and $\gamma_2: J_2 \to \R^2$ be two transverse paths such that $\phi_{\gamma_1(t_1)} = \phi_{\gamma_2(t_2)} = \phi$.  $\gamma_1$ and $\gamma_2$ \textit{intersect $\mathcal{F}$-transversally and positively at $\phi$} if there exist $a_1, b_1$ in $J_1$ and $a_2, b_2 \in J_2$ satisfying $a_1 < t_1 < b_1$, $a_2 < t_2 < b_2$, such that 
\begin{itemize}
\item $\phi_{\gamma_2(a_2)}$ is below $\phi_{\gamma_1(a_1)}$ relative to $\phi$, and
\item $\phi_{\gamma_2(b_2)}$ is above $\phi_{\gamma_1(b_1)}$ relative to $\phi$. 
\end{itemize}
In this situation we also say that $\gamma_1$ intersects $\gamma_2$  $\mathcal{F}$-transversally and positively, $\gamma_2$ and $\gamma_1$ intersect $\mathcal{F}$-transversally and negatively, or $\gamma_2$ intersects $\gamma_1$ $\mathcal{F}$-transversally and negatively at $\phi$. 
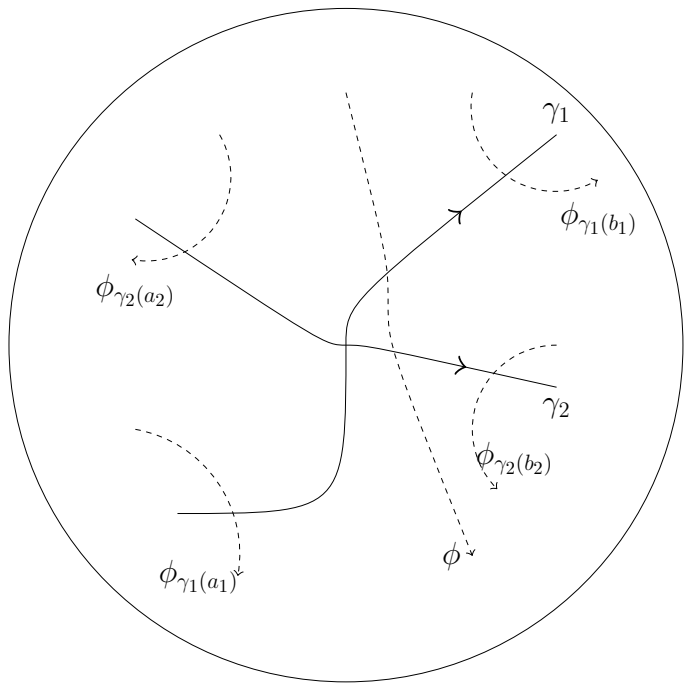
\begin{figure}
    \centering
    \scalebox{0.7}{
    \begin{tikzpicture}[scale=0.8]
        \coordinate (O) at (0,0);
        
        \draw (O) circle (8);
        
        \draw[lift] (-4,-4) .. controls (0,-4) .. (O) .. controls (0,1) .. (5,5);
        \draw[lift] (-5,3) .. controls (-0.5,0) .. (O) .. controls (0.5,0) .. (5,-1);
        
        \draw[leaf] (0,6) .. controls (1,2) .. (1,1) .. controls (1,0) .. (3,-5);
        \draw[leaf] (-3,5) arc (30:-100:2);
        \draw[leaf] (-5,-2) arc (80:-10:3);
        \draw[leaf] (3,6) arc (170:300:2);
        \draw[leaf] (5,0) arc (90:225:2);
        
        \draw (2.5,-5) node {\Large $\phi$};
        \draw (5,5.5) node {\Large $\gamma_1$};
        \draw (5,-1.5) node {\Large $\gamma_2$};
        \draw (4,-2.7) node {\Large $\phi_{\gamma_2(b_2)}$};
        \draw (6,3) node {\Large $\phi_{\gamma_1(b_1)}$};
        \draw (-5,1.3) node {\Large $\phi_{\gamma_2(a_2)}$};
        \draw (-3.5,-5.5) node {\Large $\phi_{\gamma_1(a_1)}$};
        
    \end{tikzpicture}
    }
    \caption{$\gamma_1$ and $\gamma_2$ intersect $\mathcal{F}$-transversally and positively at $\phi$.}
    \label{fig:transverse_intersection}
\end{figure}

\begin{rem}\label{rem:transitive}
One has the following transitivity property: Let $\gamma_1: J_1 \to \R^2, \gamma_2: J_2 \to \R^2$, and $\gamma_3: J_3 \to \R^2$ be transverse paths with $\phi_{\gamma_1(t_1)} = \phi_{\gamma_2(t_2)} = \phi_{\gamma_3(t_3)}= \phi$.  
If $\gamma_1$ and $\gamma_2$ intersect $\mathcal{F}$-transversally and positively at $\phi$, and $\gamma_2$ and $\gamma_3$ intersect $\mathcal{F}$-transversally and positively at $\phi$, then $\gamma_1$ and $\gamma_3$  intersect $\mathcal{F}$-transversally and positively at $\phi$. 
\end{rem}
Let now  $\mathcal{F}$ be a (possibly singular) foliation on an oriented surface $M$. Let $\gamma_1: J_1 \to M$ and $\gamma_2: J_2 \to M$ be two transverse paths such that $\phi_{\gamma_1(t_1)} = \phi_{\gamma_2(t_2)} = \phi$. 
We say that $\gamma_1$ and $\gamma_2$ \textit{intersect $\mathcal{F}$-transversally and positively at $\phi$ (resp. negatively at $\phi$)} if there exists paths $\widetilde\gamma_1: J_1 \to \widetilde{\dom(\mathcal{F})}$ and $\widetilde\gamma_2: J_2 \to \widetilde{\dom(\mathcal{F})}$, lifting $\gamma_1$ and $\gamma_2$, with a common leaf $\widetilde{\phi} = \phi_{\widetilde\gamma_1(t_1)} = \phi_{\widetilde\gamma_2(t_2)}$ that lifts $\phi$ such that $\widetilde\gamma_1$ and $\widetilde\gamma_2$ intersect $\widetilde{\mathcal{F}}$-transversally and positively at $\widetilde{\phi}$ (resp. negatively at $\widetilde{\phi}$). 
If two paths $\gamma_1$ and $\gamma_2$ intersect $\mathcal{F}$-transversally, there is $t_1'$ and $t_2'$ such that $\gamma_1(t_1') = \gamma_2(t_2')$ and such that $\gamma_1$ and $\gamma_2$ intersect $\mathcal{F}$-transversally at $\phi_{\gamma_1(t_1')} = \phi_{\gamma_2(t_2')}$. 
We say that $\gamma_1$ and $\gamma_2$ intersect $\mathcal{F}$-transversally at $\gamma_1(t_1') = \gamma_2(t_2')$.  
A transverse path $\gamma$ has a \textit{(positive) $\mathcal{F}$-transverse self-intersection at $\gamma(t_1) = \gamma(t_2)$}, $t_1 < t_2$,  if for every lift $\widetilde\gamma$ there is a deck transformation $U$ such that $\widetilde\gamma$ and $U\widetilde\gamma$ have a (positive) $\widetilde{\mathcal{F}}$-transverse intersection at $\widetilde\gamma(t_1) = U\widetilde\gamma(t_2)$. 
A transverse loop $\Gamma$ has an \textit{$\mathcal{F}$-transverse self-intersection at $\Gamma(t_1) = \Gamma(t_2)$} if its natural lift $\gamma$ has an $\mathcal{F}$-transverse self-intersection at $\gamma(t_1) = \gamma(t_2)$. 

Finally, we say that $\gamma:[a,b] \to \R^2$ (for a regular foliation $\mathcal{F}$ on $\R^2$) has a \textit{leaf on its right} resp. \textit{a leaf on its left}, if there is a leaf $\phi$ such that $\phi$ is above resp. below $\phi_{\gamma(a)}$ relative to $\phi_{\gamma(b)}$. We say that $\gamma:[a,b] \to M$ has \textit{a leaf on its right } resp. \textit{a leaf on its left}, if a lift of $\gamma$ to $\widetilde{\mathcal{F}}$ has a leaf on its right resp. a leaf on its left.

\subsection{Identity isotopies}\label{subsec:identityisotopy}

Let in the following $f$ be a homeomorphism on $M$ that is isotopic to the identity. Let $\mathcal{I}$ be the set of isotopies $I = (f_t)_{t\in[0,1]}$ between the identity and $f$. Here isotopy means a continuous path of homeomorphisms with respect to the topology defined by the uniform convergence of maps and their inverses on compact sets. For $I \in \mathcal{I}$ the  \textit{trajectory} $I(z)$ of a point $z\in M$ is defined to be the path $t \mapsto f_t(z)$.

Let $\fix(I) := \bigcap_{t\in[0,1]} \fix(f_t)$ and $\dom (I) = M \setminus \fix(I)$. There is the following preorder on $\mathcal{I}$.  
We define  $I < I'$  if 
\begin{itemize}
\item $\fix(I) \subset \fix(I')$
\item $I'$ is homotopic to $I$ relative to $\fix(I)$.
\end{itemize}
For each $I \in \mathcal{I}$, there is $I'\in \mathcal{I}$ with $I< I'$ and such that $I'$ is maximal with respect to $<$. This was proved in \cite{Jaulent2014} with certain restrictions and in \cite{Beguin2016} in full generality. See also \cite{HLS2016} for the case of diffeomorphisms. 
Maximal elements are exactly those $I \in \mathcal{I}$ such that for every $z\in \fix(f) \setminus \fix(I)$ the loop $I(z)$ is not contractible in $\dom(I)$, see \cite{Jaulent2014}.
A foliation $\mathcal{F}$ on $M$ is called \textit{transverse to I} if  
\begin{itemize}
\item the singular set $\sing(\mathcal{F})$ coincides with $\fix(I)$.
\item  for every $z\in \dom(I)$ the trajectory $I(z)$ is homotopic in $\dom(I)$ relative to the endpoints to a path $\gamma$
positively transverse to $\mathcal{F}$. 
\end{itemize}
One has the following fundamental  result. 
\begin{thm}\label{Theorem_foliation}\cite{LeCalvez2015}
For any maximal isotopy $I\in \mathcal{I}$ there exists a singular oriented foliation $\mathcal{F}$ on $M$ that is \textit{transverse to $I$}.
\end{thm}

\subsection{Admissible paths}
Let $I$ be a maximal isotopy to $f$ and $\mathcal{F}$ be transverse to $I$. 
Let $I_{\mathcal{F}}(z)$ denote the class of paths that are positively transverse to $\mathcal{F}$, that join $z$ and $f(z)$, and that are homotopic in $\dom(I)$ to $I(z)$ relative to the endpoints. Every path in this class, if the context is clear also denoted by $I_{\mathcal{F}}(z)$, is called a \textit{transverse trajectory} to $z$. One defines for every $n$, $I_{\mathcal{F}}^n(z) := \Pi_{0\leq k < n} I_{\mathcal{F}}(f^k(z))$.

A path $\gamma:[a,b] \to \dom(I)$ positively transverse to $\mathcal{F}$ is called \textit{admissible of order $n$} if it is equivalent to a path $I^n_{\mathcal{F}}(z)$ for some $z\in \dom(I)$. $\gamma$ is called \textit{admissible of order $\leq n$} if it is a subpath of a path that is admissible of order $n$. A transverse path that has a leaf on its right or a leaf on its left and that is admissible of order $\leq n$, is also admissible of order $n$ (\cite[Proposition 19]{LeCalvezTal1}). 

The concept of $\mathcal{F}$-transverse intersections allows to show admissibility for various transverse paths. 
The "fundamental proposition" in \cite{LeCalvezTal1} is 
\begin{prop}\cite[Proposition 20]{LeCalvezTal1}\label{fundprop}
Let $\gamma_1:[a_1, b_1] \to M$ and $\gamma_2:[a_2,b_2] \to M$ be transverse paths that intersect $\mathcal{F}$-transversally at $\gamma_1(t_1) = \gamma_2(t_2)$. If $\gamma_1$ is admissible of order $n_1$ and $\gamma_2$ is admissible of order $n_2$, then ${\gamma_1}|_{[a_1, t_1]}\gamma_2|_{[t_2,b_2]}$ and $\gamma_2|_{[a_2,t_2]}\gamma_1|_{[t_1,b_1]}$ are admissible of order $n_1 + n_2$. Furthermore either one of these paths is admissible of order $\min(n_1,n_2)$ or both paths are admissible of order $\max(n_1,n_2)$. 
\end{prop}

From this proposition one can deduce for example the following
\begin{prop}\label{Proposition23}(\cite[Proposition 23]{LeCalvezTal1})
Suppose that $\gamma: [a,b] \to M$ is a transverse path admissible of order $n$ and that $\gamma$ has an $\mathcal{F}$-transverse self-intersection at $\gamma(s) = \gamma(t)$ with $s< t$. Then $\gamma|_{[a, s]}\gamma|_{[t,b]}$ is admissible of order $n$ and $\gamma|_{[a,s]}(\gamma|_{[s,t]})^q\gamma|_{[t,b]}$ is admissible of order $qn$ for every $q\geq 1$. 
\end{prop}

Besides Proposition \ref{Proposition23} we will use the following, a bit more general, statement:
\begin{prop}\label{Proposition23_1}
Let $k>0$. Suppose that $\gamma: [a,b] \to M$ is a transverse path admissible of order $n$ and that $\gamma$ has positive $\mathcal{F}$-transverse self-intersections at $\gamma(s_i) = \gamma(t_i)$, $i=1, \cdots k$, with $a < s_1 < t_1 < s_2 < t_2 < \cdots<s_k < t_k < b$. 
Then $\gamma|_{[a, s_1]}\gamma|_{[t_1,s_2]}\gamma|_{[t_2,s_3]}\cdots \gamma|_{[t_k,b]}$ is admissible of order $n$.  The same holds true if all $\mathcal{F}$-transverse self-intersections above are negative. 
\end{prop}
\begin{proof}
By Proposition \ref{Proposition23}, $\gamma|_{[a, s_1]}\gamma|_{[t_1,b]}$ is admissible of order $n$. 
Furthermore it is straightforward to see that  $\gamma|_{[a, s_1]}\gamma|_{[t_1,b]}$ and $\gamma$ have a positive $\mathcal{F}$-transverse intersection at $\gamma|_{[a, s_1]}\gamma|_{[t_1,b]}(s_2) = \gamma(t_2)$. Hence, by  Proposition \ref{fundprop},  $\gamma|_{[a, s_1]}\gamma|_{[t_1,s_2]}\gamma|_{[t_2,b]}$ is admissible of order $n$  or $\gamma|_{[a,t_2]}\gamma|_{[s_2,b]}= \gamma|_{[a,s_2]}(\gamma|_{[s_2,t_2]})^2 \gamma|_{[t_2,b]}$ is admissible of order $n$. Repeating this argument inductively, that is applying Proposition \ref{fundprop} to the paths $\gamma|_{[a, s_1]}\gamma|_{[t_1,b]}$ and $\gamma|_{[a,s_2]}(\gamma|_{[s_2,t_2]})^k\gamma|_{[s_2, b]}$, we get that $\gamma|_{[a, s_1]}\gamma|_{[t_1,s_2]}\gamma|_{[t_2,b]}$ is admissible of order $n$ or \linebreak  $\gamma|_{[a,s_2]}(\gamma|_{[s_2,t_2]})^q \gamma|_{[t_2,b]}$ is admissible of order $n$ for all $q\geq 1$. That the latter is impossible is shown in \cite[Proof of Proposition 23]{LeCalvezTal1}. Hence, repeating this argument for $i=3, \cdots, k$, shows the claim. 
\end{proof}

Certain assumptions on a transverse loop $\Gamma$ guarantee the existence of periodic points of $f$. 
An $\mathcal{F}$-transverse loop $\Gamma$ is \textit{linearly admissible of order $q$} if 
there exists two sequences $(r_k)_{k\geq 0}$ and $(s_k)_{k\geq 0}$ of natural numbers with 
\begin{align*}
\begin{split}
&\lim_{k\to \infty} r_k = \lim_{k\to \infty} s_k = +\infty, \, \limsup_{k\to \infty} r_k/s_k \geq 1/q, \text{and such that }\\
&\gamma|_{[0,r_k]} \text{ is admissible of order } \leq s_k.
\end{split}
\end{align*}
If $z$ is a periodic point of period $q$, then there exists a transverse loop $\Gamma'$ whose natural lifts satisfies $\gamma'|_{[0,1]} = I_{\mathcal{F}}^q(z)$. A transverse loop $\Gamma$ is called \textit{associated} to $z$ if it is $\mathcal{F}$-equivalent to $\Gamma'$. 
Note that $\Gamma$ is then linearly admissible of order $q$.
The following important realization result asserts a partial converse. 
\begin{prop}(\cite[Proposition 26]{LeCalvezTal1})\label{Proposition26}
Let $\Gamma$ be a linearly admissible transverse loop of order $q$ that has an $\mathcal{F}$-transverse self-intersection. 
Then for every rational number $r/s \in (0,1/q]$ written in an irreducible way, $\Gamma^r$ is associated to a periodic point of period $s$. 
\end{prop}

\section{Non-simple free homotopy classes and $\mathcal{F}$-transverse self-intersections}\label{sec:proof_foliation}
Let throughout this section $M$ be an oriented closed surface,  $\mathcal{F}$ a (singular) oriented foliation on $M$.
In this section we will study $\mathcal{F}$-transverse loops in $\dom(\mathcal{F})$ that are not freely homotopic to a multiple of a simple loop and derive some useful properties. At first, we show that the existence of a $\mathcal{F}$-transverse  self-intersection is sufficient for a loop to be of this type. 
\begin{lem}\label{lem:transverse_not_simple}
If $\Gamma$ is an $\mathcal{F}$-transverse loop in $\dom(\mathcal{F})$ that has an $\mathcal{F}$-transverse self-intersection then it is not freely homotopic in $\dom(\mathcal{F})$ to a multiple of a simple loop. 
\end{lem}
\begin{proof}
Choose a Riemannian metric $g$ on $\dom(\mathcal{F})$ and let $\widetilde{g}$ be the lift to the universal cover $\widetilde{\dom(\mathcal{F})}$, $\widetilde{d}(\cdot, \cdot)$ its induced metric.
For a lift $\widetilde{\gamma}$ of $\gamma$ and for $\epsilon > 0$ denote by $\mathcal{U}(\widetilde\gamma, \epsilon):= \{x \in \widetilde{\dom(\mathcal{F})} \, |\, \exists t\in \R \text{ with } \widetilde{d}(\widetilde{\gamma}(t), x)  < \epsilon\}$ the $\epsilon$-neighbourhood of $\widetilde\gamma$. 
 Since $\Gamma$ is $\mathcal{F}$-transverse we can choose $\epsilon> 0$ sufficiently small such that, for any lift $\widetilde\gamma$ of $\gamma$, every leaf of $\widetilde{\mathcal{F}}$ that intersects $\mathcal{U}(\widetilde\gamma, \epsilon)$  also intersects $\widetilde\gamma$.
Let now $\widetilde\gamma_1$ and $\widetilde\gamma_2$ be two lifts of $\gamma$ that intersect $\mathcal{F}$-transversally and positively at $\widetilde\gamma_1(t_1) = \widetilde\gamma_2(t_2)$.
We show that 
\begin{align}\label{nonsimple}
\begin{split}
&\forall C>0, \exists a, b \in \R \text{ such that } \widetilde\gamma_1(a) \in R(\widetilde\gamma_2), \, \widetilde\gamma_1(b) \in L(\widetilde\gamma_2), \text{ and } \\ &\widetilde\gamma_1(a), \, \widetilde\gamma_1(b) \text{ do not lie in the $C$-neighbourhood } \mathcal{U}(\widetilde\gamma_2, C) \text{ of }\widetilde\gamma_2.
\end{split}
\end{align} 
If $\Lambda$ is a loop freely homotopic to $\Gamma$, then \eqref{nonsimple} still holds for the lifts $\lambda_1$ and $\lambda_2$ of $\Lambda$  that are obtained by lifting a homotopy of $\Gamma$ to $\Lambda$ to homotopies that extend $\widetilde\gamma_1$ and $\widetilde\gamma_2$, respectively. In particular, the images of $\lambda_1$ and $\lambda_2$ are non-identical and so $\Lambda$ cannot be a multiple of a simple loop. 

Let $T: \widetilde{\dom(\mathcal{F})} \to \widetilde{\dom(\mathcal{F})}$ be the deck transformation that is given by $T(\widetilde\gamma_1(t)) = \widetilde\gamma_1(t+1)$. 
There is some $k \in \N$ sufficiently large such that $\widetilde\gamma_2$ lies on the right of $\phi_{\widetilde\gamma_1(t_1+k)}= T^k(\phi_{\widetilde\gamma_1(t_1)})$ and $\widetilde\gamma_2$ lies on the left of $\phi_{\widetilde\gamma_1(t_1-k)}= T^{-k}(\phi_{\widetilde\gamma_1(t_1)})$. 
Consider the lifts $T^{2kl}\widetilde\gamma_2(t):= T^{2kl}(\widetilde\gamma_2(t))$, $l\in \Z$, of $\gamma$. 
$\widetilde\gamma_1$ and $T^{2kl}\widetilde\gamma_2$  intersect $\mathcal{F}$-transversally and positively at $\widetilde\gamma_1(t_1 + 2kl) = T^{2kl}\widetilde\gamma_2(t_2)$. 
$T^{2kl}\widetilde\gamma_2$ is on the right of $\phi_{\widetilde\gamma_1(t+2kl+k)}= T^{2kl+k}(\phi_{\widetilde\gamma_1(t)})$, and on the left of $\phi_{\widetilde\gamma_1(t+2kl-k)}= T^{2kl-k}(\phi_{\widetilde\gamma_1(t)})$. Moreover, $\phi_{\widetilde\gamma_1(t+2kl+k)}$ is on the left of $T^{2kl}\widetilde\gamma_2$ and $\phi_{\widetilde\gamma_1(t+2kl-k)}$ is on the right of $T^{2kl}\widetilde\gamma_2$.
It follows that for all $l_1, l_2 \in \Z, l_1< l_2$, $T^{2kl_2}\widetilde\gamma_2$ is on the right of $T^{2kl_1}\widetilde\gamma_2$ and $T^{2kl_1}\widetilde\gamma_2$ is on the left of $T^{2kl_2}\widetilde\gamma_2$. 
No leaf of $\widetilde{\mathcal{F}}$ intersects both $T^{2kl_1}\widetilde\gamma_2$ and $T^{2kl_2}\widetilde\gamma_2$ for $l_1 \neq l_2$, hence the sets $\mathcal{U}( T^{2kl}\widetilde\gamma_2, \epsilon)$,  $l\in \Z$, are pairwise disjoint. 
Since any path from $T^{2kl_1}\widetilde\gamma_2$ to $T^{2kl_2}\widetilde\gamma_2$,  $l_1< l_2 \in \Z$, has to cross the lifts $T^{2k(l_1+1)}\widetilde\gamma_2, T^{2k(l_1+2)}\widetilde\gamma_2, \cdots, T^{2k(l_2-1)}\widetilde\gamma_2$, the $\widetilde{d}$-distance of the images of  $T^{2kl_1}\widetilde\gamma_2$ and $T^{2kl_2}\widetilde\gamma_2$ in $\widetilde{\dom(\mathcal{F})}$ is bounded from below by $2\epsilon(l_2 - l_1-1) + 2\epsilon = 2\epsilon(l_2 - l_1)$. 
Hence, for any $C> 0$ we have for $l > \frac{C}{2\epsilon}$ that $\widetilde\gamma_1(t_1 + 2kl)\in L(\widetilde\gamma_2)$, $\widetilde\gamma_2(t_1 - 2kl) \in R(\widetilde\gamma_2)$, and both points do not lie in $\mathcal{U}(\widetilde\gamma_2, C)$, see Figure \ref{fig:Lemma3.1}. 

\end{proof}

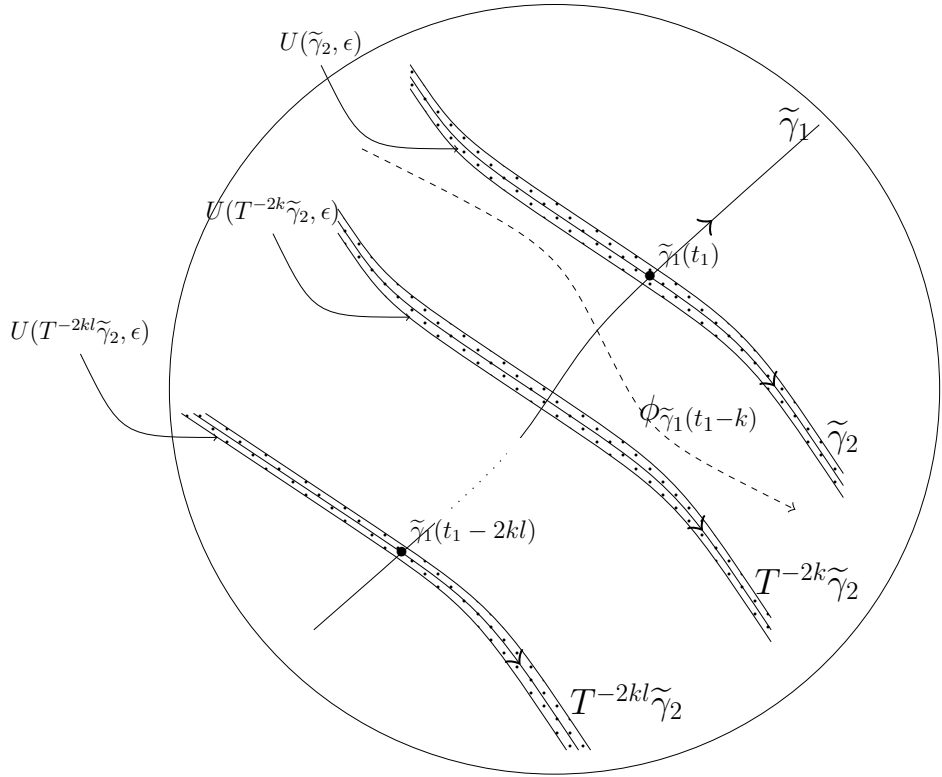
\begin{figure}
    \centering
    \scalebox{0.8}{
    \begin{tikzpicture}[scale=0.8]
        \coordinate (O) at (0,0);
        
        \draw (O) circle (8);
        
								\path[name path=l1] (-5,-5) .. controls (-1,-1.5) .. (O) .. controls (1,1.5) .. (5.5,5.5);

				\begin{scope}
				\clip (-6,-6) rectangle (-2.5,-2.5);
        \draw[use path=l1,lift];
				\end{scope}
				
				\begin{scope}
				\clip (-1,-1) rectangle (6,6);
        \draw[use path=l1,lift];
				\end{scope}
				\begin{scope}
				\clip (-2.5,-2.5) rectangle (-1,-1);
        \draw[use path=l1,loosely dotted, line width =0.2mm];
				\end{scope}
				
				\path[name path=l2] (-3,6.5) .. controls (-2,5) .. (1,3) .. controls (4,1) .. (6,-2);
				
				\path[name path=p1,  yshift=0.25cm] (-3,6.5) .. controls (-2,5) .. (1,3) .. controls (4,1) .. (6,-2);
        \path[name path=p2,yshift=-0.25cm] (-3,6.5) .. controls (-2,5) .. (1,3) .. controls (4,1) .. (6,-2);
				\draw[use path=l2, lift];
				\draw[use path=p1, -];
				\draw[use path=p2,-];
				
				\draw[lift,yshift=-3cm,xshift=-1.5cm] (-3,6.5) .. controls (-2,5) .. (1,3) .. controls (4,1) .. (6,-2);
				
				\path[name path=p3, yshift=-2.75cm, xshift=-1.5cm] (-3,6.5) .. controls (-2,5) .. (1,3) .. controls (4,1) .. (6,-2);
				\path[name path=p4,yshift=-3.25cm, xshift=-1.5cm](-3,6.5) .. controls (-2,5) .. (1,3) .. controls (4,1) .. (6,-2);
								
				\draw[use path=p3,-];
       \draw[use path=p4,-];

			\path[name path=l4,yshift=-5.5cm,xshift=-5.5cm] (-2,5) .. controls (-2,5) .. (1,3) .. controls (4,1) .. (6,-2);
			        			\path[name path=p5,yshift=-5.5cm,xshift=-5.25cm] (-2,5) .. controls (-2,5) .. (1,3) .. controls (4,1) .. (6,-2);
			\path[name path=p6,yshift=-5.5cm,xshift=-5.75cm] (-2,5) .. controls (-2,5) .. (1,3) .. controls (4,1) .. (6,-2);
			\draw[use path=l4,lift];
   \draw[use path=p5,-];
\draw[use path=p6,-];

	\draw[leaf] (-4,5) .. controls (0,3) .. (1,1) .. controls (2,-1) .. (5,-2.5);

\tikzfillbetween[
    of=p1 and p2] {pattern=mydots};
\tikzfillbetween[
    of=p3 and p4] {pattern=mydots};
				\tikzfillbetween[
    of=p5 and p6] {pattern=mydots};
				
				
				 \path[name intersections={of={l1} and {l2}, by={I12}}];
				 \path[name intersections={of={l1} and {l4}, by={I14}}];
				 \node at (I12) {\textbullet};
         \node at (I14) {\textbullet};

        \draw (3,-0.5) node {\Large $\phi_{\widetilde\gamma_1(t_1-k)}$};
        \draw (5,5.5) node {\Large $\widetilde\gamma_1$};
        \draw (6,-1) node {\Large $\widetilde\gamma_2$};
				\draw (5.25,-4) node {\Large $T^{-2k}\widetilde\gamma_2$};
				\draw (1.5,-6.5) node {\Large $T^{-2kl}\widetilde\gamma_2$};
        \node[anchor=south west] at (I12)  {$\widetilde\gamma_1(t_1)$};
        \node[anchor=south west] at (I14)  {$\widetilde\gamma_1(t_1-2kl)$};
       
        \draw[->,yshift=5cm, xshift=5cm] (170:10) .. controls (-9,0) .. (180:7) node[pos=0,above] {$U(\widetilde\gamma_2,\epsilon)$};
				        \draw[->,yshift=-1cm] (170:10) .. controls (-9,0) .. (180:7) node[pos=0,above] {$U(T^{-2kl}\widetilde\gamma_2,\epsilon)$};
  \draw[->,yshift=1.5cm, xshift=4cm] (170:10) .. controls (-9,0) .. (180:7) node[pos=0,above] {$U(T^{-2k}\widetilde\gamma_2,\epsilon)$};
        
    \end{tikzpicture}
    }
    \caption{Situation in the proof of Lemma \ref{lem:transverse_not_simple}. The points $\widetilde\gamma_1(t_1-2kl)$ are on the right of $\widetilde\gamma_2$ for all $l>0$ and are not contained in $U(\widetilde\gamma_2,C)$ if $l>C/2\epsilon$.}
    \label{fig:Lemma3.1}
\end{figure}

Let now, and throughout the section if not explicitly stated otherwise, $\Gamma : S^1 \to \dom(\mathcal{F})$ be an \textit{$\mathcal{F}$-transverse loop that is not freely homotopic in $\dom(\mathcal{F})$ to a multiple of a simple loop}, denote $\gamma: \R \to \dom(\mathcal{F})$ its natural lift. We will assume that $\dom(\mathcal{F})$ is connected, otherwise consider instead of $\dom(\mathcal{F})$ the connected component that contains $\Gamma$.  
In the remainder of the section we will prove a few properties of lifts of $\Gamma$ to $\widetilde{\dom(\mathcal{F})}$. In particular we will see that the converse of Lemma \ref{lem:transverse_not_simple} holds, i.e. that $\Gamma$ has an $\mathcal{F}$-transverse self-intersection. 

Since there is a primitive non-simple free homotopy class of loops in $\dom(\mathcal{F})$,  $\dom(\mathcal{F})$ is not homeomorphic to a sphere, a sphere minus one or two points or the torus. Equip $\dom(\mathcal{F})$ with a complex structure. By the uniformization theorem, we can identify the universal cover $\widetilde{\dom(\mathcal{F})}$ with the unit disc. It follows that $\dom(\mathcal{F})$ admits a complete hyperbolic metric $g_{\hyp}$ that lifts to the Poincar\'e metric on $\widetilde{\dom(\mathcal{F})}$. 
We obtain a circle compactification $\widetilde{\dom(\mathcal{F})} \cup S_{\infty}$ by adding the boundary at infinity $S_{\infty}$ to $(\widetilde{\dom(\mathcal{F})}, \widetilde{g_{\hyp}})$. 
In the hyperbolic surface $(\dom(\mathcal{F}), g_{\hyp})$ there is a sequence of compact sub-surfaces, $C_1 \subset \cdots  C_k \subset C_{k+1} \cdots M$, such that their boundary $\partial C_k$ is a finite union of simple closed geodesics, and every non-trivial closed curve in $\dom(\mathcal{F})$ is freely homotopic to a closed curve that is contained in $C_k$ for sufficiently large $k$, see e.g. \cite[Proposition 2.3]{Basmajian2008}. 
In particular every free homotopy class that is not a multiple of a simple class has a geodesic representative. 

For a lift $\widetilde\gamma$ of $\gamma$ to the universal cover $\widetilde{\dom(\mathcal{F})}$ and a deck transformation $U$ on $\widetilde{\dom(\mathcal{F})}$, we write $U\widetilde{\gamma}: \R \to \widetilde\dom(\mathcal{F})$ for the lift satisfying $U\widetilde\gamma(t) = U(\widetilde\gamma(t))$ for every $t\in \R$. We call the deck transformation $T$ with $T\widetilde\gamma(t) = \widetilde\gamma(t+1)$ the \textit{shift} of $\widetilde\gamma$. Note that every deck transformation $U$ extends to a homeomorphism $\overline{U}$ on $\widetilde{\dom(\mathcal{F})} \cup S_{\infty}$ and, since $\Gamma$ has a geodesic representative in $\dom(\mathcal{F})$, any shift $T$ for some lift $\widetilde\gamma$ of $\Gamma$ is a hyperbolic transformation, in particular $\overline{T}$ admits exactly two fixed points $\widetilde{\gamma}^+ = \lim_{t\to+ \infty} \widetilde{\gamma}(t) = \lim_{n\to + \infty} T^n\widetilde\gamma(t) \in S_{\infty}$ and   $\widetilde{\gamma}^- = \lim_{t\to- \infty} \widetilde{\gamma}(t) = \lim_{n\to + \infty} T^{-n}\widetilde\gamma(t)\in S_{\infty}$. 
For any lift $\widetilde\gamma$ we get two non-empty arcs  $\widehat{L}(\widetilde{\gamma}) := (\widetilde\gamma^+, \widetilde\gamma^-)$ and $\widehat{R}(\widetilde{\gamma}) := (\widetilde\gamma^-, \widetilde\gamma^+)$ on $S_{\infty}$,  where we equip $S_{\infty}$ with the counter-clockwise orientation.   
   We say that two lifts $\widetilde\gamma_1: \R \to \widetilde{\dom(\mathcal{F})}$ and 
$\widetilde\gamma_2: \R \to \widetilde{\dom(\mathcal{F})}$ are  \textit{translates of each other} if the shifts $T_1$ and $T_2$ for $\widetilde\gamma_1$ and $\widetilde\gamma_2$, respectively, coincide. 
If the free homotopy class of $\Gamma$ is primitive then this holds if and only if $\widetilde\gamma_1(t+k) = \widetilde\gamma_2(t)$ for some $k\in \Z$ and all $t \in \R$.
Using that $\Gamma$ has a closed geodesic representative in $(\dom(\mathcal{F}), g_{\hyp})$, one obtains that two lifts $\widetilde\gamma_1$ and $\widetilde\gamma_2$ are not translates of each other if and only if  $\widetilde\gamma_1^{\pm}$ and $\widetilde\gamma_2^{\pm}$ are four pairwise distinct points in $S_{\infty}$. Since $\Gamma$ is not freely homotopic to a multiple of a simple loop, there are two lifts $\widetilde\gamma_1$ and $\widetilde\gamma_2$  such that $\widetilde\gamma_1^{\pm}$ \textit{separate} $\widetilde\gamma_2^{\pm}$ in $S_{\infty}$, i.e. $\widehat{L}(\widetilde\gamma_1) \cap \widehat{R}(\widetilde\gamma_2) \neq \emptyset$ and  $\widehat{R}(\widetilde\gamma_1) \cap \widehat{L}(\widetilde\gamma_2) \neq \emptyset$. 

The following lemma states that the foliation separates two lifts with different asymptotics. 
\begin{lem}\label{Lemma1}
Let $\widetilde\gamma_1$ and $\widetilde\gamma_2$ be two lifts of $\gamma$ 
and assume there is $t_1, t_2\in \R$ such that $\phi_{\widetilde\gamma_1(t_1)} = \phi_{\widetilde\gamma_2(t_2)} =: \phi$. 
\begin{itemize}
\item If $\widetilde\gamma_1^- \in \widehat{R}(\widetilde\gamma_2)$ (resp. $\widetilde\gamma_1^- \in \widehat{L}(\widetilde\gamma_2)$), then 
there is $a_1 <  t_1$ and $a_2 <  t_2$ such that $\phi_{\widetilde\gamma_1(a_1)}$ is above (resp. below) $\phi_{\widetilde\gamma_2(a_2)}$ relative to $\phi$. 
\item If $\widetilde\gamma_1^+ \in \widehat{L}(\widetilde\gamma_2)$ (resp. $\widetilde\gamma_1^+  \in (\widehat{R}(\widetilde\gamma_2)$), then there is $b_1 > t_1$ and $b_2 > t_2$ such that 
$\phi_{\widetilde\gamma_1(b_1)}$ is below (resp. above) $\phi_{\widetilde\gamma_2(b_2)}$ relative to $\phi$.
\end{itemize}
\end{lem}

\begin{proof}
Since $\Gamma$ is not freely homotopic to a multiple of a simple loop, we can choose a deck transformation $S$ on $\widetilde{\dom(\mathcal{F})}$ such that for any lift $\widetilde\gamma$ of $\gamma$,  $(S\widetilde\gamma)^{+}\in \widehat{L}(\widetilde\gamma)$ and $(S\widetilde\gamma)^{-}\in \widehat{R}(\widetilde\gamma)$. It holds then more generally that for all lifts $\widetilde\gamma$ and for all $n\in \Z$,   $(T^nS\widetilde\gamma)^+  \in   \widehat{L}(\widetilde\gamma)$, $(T^nS\widetilde\gamma)^{-} \in \widehat{R}(\widetilde\gamma)$, 
$(T^nS^{-1}\widetilde\gamma)^+\in \widehat{R}(\widetilde\gamma)$ and  $(T^nS^{-1}\widetilde\gamma)^- \in \widehat{L}(\widetilde\gamma)$, where $T$ denotes the shift for $\widetilde\gamma$. 

Let now  $\widetilde\gamma_1$ and $\widetilde\gamma_2$ be two lifts considered in the Lemma, and $T_1$ and $T_2$ their shifts. 
We show that if $\widetilde\gamma_1^- \in \widehat{R}(\widetilde\gamma_2)$ then there is $a_1 < t_1$ and $a_2 < t_2$ such that 
$\phi_{\widetilde\gamma_1(a_1)}$ is above $\phi_{\widetilde\gamma_2(a_2)}$ relative to $\phi$.  The other three statements of the Lemma are proved similarly. 
So assume $\widetilde\gamma_1^- \in \widehat{R}(\widetilde\gamma_2)$. Since $\widetilde\gamma_1^- \neq \widetilde\gamma_2^-$ are the unique repelling fixed points in $\dom(\mathcal{F}) \cup S_{\infty}$ of $\overline{T_1}$ and $\overline{T_2}$, respectively, there is $n_1, n_2 \in \N$ such that 
$L(T_1^{-n_1}S\widetilde\gamma_1) \cap R(T_2^{-n_2}S^{-1}\widetilde\gamma_2) = \emptyset$.
Choose $a_1 < t_1$ and $a_2 < t_2$ such that $\widetilde\gamma_1(a_1) \in  L(T_1^{-n_1}S\widetilde\gamma_1)$ and $\widetilde\gamma_2(a_2) \in 
R(T_2^{-n_2}S^{-1}\widetilde\gamma_2)$, see Figure \ref{fig:lemma_2.2}.  Let $\phi_i := \phi_{\widetilde\gamma_i(a_i)}$ for $i=1,2$.
Since the lifts are positively transverse to $\mathcal{F}$ we have that $\phi_1^- \in  L(T_1^{-n_1}S\widetilde\gamma_1)$ and 
$\phi^+_2 \in R(T_2^{-n_2}S^{-1}\widetilde\gamma_2)$, as well as $\phi_1^{+} \in R(\widetilde\gamma_1) \cap R(\widetilde\gamma_2)$ and $\phi_2^- \in L(\widetilde\gamma_1) \cap L(\widetilde\gamma_2)$. 
Hence  $A_1 := L(T_1^{-n_1}S\widetilde\gamma_1) \cup (R(\widetilde\gamma_1) \cap R(\widetilde\gamma_2) \cap R(\phi))$ contains $\phi_1$ and 
$A_2:=R(T_2^{-n_2}S^{-1}\widetilde\gamma_2) \cup (L(\widetilde\gamma_1) \cap L(\widetilde\gamma_2)) \cap R(\phi))$ contains $\phi_2$. $A_1$ and $A_2$ are disjoint, are both connected, are both contained in  $R(\phi)$, and  the points on $\phi$ that are on the boundary of $A_1$ lie above the points of $\phi$ that lie on the boundary of $A_2$. It follows that $\phi_1$ is above $\phi_2$ relative to $\phi$. 
\end{proof}

\begin{figure}
    \centering
    \scalebox{0.8}{
    \begin{tikzpicture}[scale=0.8]
        \coordinate (O) at (0,0);
        
        \draw (O) circle (8);
        
        \path[name path=lift1] (265:8) .. controls (1,-5) and (0,6) .. (70:8) node[above]{\Large $\tilde{\gamma}_1^+$} node[pos=0,below]{\Large $\tilde{\gamma}_1^-$};
        \path[name path=lift2] (200:8) .. controls (-3,0) and (4,2) .. (40:8) node[right]{\Large $\tilde{\gamma}_2^+$} node[pos=0,left]{\Large $\tilde{\gamma}_2^-$};
        \path[name path=lift3] (300:8) .. controls (0,-5) .. (240:8) node[left]{\Large $T_1^{-n_1} S \tilde{\gamma}_1$};
        \path[name path=lift4] (160:8) .. controls (-5,0) and (-4,-2) .. (220:8) node[left]{\Large $T_2^{-n_2} S^{-1} \tilde{\gamma}_2$};
        
        \path[name path=leaf1] (130:8) .. controls (2,2) and (4,0) .. (0:8) node[pos=0.8,above]{\Large $\phi$};
        \path[name path=leaf2] (250:8) .. controls (0,-7.5) and (2,-5) .. (320:8) node[pos=0.8,above]{\Large $\phi_1$};
        \path[name path=leaf3] (145:8) .. controls (-5,3) .. (210:8) node[pos=0.2,right]{\Large $\phi_2$};
        
        \foreach \i in {1,2,3,4} {
            \draw[use path=lift\i, lift];
        }
        
        \foreach \j in {1,2,3} {
            \draw[use path=leaf\j, leaf];
        }
        
        \foreach \i in {1,2} {
            \foreach \j in {1,2,3} {
                \path[name intersections={of={lift\i} and {leaf\j}, by={I\i\j}}];
                \node at (I\i\j) {\textbullet};
            }
        }
        
        \node[anchor=west] at (I21) {$\tilde{\gamma}_2(t_2)$};
        \node[anchor=south west] at (I11) {$\tilde{\gamma}_1(t_1)$};
        \node[anchor=north west] at (I12) {$\tilde{\gamma}_1(a_1)$};
        \node[anchor=south east] at (I23) {$\tilde{\gamma}_2(a_2)$};
        
        \path[name path=circle_right] (0:8) arc (0:-120:8);
        \path[name path=circle_left] (130:8) arc (130:220:8);
        
        \path[name path=right1, intersection segments={of=lift3 and lift1, sequence={A1[reverse] -- B1}}];
        \path[name path=right2, intersection segments={of=right1 and lift2}];
        \path[name path=right3, intersection segments={of=right2 and leaf1}];
        \path[name path=right, intersection segments={of={right3 and circle_right}, sequence={A1 -- B1}}, pattern=mydots];
        
        \path[name path=left1, intersection segments={of=lift4 and lift2, sequence={A1[reverse] -- B1}}];
        \path[name path=left2, intersection segments={of=left1 and lift1}];
        \path[name path=left3, intersection segments={of=leaf1 and left2, sequence={A0 -- B0[reverse]}}];
        \path[name path=left, intersection segments={of={left3 and circle_left}, sequence={A* -- B*}}, pattern=mydots];
        
        \draw[->] (-20:10) .. controls (8,-1) .. (-20:6) node[pos=0,below] {$A_1$};
        \draw[->] (170:10) .. controls (-9,0) .. (180:7) node[pos=0,above] {$A_2$};
    \end{tikzpicture}
    }
    \caption{Lifts and leaves in the proof of Lemma \ref{Lemma1}. \\ The dotted areas are $A_1 = L(T_1^{-n_1} S \tilde{\gamma}_1) \cup (R(\tilde{\gamma}_1) \cap R(\tilde{\gamma}_2) \cap R(\phi))$, \\ $A_2 = R(T_2^{-n_2} S^{-1} \tilde{\gamma}_2) \cup (L(\tilde{\gamma}_1) \cap L(\tilde{\gamma}_2) \cap R(\phi))$.}
    \label{fig:lemma_2.2}
\end{figure}
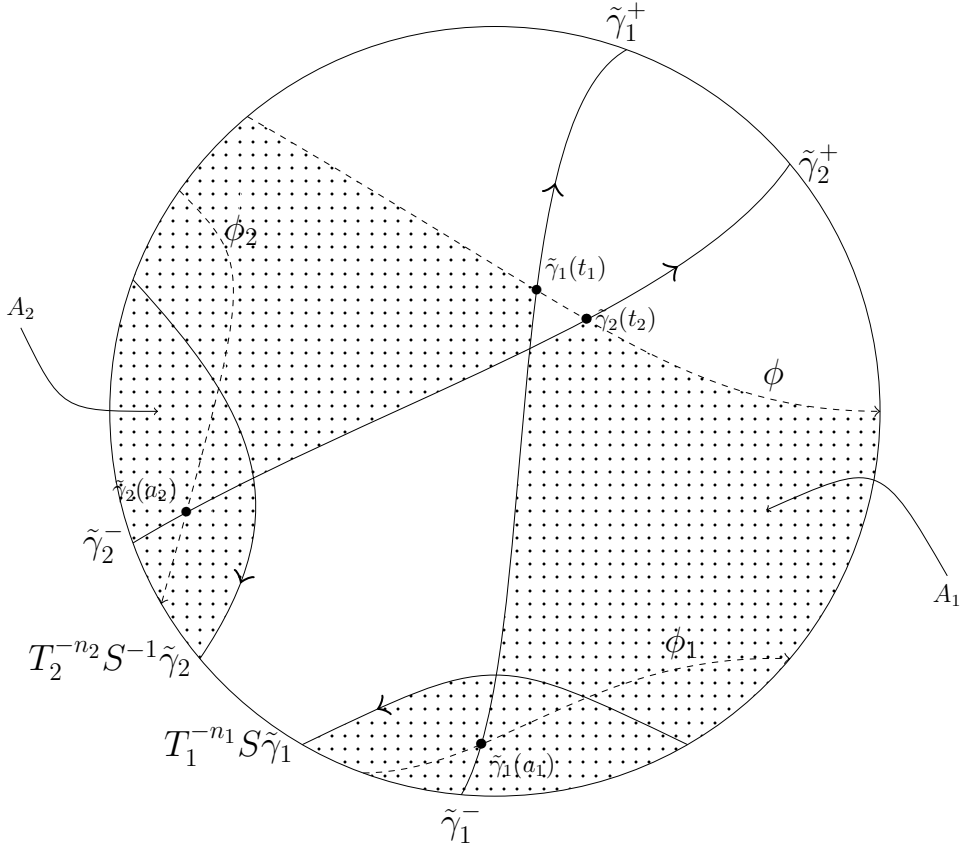

\begin{cor}\label{cor2}
Let $\widetilde\gamma_1$ and $\widetilde\gamma_2$ be two lifts of $\gamma$ such that 
$\widetilde{\gamma}^{\pm}_1$ separate $\widetilde{\gamma}_2^{\pm}$ on $S_{\infty}$. Then $\widetilde\gamma_1$ and $\widetilde\gamma_2$ intersect $\widetilde{\mathcal{F}}$-transversally. 
In particular, $\Gamma$ has at least one $\mathcal{F}$-transverse self-intersection.
\end{cor}
\begin{proof}
The first statement follows directly from the Lemma and the definition of an $\widetilde{\mathcal{F}}$-transverse intersection. 
Since $\Gamma$ lies in a non-simple free homotopy class, there are at least two such lifts $\widetilde\gamma_1$ and $\widetilde\gamma_2$ of $\gamma$. 
\end{proof}

We also note:
\begin{lem}\label{lem:transverse_sep}
If two lifts $\widetilde\gamma_1$ and $\widetilde\gamma_2$ of $\gamma$ intersect $\widetilde{\mathcal{F}}$-transversally, then $\widetilde\gamma_1^{\pm}$ separate $\widetilde\gamma_2^{\pm}$ in $S_{\infty}$.
\end{lem}

\begin{proof}
Let $\widetilde\gamma_1$ and $\widetilde\gamma_2$ two lifts that intersect $\widetilde{\mathcal{F}}$-transversally.
Since $\Gamma$ has a closed geodesic representative and since the boundary at infinity $S_{\infty}$ of $(\widetilde{\dom(\mathcal{F})}, \widetilde{g_{\hyp}})$ can be identified with equivalence classes of geodesic rays, where two geodesic rays are equivalent if and only if they stay in finite distance to each other, it suffices to show that $\widetilde{d}(\widetilde\gamma_1(t), \widetilde\gamma_2(t)) \to +\infty$ as $t\to \pm\infty$, where $\widetilde{d}$ is the metric on $\widetilde{\dom(\mathcal{F})}$ induced by Poincar\'e metric $\widetilde{g_{\hyp}}$. But this follows directly from the proof of Lemma \ref{lem:transverse_not_simple} by considering in that proof as a metric on $\dom(\mathcal{F})$ the hyperbolic metric $g_{\hyp}$.   
\end{proof}

We proceed with a few further observations.
\begin{lem}\label{Lemma4}
There are two lifts $\widetilde\gamma_1$ and $\widetilde\gamma_2$ of $\gamma$ such that 
\begin{itemize}
\item $\widetilde\gamma_1$ and $\widetilde\gamma_2$ intersect $\widetilde{\mathcal{F}}$-transversally and positively, 
\item there are no lifts $\widetilde\gamma$ of $\gamma$  that intersect both  $\widetilde\gamma_1$ and $\widetilde\gamma_2$ $\widetilde{\mathcal{F}}$-transversally and positively, and
\item there are no lifts $\widetilde\gamma$ of $\gamma$ that intersect both $\widetilde\gamma_1$ and $\widetilde\gamma_2$ $\widetilde{\mathcal{F}}$-transversally and negatively. 
\end{itemize}
\end{lem}
\begin{rem}\label{rem:afterLemma4} Note that if $\widetilde\gamma_1$ and $\widetilde\gamma_2$ satisfy the properties stated in the Lemma, then so do $U\widetilde\gamma_1$ and $U\widetilde\gamma_2$ for any deck transformation $U$ on $\widetilde{\dom(\mathcal{F})}$, as well as $\widetilde\gamma_1$ and $T_1\widetilde\gamma_2$, where $T_1$ is the shift for $\widetilde\gamma_1$. 
\end{rem}

\begin{proof}[Proof of \ref{Lemma4}]
Recall that two lifts are translates of each other if their shifts coincide. We will denote in this proof by $[\widetilde\gamma]$ the equivalence class of a lift $\widetilde\gamma$ with respect to this equivalence relation.   
Consider in the following all pairs  $([\widetilde\gamma], [\widetilde\gamma'])$, where $\widetilde\gamma$ and $\widetilde\gamma'$ are lifts of $\gamma$ that  intersect $\widetilde{\mathcal{F}}$-transversally and positively. For each pair $\mathfrak{p} = ([\widetilde\gamma], [\widetilde\gamma'])$ consider the (well-defined) arcs  $C(\mathfrak{p}) := (\widetilde\gamma^{+}, \widetilde\gamma'^{-})$ and 
$D(\mathfrak{p}) := (\widetilde\gamma^-, \widetilde\gamma'^+)$ in $S_{\infty}$, see Figure \ref{fig:lemma_4}. 
We say $\mathfrak{p}_1 = ([\widetilde\gamma_1], [\widetilde\gamma_1']) \leq \mathfrak{p}_2 = ([\widetilde\gamma_2], [\widetilde\gamma_2'])$  
if $C(\mathfrak{p}_2) \subset C(\mathfrak{p}_1)$ and $D(\mathfrak{p}_2) \subset D(\mathfrak{p}_1)$. '$\leq$' defines a partial order on the set of pairs above. 

For a given pair $\mathfrak{p} = ([\widetilde\gamma], [\widetilde\gamma'])$ as above there are only finitely many equivalence classes $[\widetilde\gamma_*]$ such that for a representative $\widetilde\gamma_*$, 
\begin{align}\label{A1}
\widetilde\gamma_*^{+} \in  C(\mathfrak{p}) \text{  and } \widetilde\gamma_*^- \in D(\mathfrak{p}), 
\end{align}
as well as finitely many classes  $[\widetilde\gamma_*]$ such that for a representative $\widetilde\gamma_*$, 
\begin{align}\label{A2}
\widetilde\gamma_*^{-} \in  C(\mathfrak{p}) \text{ and } \widetilde\gamma_*^+ \in D(\mathfrak{p}).
\end{align}
Indeed, note that if $\widetilde\gamma_*$ satisfies \eqref{A1} or \eqref{A2}, then it intersects $\widetilde\gamma$ and $\widetilde\gamma'$. Also, the group of deck transformations act freely and co-compactly, so there are only finitely many deck transformations $U$ such that $U\widetilde\gamma|_{[0,1]} \cap \widetilde\gamma|_{[0,1]}\neq \emptyset$. 
Hence, for every lift $\widetilde\gamma_*$ with  $\widetilde\gamma_*\cap \widetilde\gamma \neq \emptyset$, there is $r_*, r \in \Z$ such that $T^r U \widetilde\gamma = T_*^{r_*} \widetilde\gamma_*$, for one of the finite deck transformation $U$ from above, where $T$, $T_*$ denote the shift for $\widetilde\gamma$, and $\widetilde\gamma_*$ respectively. In particular, $[\widetilde\gamma_*] = [T^r U \widetilde\gamma]$. 
Assume $(U \widetilde\gamma)^+ \in (\widetilde\gamma^+, \widetilde\gamma^-)$. Then, for all $n\in \Z$,  $T^n U \widetilde\gamma$ does not satisfy \eqref{A2}. Furthermore, note that $\overline{T}$ restricted to $S_{\infty}$ has exactly two fixed points, one repelling fixed point,   $\widetilde\gamma^-$, and one attracting fixed point, $\widetilde\gamma^+$. Hence there is $r < s$ such that no representative of $[T^n U \widetilde\gamma]$ satisfies \eqref{A1} for $n \notin [r, s]$. Similarly, if $(U \widetilde\gamma)^+ \in (\widetilde\gamma^-, \widetilde\gamma^+)$, then there is $r < s$ such that no representative of $[T^n U \widetilde\gamma]$ satisfies \eqref{A1} or \eqref{A2} for $n \notin [r, s]$.
Hence there are only finitely many equivalence classes $[\widetilde\gamma_*]$ such that its representatives satisfy \eqref{A1} or \eqref{A2}. 

We conclude that for every pair $\mathfrak{p}$ there are only finitely many pairs that are $\geq \mathfrak{p}$ and so there exists maximal elements. 
Let $\widetilde\gamma_1$ and $\widetilde\gamma_2$ be lifts of $\gamma$ such that $\mathfrak{p}_0:=([\widetilde\gamma_1], [\widetilde\gamma_2])$ is maximal with respect to '$\leq$'. 
In particular, there is no $\widetilde\gamma_*$ that satisfies \eqref{A1} or \eqref{A2} for $\mathfrak{p}_0$, which means, using Lemma \ref{lem:transverse_sep}, that the properties stated in the Lemma hold for the lifts $\widetilde\gamma_1$ and $\widetilde\gamma_2$.
\end{proof}

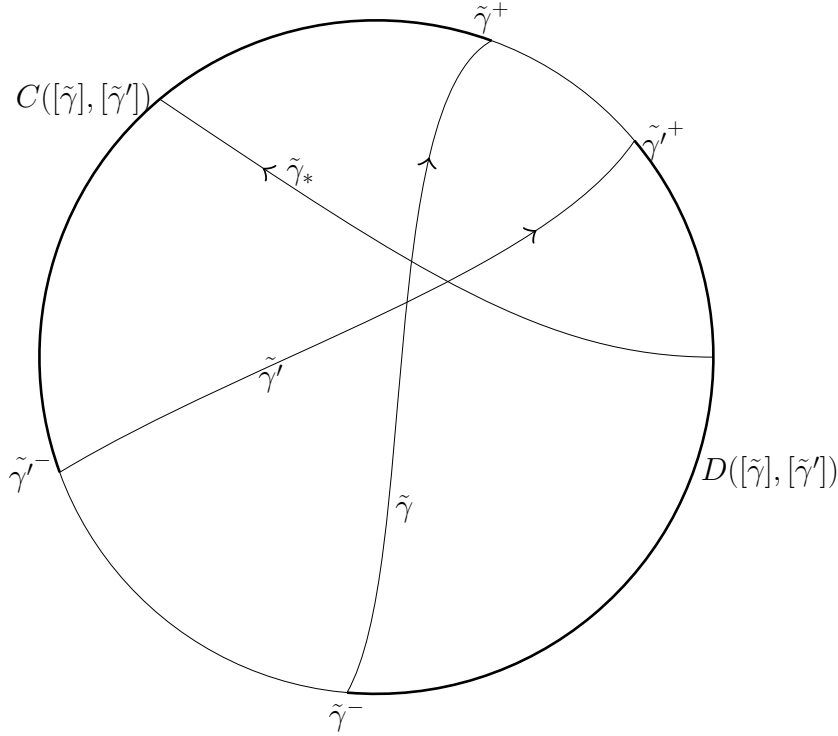
\begin{figure}
    \centering
    \scalebox{0.7}{
    \begin{tikzpicture}[scale=0.8]
        \coordinate (O) at (0,0);

        \draw (O) circle (8);
        
        \path[name path=lift1] (265:8) .. controls (1,-5) and (0,6) .. (70:8) node[above]{\Large $\tilde{\gamma}^+$} node[pos=0,below]{\Large $\tilde{\gamma}^-$} node[pos=0.3, right]{\Large $\tilde{\gamma}$};
        \path[name path=lift2] (200:8) .. controls (-3,0) and (4,2) .. (40:8) node[right]{\Large $\tilde{\gamma'}^+$} node[pos=0,left]{\Large $\tilde{\gamma'}^-$} node[pos=0.3, right]{\Large $\tilde{\gamma'}$};
       
              \path[name path=lift3] (0:8) .. controls (4,0) and (1,2) .. (130:8) node[pos=0.8,above]{\Large $\tilde{\gamma}_*$};

				\path[name path=circle_right] (70:8) arc (70:200:8);
        \path[name path=circle_left] (40:8) arc (40:-95:8);
        \draw[use path=circle_right, line width=0.5mm];
				\draw[use path=circle_left, line width=0.5mm];

        \foreach \i in {1,2,3} {
            \draw[use path=lift\i, lift];
        }

		\draw node[anchor=east] at (130:8) {\Large $C([\tilde{\gamma}],[ \tilde{\gamma}'])$};
		\draw node[anchor=west] at (340:8){\Large $D([\tilde{\gamma}], [\tilde{\gamma}'])$};
      
    \end{tikzpicture}
    }
    \caption{Lifts in the proof of Lemma \ref{Lemma4}. For each pair $([\widetilde\gamma], [\widetilde\gamma'])$, there are only finitely many $[\widetilde\gamma_*]$ such that ${\widetilde\gamma_*}^+ \in C([\widetilde\gamma], [\widetilde\gamma'])$ and ${\widetilde\gamma_*}^- \in D([\widetilde\gamma], [\widetilde\gamma'])$. }
    \label{fig:lemma_4}
\end{figure}

We get the following consequence of Lemma \ref{Lemma4}. 

\begin{lem}\label{Lemma6}
There is $\overline{t}, \underline{t} \in [0,1)$ with $\Gamma(\overline{t}) = \Gamma(\underline{t})$ such that there exists two lifts $\widetilde\gamma_1$ and $\widetilde\gamma_2$ of $\Gamma$ to $\widetilde{\dom(\mathcal{F})}$ that intersect $\widetilde{\mathcal{F}}$-transversally and positively in $\widetilde\gamma_1(\overline{t}) = \widetilde\gamma_2(\underline{t})$, and such that  
\begin{enumerate}
\item if two lifts $\widetilde\gamma$ and $\widetilde\gamma'$ intersect $\widetilde{\mathcal{F}}$-transversally and positively, then $\widetilde\gamma$  does not intersect the leaves $\phi_{\widetilde\gamma'(\overline{t} +k)}$, for all $k\in \Z$; and
\item   if two lifts $\widetilde\gamma$ and $\widetilde\gamma'$ intersect $\widetilde{\mathcal{F}}$-transversally and negatively, then $\widetilde\gamma$  does not intersect the leaves $\phi_{\widetilde\gamma'(\underline{t} +k)}$, for all $k\in \Z$.
\end{enumerate}
\end{lem}

\begin{proof}
Let $\widetilde\gamma_1$ and $\widetilde\gamma_2$ be two lifts of $\gamma$ that satisfy the properties in Lemma \ref{Lemma4}, in particular, $\widetilde\gamma_1$ and $\widetilde\gamma_2$ intersect $\widetilde{\mathcal{F}}$-transversally and positively in $ \widetilde\gamma_1(\overline{t})= \widetilde\gamma_2(\underline{t})$ for some $\overline{t}, \underline{t}\in \R$. By Remark \ref{rem:afterLemma4} we can choose $\widetilde\gamma_1$ and $\widetilde\gamma_2$ such that also $\overline{t}, \underline{t} \in [0,1)$. Furthermore by the same remark the properties of Lemma \ref{Lemma4} are satisfied for $\widetilde\gamma_1$ and  $\widetilde\gamma_{k,2}:={T_1^k}\widetilde\gamma_2$, for any $k\in \Z$, where $T_1$ is the shift for $\widetilde\gamma_1$. Note that $\widetilde\gamma_1$ and $\widetilde\gamma_{k,2}$ intersect $\widetilde{\mathcal{F}}$-transversally in $\widetilde\gamma_1(\overline{t} + k)= \widetilde\gamma_{k,2}(\underline{t})$. Analogously, for all $k\in \Z$, the properties of Lemma \ref{Lemma4} hold for  $\widetilde\gamma_{k,1} = {{T}^k_2}\widetilde\gamma_1$ and $\widetilde\gamma_2$, where $T_2$ is the shift for $\widetilde\gamma_2$. 

We prove now (1), the proof of (2) is analogous.
So let $\widetilde\gamma$ and $\widetilde\gamma'$ be two lifts of $\gamma$ that intersect $\widetilde{\mathcal{F}}$-transversally and positively. 
Let $U$ be the deck transformation such that $U\widetilde\gamma' = \widetilde\gamma_1$. Then, for all $k\in \Z$,  $\widetilde\gamma'=U^{-1}\widetilde\gamma_1$ and  $U^{-1}\widetilde\gamma_{k,2}$ satisfy the properties of Lemma $\ref{Lemma4}$ and they intersect $\widetilde{\mathcal{F}}$-transversally and positively in $\widetilde\gamma'(\overline{t} + k)$.  
Assume that $\widetilde\gamma$ intersects the leaf $\phi_{\widetilde\gamma'(\overline{t} +k)}$ for some $k\in \Z$. Then, $\widetilde\gamma$ and $\widetilde\gamma'$ also intersect $\mathcal{F}$-transversally and positively at $\phi_{\widetilde{\gamma'}(\overline{t} + k)}$. By the transitivity property, see Remark \ref{rem:transitive}, $\widetilde\gamma$ intersects $U^{-1}\widetilde\gamma_{k,2}$ $\widetilde{\mathcal{F}}$-transversally and positively, which contradicts the properties of Lemma $\ref{Lemma4}$ for the lifts $\widetilde\gamma'$ and  $U^{-1}\widetilde\gamma_{k,2}$.

\end{proof}

\begin{lem}\label{Lemma7}
Let $a_0,a_1,b_0,b_1 \in \R$ with $a_0 < b_0$ and $a_1 < b_1$. Let $\widetilde\gamma_0, \widetilde\gamma_1$ be two lifts of $\gamma$ which are not translates of each other and
 such that 
$\widetilde\gamma_0|_{[a_0,b_0]}$ is equivalent to $\widetilde\gamma_1|_{[a_1,b_1]}$. Then 
\begin{enumerate}
\item $\max\{b_0-a_0, b_1 -a_1\} \leq 1$ if $\widetilde\gamma_0$ and $\widetilde\gamma_1$ intersect $\widetilde{\mathcal{F}}$-transversally,
\item $\min\{b_0-a_0, b_1 -a_1\} > \max\{\lfloor b_0-a_0 \rfloor, \lfloor b_1 -a_1 \rfloor\} - 2$,
\item $\max\{b_0-a_0, b_1 -a_1\} < 6$. 
\end{enumerate}
\end{lem}
\begin{proof}
Part $(1)$ follows directly from Lemma \ref{Lemma6}.
For (2) it is sufficient, by symmetry, to show that $b_0-a_0 > \lfloor b_1-a_1 \rfloor -2$, and to show this, we may assume $b_1-a_1 \geq 3$. 
Choose $\overline{t} \in [0,1)$ according to Lemma \ref{Lemma6}. Let $k<  l \in \Z$ with $a_1\leq \overline{t} + k< a_1 +1$, and $b_1-1 < \overline{t} + l \leq b_1$. 
Any lift $\widetilde\gamma$ that intersects $\widetilde\gamma_1$  $\widetilde{\mathcal{F}}$-transversally and positively in $\widetilde\gamma(t) = \widetilde\gamma_1(t_1)$ for some $t\in \R$ and $t_1 \in (\overline{t} +k, \overline{t} +l)$, intersects already 
$\widetilde\gamma_1|_{[\overline{t} +k, \overline{t} +l]}$ $\widetilde{\mathcal{F}}$-transversally and positively, and hence also $\widetilde\gamma_0|_{[a_0,b_0]}$ $\widetilde{\mathcal{F}}$-transversally and positively in  $\widetilde\gamma(t') =  \widetilde\gamma_0(t_0)$ for some $t'\in \R$ and $t_0 \in (a_0,b_0)$. The latter holds since a subpath of $\widetilde\gamma_0|_{[a_0,b_0]}$ is equivalent to $\widetilde\gamma_1|_{[\overline{t} +k, \overline{t} +l]}$. 
There are finitely many (pairwise non-identical) images $\{\widetilde\gamma(t)\, |\, t\in \R\}$ of such lifts $\widetilde\gamma$, say $N$. 
Since $\Gamma$ has at least one ${\mathcal{F}}$-transverse self-intersection, there are, for any subpath of some lift of $\gamma$ of length $<l-k-1$, strictly less then $N$ images of lifts that intersect this subpath $\widetilde{\mathcal{F}}$-transversally and positively. Hence $b_0 -a_0 \geq l-k-1$. On the other hand, $l-k > \lfloor b_1 - a_1 \rfloor -1$, and $(2)$ follows. 

To show $(3)$, we argue by contradiction. So assume the contrary, w.l.o.g. $b_1 -a_1 \geq 6$. By $(1)$, we can assume that $\widetilde\gamma_0$ and $\widetilde\gamma_1$ do not intersect $\widetilde{\mathcal{F}}$-transversally.  
By $(2)$, $b_0 - a_0 > 4$. 
Consider $\widetilde\gamma_2 := T_0\widetilde\gamma_1$, where $T_0$ is the shift for $\widetilde\gamma_0$, see Figure \ref{fig:lemma_2.7}.
$\widetilde\gamma_2|_{[a_1, b_1]}$ is equivalent to $\widetilde\gamma_0|_{[a_0 +1, b_0 +1]}$.
Let $c_1\in (a_1, b_1)$ be the parameter such that $\widetilde\gamma_2|_{[a_1, c_1]}$ is equivalent to $\widetilde\gamma_0|_{[a_0 +1, b_0]}$.  
So, again by $(2)$, $c_1 -a_1 > \lfloor b_0 - (a_0 +1) \rfloor - 2 \geq 1$. 
On the other hand, $\widetilde\gamma_2|_{[a_1, c_1]}$, which is equivalent to $\widetilde\gamma_0|_{[a_0 +1, b_0]}$, is by assumption equivalent to a subpath of $\widetilde\gamma_1|_{[a_1, b_1]}$. 
Note now, that $\widetilde\gamma_2$ intersects $\widetilde\gamma_1$ $\widetilde{\mathcal{F}}$-transversally. Indeed, since $\widetilde\gamma_1$ and $\widetilde\gamma_0$ are no translates of each other, the attracting fixed point of $\overline{T_0}$ in $S_{\infty}$ is disjoint from $\widetilde\gamma_1^+$. It follows that either $\widetilde\gamma_2^{\pm} = (T_0 \widetilde\gamma_1)^{\pm}$ separate $\widetilde\gamma_1^{\pm}$, or these limit points have one of the following four orders on $S_{\infty}$: 
$(a) \, \widetilde\gamma_1^-, \widetilde\gamma_1^+, \widetilde\gamma_2^+, \widetilde\gamma_2^-$; $(b)\,  \widetilde\gamma_1^+, \widetilde\gamma_1^-, \widetilde\gamma_2^-, \widetilde\gamma_2^+$; $(c) \, \widetilde\gamma_1^-, \widetilde\gamma_1^+, \widetilde\gamma_2^-, \widetilde\gamma_2^+$; or $(d) \,  \widetilde\gamma_1^+, \widetilde\gamma_1^-,  \widetilde\gamma_2^+, \widetilde\gamma_2^-$.
Since the asymptotics of $\widetilde\gamma_0$ and $\widetilde\gamma_1$ do not separate in $S_{\infty}$, $\overline{T_0}$ turns $\widetilde\gamma_1^-$ and $\widetilde\gamma_1^+$ to the same direction on $S_{\infty}$ and so $(a)$ and $(b)$ can be excluded. Since there are leaves of $\widetilde{\mathcal{F}}$ that intersect both $\widetilde\gamma_1$ and $\widetilde\gamma_2$, $(c)$ and $(d)$ can be excluded. 
Hence $\widetilde\gamma_2^{\pm}$ separate $\widetilde\gamma_1^{\pm}$ and by Corollary \ref{cor2}, $\widetilde\gamma_1$ and $\widetilde\gamma_2$ intersect $\widetilde{\mathcal{F}}$-transversally. By $(1)$,  $c_1 -a_1 \leq 1$, and we obtain a contradiction. 
\end{proof}

\begin{figure}
    \centering
    \scalebox{0.8}{
    \begin{tikzpicture}[scale=0.8]
        \coordinate (O) at (0,0);
        
        \draw (O) circle (8);
        
        \path[name path=lift1] (225:8) .. controls (-5,-6) and (0,6) .. (70:8) node[above]{\Large $\widetilde{\gamma}_0$};
        \path[name path=lift2] (240:8) .. controls (0,-4) and (2,4) .. (40:8) node[right]{\Large $\widetilde{\gamma}_1$};
        \path[name path=lift3] (255:8) .. controls (0,-7) and (1,5) .. (55:8) node[right]{\Large $\widetilde{\gamma}_2 = T_0\widetilde{\gamma}_1$};
        
        \path[name path=leaf1] (-4,-1) .. controls (-2,-3) .. (-1,-5);
        \path[name path=leaf2] (-4,0) .. controls (-1,-2) .. (1,-4);
        \path[name path=leaf3] (-2,3) .. controls (1,2) .. (4,1);
        \path[name path=leaf4] (-1,5) .. controls (1,3) .. (2.5,3);

        \foreach \i in {1,2,3} {
            \foreach \j in {1,2,3,4} {
                \ifthenelse{\i = 2 \AND \j = 2}
                { }
                {
                    \path[name intersections={of={lift\i} and {leaf\j}, by={I\i\j}}];
                    \node at (I\i\j) {\textbullet};
                }
            }
        }

        \foreach \i in {1,2,3} {
            \draw[use path=lift\i, lift];
        }
        
        \foreach \j in {1,2,3,4} {
            \draw[use path=leaf\j, leaf];
        }
        
        \node[anchor=east] at (I11) {$\widetilde{\gamma}_0(a_0)$};
        \node[anchor=west] at (I12) {$\widetilde{\gamma}_0(a_0+1)$};
        \node at ($(I13)+(-1,-0.2)$) {$\widetilde{\gamma}_0(b_0)$};
        \node[anchor=east] at (I14) {$\widetilde{\gamma}_0(b_0+1)$};
        \node[anchor=east] at (I21) {$\widetilde{\gamma}_1(a_1)$};
        \node at ($(I23)+(0.9,0.2)$) {$\widetilde{\gamma}_1(b_1)$};
        \node[anchor=west] at (I32) {$\widetilde{\gamma}_2(a_1)$};
        \node at ($(I33)+(-0.8,-0.2)$) {$\widetilde{\gamma}_2(c_1)$};
        \node at ($(I34)+(-0.5,0.5)$) {$\widetilde{\gamma}_2(b_1)$};
    \end{tikzpicture}
    }
    \caption{Lifts and leaves in the proof of Lemma \ref{Lemma7} $(3)$.}
    \label{fig:lemma_2.7}
\end{figure}
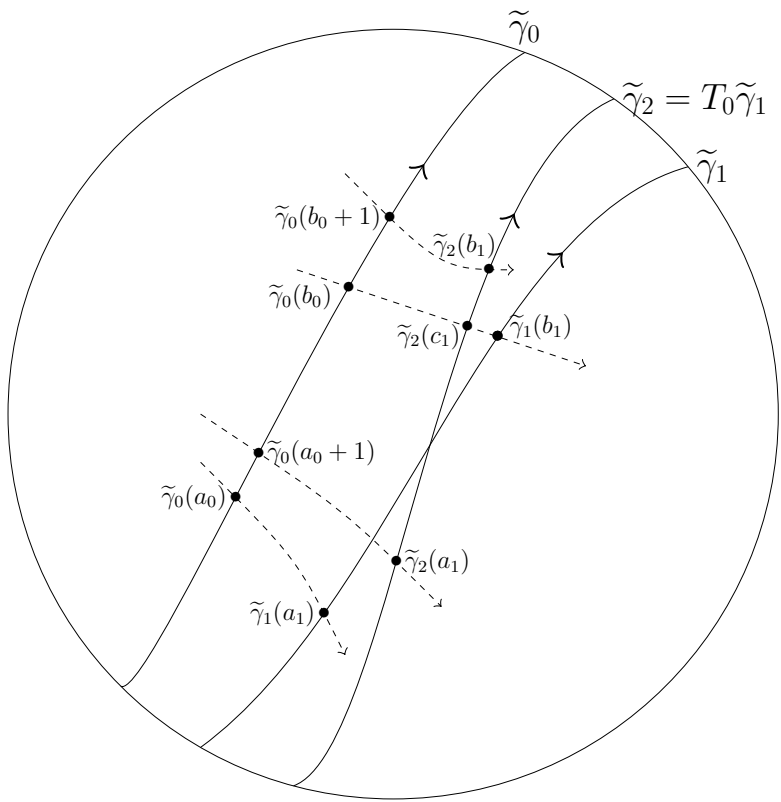

We end this section with the following two lemmata about the existence of a convenient choice of loops equivalent to $\Gamma$.    

\begin{lem}\label{Lemma9}
Let $\Gamma$ be any $\mathcal{F}$-transverse loop and assume that $[\Gamma]_{\widehat{\pi}(\dom(\mathcal{F}))} = m \alpha$, where $\alpha$ is a primitive free homotopy class and $m>1$. Then there is an $\mathcal{F}$-transverse loop $\Gamma'$ with $[\Gamma']_{\widehat{\pi}(\dom(\mathcal{F}))} = \alpha$ such that $\Gamma$ is equivalent to $(\Gamma')^m$.
\end{lem}
\begin{lem}\label{Lemma10}
Let $\Gamma$ be an $\mathcal{F}$-transverse loop such that $[\Gamma]_{\widehat{\pi}(\dom(\mathcal{F}))}\in \widehat{\pi}(\dom(\mathcal{F}))$ is primitive, and let $k:=\selfi_{\dom(\mathcal{F})}([\Gamma])$. Then, up to a modification of $\Gamma$ in its equivalence class, there are pairwise distinct points $x_1, \cdots x_k \in \dom(\mathcal{F})$, pairwise distinct parameters $t_1, \cdots t_k$, $t'_1, \cdots t'_k \in [0,1)$ with  $t_i < t'_i$, for $i\in \{1, \cdots, k\}$,  and  lifts $\widetilde\gamma_1, \cdots \widetilde\gamma_k$ of $\Gamma$, pairwise no translates of each other, such that, for all $i\in \{1, \cdots, k\}$, $x_i = \Gamma(t_i) = \Gamma(t'_i)$ and $\widetilde\gamma$ and $\widetilde\gamma_i$ intersect $\widetilde{\mathcal{F}}$-transversally (positively or negatively) in $\widetilde\gamma(t_i) = \widetilde\gamma_i(t'_i)$. 
\end{lem}
\begin{proof}[Proof of  Lemma \ref{Lemma9}]
Let $\widetilde\gamma$ be a lift of $\Gamma$ and $T$ the shift for $\widetilde\gamma$. There is a deck transformation $S$ on $\widetilde{\dom(\mathcal{F})}$ such that $S^m = T$ and such that after identifying $S$ with an element in the fundamental group of $\dom(\mathcal{F})$its projection to $\widehat{\pi}(\dom(\mathcal{F}))$ is $\alpha$. Consider the lifts $S^k \widetilde\gamma$ of $\Gamma$ for $k= 0,1, \cdots$. 
We first show that these lifts are all pairwise equivalent in $\widetilde{\dom(\mathcal{F})}$. For that it suffices to show that $S\widetilde\gamma$ is equivalent to $\widetilde\gamma$. 
Note that $S^k\widetilde\gamma$ intersects $\widetilde\gamma$ for all $k\in \N$, otherwise, for some $k\in \N$,  $S^k\widetilde\gamma$ is completely on the left or completely on the right of $\widetilde\gamma$ and by induction this will hold for $S^{km}\widetilde\gamma$, a contradiction to $S^{km} = T^k$. For each $k\in \N$ choose $s_k, t_k \in \R$ such that $\widetilde\gamma(s_k) = S^k\widetilde\gamma(t_k)$. We argue by contradiction and assume that $S\widetilde\gamma$ is not equivalent to $\widetilde\gamma$. Then there is $s', t' \in \R$ such that $\phi_{\widetilde\gamma(s')}$ is above or below $\phi_{S\widetilde\gamma(t')}$ relative to $\phi_{\widetilde\gamma(s_1)}$. Assume that $\phi_{\widetilde\gamma(s')}$ is below $\phi_{S\widetilde\gamma(t')}$ relative to $\phi_{\widetilde\gamma(s_1)}$, and $\phi_{\widetilde\gamma(s')}$ and $\phi_{S\widetilde\gamma(t')}$ are on the left of $\phi_{\widetilde\gamma(s_1)}$, the other cases are treated similarly. 
We will deduce by induction that for all $k=1,2, \cdots$ there is $r_k \in \R$ such that  $\phi_{\widetilde\gamma(s')}$ is below $\phi_{S^k\widetilde\gamma(t')}$ relative to $\phi_{\widetilde\gamma(r_k)}$ and   $\phi_{\widetilde\gamma(s')}$ and $\phi_{S^k\widetilde\gamma(t')}$ are on the left of $\phi_{\widetilde\gamma(r_k)}$.  
This will give a contradiction for $k=m$. 
So let $k\geq 1$ and assume that there is $r_k \in \R$ such that $\phi_{\widetilde\gamma(s')}$ is below $\phi_{S^k\widetilde\gamma(t')}$ relative to $\phi_{\widetilde\gamma(r_k)}$, and $\phi_{\widetilde\gamma(s')}$ and $\phi_{S^k\widetilde\gamma(t')}$ are on the left of $\phi_{\widetilde\gamma(r_k)}$.  
Note that $S^k\widetilde\gamma$ and $S^{k+1}\widetilde\gamma$ intersect in $S^k\widetilde\gamma(s_1) = S^{k+1}\widetilde\gamma(t_1)$, 
$\phi_{S^k\widetilde\gamma(s')}$ is below $\phi_{S^{k+1}\widetilde\gamma(t')}$ relative to $\phi_{S^k\widetilde\gamma(s_1)}$, and $\phi_{S^k\widetilde\gamma(s')}$ and $\phi_{S^{k+1}\widetilde\gamma(t')}$ are on the left of $\phi_{S^k\widetilde\gamma(s_1)}$. 
In particular, there are subpaths of $\widetilde\gamma$,  $S^k\widetilde\gamma$, and of $S^{k+1}\widetilde\gamma$, respectively, say $\gamma_i:[a_i,b_i] \to \dom(\mathcal{F})$, $a_i \leq b_i$ for $i=1,2,3$, 
such that $\gamma_1(b_1) = \gamma_2(a_2)$, $\gamma_1(a_1) = \gamma_3(a_3)$, and $\gamma_2(b_2) = \gamma_3(b_3)$, and 
\begin{align*}
\begin{split}
&\{ \gamma_1(b_1) = \gamma_2(a_2), \,   \gamma_1(a_1) = \gamma_3(a_3),  \,  \gamma_2(b_2) = \gamma_3(b_3)\} = \\&\{\widetilde{\gamma}(s_k) = S^k\widetilde\gamma(t_k), \, \widetilde\gamma(s_{k+1}) = S^{k+1}\widetilde\gamma(t_{k+1}), \, S^k\widetilde\gamma(s_1) = S^{k+1}\widetilde\gamma(t_1) \}.
\end{split}
\end{align*} 
It follows with $r_{k+1} = \min\{s_k,s_{k+1}\}$ that $\phi_{S^k\widetilde\gamma(s')}$ is below $\phi_{S^{k+1}\widetilde\gamma(t')}$ relative to $\phi_{\widetilde\gamma(r_{k+1})}$.  Indeed, if $\phi_{\widetilde\gamma(r_{k+1})}$ is on the right of $\phi_{S^k\widetilde\gamma(s_1)}$, this is obvious, and if $\phi_{\widetilde\gamma(r_{k+1})}$ is on the left of $\phi_{S^k\widetilde\gamma(s_1)}$, then, since $\widetilde{\mathcal{F}}$ is non-singular, $\phi_{\widetilde\gamma(r_{k+1})}$ intersects all sides of the triangle formed by $\gamma_1, \gamma_2$, and $\gamma_3$, and hence intersects both  $S^k\widetilde\gamma$ and $S^{k+1}\widetilde\gamma$, which also implies the above. 
Also, $\phi_{\widetilde{\gamma}(s')}$ is clearly below $\phi_{S^k\widetilde\gamma(t')}$ relative to $\phi_{\widetilde\gamma(r_{k+1})}$.
With $u := \max\{s',t'\}$ we have that   
$\phi_{S^k\widetilde\gamma(u)}$ is below $\phi_{S^{k+1}\widetilde\gamma(t')}$ relative to $\phi_{\widetilde\gamma(r_{k+1})}$, and $\phi_{\widetilde{\gamma}(s')}$ is below $\phi_{S^k\widetilde\gamma(u)}$ relative to  $\phi_{\widetilde\gamma(r_{k+1})}$. 
By the transitivity of the relation 'below relative to  $\phi_{\widetilde\gamma(r_{k+1})}$',
we conclude that $\phi_{\widetilde{\gamma}(s')}$  is below $\phi_{S^{k+1}\widetilde\gamma(t')}$ relative to $\phi_{\widetilde\gamma(r_{k+1})}$. Moreover, by construction, $\phi_{\widetilde{\gamma}(s')}$  and $\phi_{S^{k+1}\widetilde\gamma(t')}$ are on the left of $\phi_{\widetilde\gamma(r_{k+1})}$, so the induction step is complete. 
This gives the desired contradiction, hence the lifts $S^k\widetilde\gamma$, $k=0,1,\cdots$, are pairwise equivalent. 

Let $Z\subset \widetilde{\dom(\mathcal{F})}$ be the union of the images of the lifts $S^k\widetilde\gamma$, $k=0, \cdots, m-1$. 
For each $t$ consider the unique point $x\in \widetilde{\dom(\mathcal{F})}$ with the property that 
\begin{itemize}
\item $x \in \phi_{\widetilde\gamma(t)}$, and 
\item there is no $y\in Z, y \neq x \text { such that } y \in \phi^{+}_x$.
\end{itemize}
Since $S^k\widetilde\gamma, k=0, \cdots m-1$ are pairwise equivalent, it is easy to see that this map defines a transverse line $t \mapsto \widetilde\gamma'(t)$ in $\widetilde{\dom(\mathcal{F})}$, equivalent to $\widetilde\gamma$. 
The deck transformation $S$ leaves the image of $Z = \cup_{k=0}^{m-1}S^k\widetilde\gamma$ in $\widetilde{\dom(\mathcal{F})}$ invariant. 
Since $S$ is orientation preserving, preserves $Z$ and the foliation $\widetilde{\mathcal{F}}$, $\widetilde\gamma'$ is also invariant under $S$. Hence $\widetilde\gamma'$ is a lift of a loop $\Gamma'$ on $\dom(\mathcal{F})$ with the desired properties. The proof of the lemma is complete. 
\end{proof}

\begin{proof}[Proof of Lemma \ref{Lemma10}]
By perturbing $\Gamma$ in its equivalence class we may assume that for any $\widetilde{x}\in \widetilde{\dom(\mathcal{F})}$ there are at most two lifts $\widetilde\gamma$ of $\Gamma$ that intersect in $\widetilde{x}$. 
Let $\Lambda: S^1 \to (\dom(\mathcal{F}), g_{\hyp})$ be a closed geodesic that is freely homotopic to $\Gamma$. We have $\selfi(\Lambda) = \selfi_{\dom(\mathcal{F})}([\Gamma])=k$. Let  $\widetilde{\lambda}:\R \to \widetilde{\dom(\mathcal{F})}$ be a lift of $\Lambda$, $T$ the shift for $\widetilde{\lambda}$. Since $\Lambda$ is a primitive closed geodesic in a hyperbolic surface, we have lifts $\widetilde\lambda_1, \cdots, \widetilde\lambda_{k}$ with  $\widetilde\lambda^{+}_i \in (\widetilde\lambda^{+}, \widetilde\lambda^-)\subset S_{\infty}$ and $\widetilde\gamma^{-}_i \in (\widetilde\lambda^-, \widetilde\lambda^{+})\subset S_{\infty}$  such that for every $l \in \Z$ and $i\neq j$,  $T^l\widetilde\lambda_i$ and $\widetilde\lambda_j$ are not translates of each other.   
Lifting a free homotopy from $\Lambda$ to $\Gamma$ to homotopies of the universal cover $\widetilde{\dom(\mathcal{F})}$ that extend $\widetilde\lambda_1, \cdots, \widetilde\lambda_{k}$, respectively,  we obtain lifts $\widetilde\gamma'_1, \cdots, \widetilde\gamma'_{k}$, with the same properties. By applying multiples of the  shifts $T'_1, \cdots, T'_n$ for $\widetilde\gamma'_1, \cdots, \widetilde\gamma'_{k}$, respectively, and applying multiples of $T$ to them, we may additionally assume that $\widetilde\gamma'_i|_{[0,1)}$ and $\widetilde\gamma|_{[0,1)}$ intersect, for each $i\in \{1, \cdots, k\}$.  Choose $s_i, s'_i \in \R$ with $\widetilde\gamma(s_i) = \widetilde\gamma'_i(s'_i)$. 
Let then $S_i$, $i\in \{1, \cdots, k\}$, be the deck transformation with $S_i \widetilde\gamma = \widetilde\gamma'_i$.
With the following choice of lifts $\widetilde\gamma_1, \cdots, \widetilde\gamma_k$ of $\gamma$ and the parameters $t_1, \cdots, t_n$, $t'_1 \cdots, t'_n$, the properties of the lemma are satisfied. 
If $s_i < s'_i$ set $\widetilde\gamma_i := \widetilde\gamma'_i$, and $t_i := s_i$, $t'_i:=s'_i$. If $s_i > s'_i$, set $\widetilde\gamma_i := S_i^{-1}\widetilde\gamma$, and $t_i := s'_i$, $t'_i := s_i$. 
\end{proof}


\section{Geometric self-intersections, growth of periodic points and entropy}\label{sec:proof_entropy}

In this section, let $M$ be an oriented closed surface, $f: M\to M$ a homeomorphism isotopic to the identity, $I$ a maximal identity isotopy for $f$, and $\mathcal{F}$ an oriented foliation transverse to $I$. As in the introduction, denote by $N_{\per}(f,n)$ the number of  $n$-periodic points of $f$ of period $\leq n$, let $\Per^{\infty}(f) := \limsup_{n\to +\infty} \log(N_{\per}(f,n))/n$, and let $h_{\topo}(f)$ denote the topological entropy of $f$.  Let $\Gamma$ be an $\mathcal{F}$-transverse loop that is not freely homotopic to a multiple of a simple loop in $\dom(I)$, say $[\Gamma]_{\widehat{\pi}(\dom(I))} = m \alpha$, where $\alpha$ is a non-simple primitive class in $\dom(I)$ and $m\in \N$. 
In this section we prove the following 

\begin{prop}\label{prop:forcing1}
If $\Gamma$ is linearly admissible of order $q$,  then 
\begin{align}
\Per^{\infty}(f) \geq \frac{m}{q} \max\left\{\frac{\log(\selfi_{\dom(I)}(\alpha) + 1)}{16}, \frac{\log 2}{2} \right\}.
\end{align}
\end{prop}

\begin{prop}\label{prop:forcing2}
If  $\Gamma$ is linearly admissible of order $q$, then  
\begin{align}
h_{\topo}(f) \geq \frac{m}{q} \max\left\{\frac{\log(\selfi_{\dom(I)}(\alpha) + 1)}{16}, \frac{\log 2}{2} \right\}.
\end{align}
\end{prop}

\begin{rem}
Note that the lower bound $m\log(2) / 2q$ on the topological entropy of $f$ is larger then the bound obtained in Proposition 38 in \cite{LeCalvezTal1} and the bound obtained in Theorem $N$ in \cite{LeCalvezTal2}. 
\end{rem}

Theorem \ref{intro:thm1} follows from Propositions \ref{prop:forcing1} and \ref{prop:forcing2}.
 
\begin{proof}[Proof of Theorem \ref{intro:thm1}]
If $x$ is a $q$-periodic point of $f$ then $I^q(x)$ is homotopic with fixed endpoints to a $\mathcal{F}$-transverse loop $\Gamma$ in $\dom(I)$. By assumption, $[\Gamma]_{\widehat{\pi}(\dom(I))} = m\alpha$. The  natural lift $\gamma$ of $\Gamma$ satisfies that $\gamma|_{[0,k]}$ is admissible of order $kq$, for every $k\in \N$.   In particular, $\Gamma$ is linearly admissible of order $q$. Hence the lower bounds on $Per^{\infty}(f)$ and $h_{\topo}(f)$ follow from Propositions \ref{prop:forcing1} and \ref{prop:forcing2}. 
\end{proof}

Before we proceed with the proofs of Propositions \ref{prop:forcing1} and \ref{prop:forcing2}, we give a proof of Theorem \ref{intro:thm2} from the introduction.  

\begin{proof}[Proof of \ref{intro:thm2}]
For any identity isotopy $I$ for $f$ there is by \cite[corollary 1.3]{Beguin2016} a maximal identity isotopy $\widehat{I}$ for $f$ such that $\fix(I) \subset \fix(\widehat{I})$ and such that the loop $\widehat{I}(y)$ is homotopic to $I(y)$ in $\dom(I)$ for any $y\in \dom(I) \cap \fix(f)$. In particular $[\widehat{I}(x)] = \alpha$ in $\widehat{\pi}(M)$ and so $\widehat{I}(x)$ defines a free homotopy class $\widehat{\alpha}$ of loops in $\dom(\widehat{I})$ whose pushforward by the inclusion  $\dom(\widehat{I}) \hookrightarrow M$ is $\alpha$. The class $\widehat{\alpha}$ is primitive and $\selfi_{M}(\widehat{\alpha}) \geq \selfi_{M}(\alpha)$. So the statement follows from Theorem \ref{intro:thm1}.  
\end{proof}
		
The main steps in the proofs of Propositions \ref{prop:forcing1} and \ref{prop:forcing2} are adaptions of the main steps in the proofs of Propositions 31 and 38 in \cite{LeCalvezTal1}. The notations and arguments are kept closely to those in the corresponding proofs in \cite{LeCalvezTal1}.

\begin{proof}[Proof of Proposition \ref{prop:forcing1}]
By modifying $\Gamma$ in its equivalence class if necessary we may assume by Lemma \ref{Lemma9} that there is a  $\mathcal{F}$-transverse loop $\Gamma'$, primitive in $\dom(I)$, and $m\in \N$ such that $\Gamma = (\Gamma')^m$. 
We first give the proof of Proposition \ref{prop:forcing1} for $m=1$, i.e. $\Gamma = \Gamma'$.  It will be easy to adapt that proof for $m>1$. 
Let $k= \selfi_{\dom(I)}([\Gamma])$. By Lemma \ref{Lemma10} there are, after modifying $\Gamma$ further in its equivalence class, 
pairwise distinct points $x_1, \cdots, x_k \in \dom(\mathcal{F})$, pairwise distinct parameters $t_1, \cdots, t_k$, $t'_1, \cdots, t'_k \in [0,1)$ with  $t_i < t'_i$, and lifts $\widetilde\gamma_1, \cdots, \widetilde\gamma_k$ of $\Gamma$, pairwise not translates of each other,  such that, for all $i=1, \cdots, k$, $x_i = \Gamma(t_i) = \Gamma(t'_i)$ and $\widetilde\gamma$ and $\widetilde\gamma_i$ intersect $\widetilde{\mathcal{F}}$-transversally (positively or negatively) in $\widetilde\gamma(t_i) = \widetilde\gamma_i(t'_i)$. 
  \color{black}

By Lemma \ref{Lemma6} there is $\overline{x} \in \{x_1, \cdots, x_k\}$  and $\overline{t}, \underline{t} \in [0,1)$ with $\overline{x}= \Gamma(\overline{t}) = \Gamma(\underline{t})$, such that for any lifts $\widetilde\gamma$ and $\widetilde\gamma'$ of $\Gamma$ that intersects $\mathcal{F}$-transversally and positively, resp. negatively, $\widetilde\gamma$ does not intersect the leaves $\phi_{\widetilde\gamma'(\overline{t} +k)}$, resp. $\phi_{\widetilde\gamma'(\underline{t} +k )}$, for all $k\in \Z$. We may assume that $\overline{x} = x_1$, i.e. $\{t_1, t'_1\} = \{\overline{t}, \underline{t}\}$. 
For each $i \in \{1,\cdots k\}$ and $l\in \N$ let $\gamma_{i,l}$ be the transverse path $\gamma_{[0,t_i]}\gamma_{[t'_i,l]}$ in $\dom(\mathcal{F})$, furthermore let $\gamma_{0,l} := \gamma_{[0,l]}$. Let $n\geq 1$. 
For any $\rho= (\rho_1, \cdots, \rho_n)$ with $\rho_i \in \{0, \cdots, k\}$, and $l\geq 1$, 
consider the path $\gamma_{\rho,l} := \gamma_{\rho_1,l}\gamma_{\rho_2,l} \cdots \gamma_{\rho_n,l}$. 
The path defines a transverse loop $\Gamma_{\rho,l}$.

\begin{claim}\label{claim:admissible}
\begin{enumerate}
\item For each $\rho= (\rho_1, \cdots, \rho_n)$ with $\rho_i \in \{0, \cdots, k\}$, $i = 1, \cdots, n$  and  $l\geq 2$, $\Gamma_{\rho,l}$ is linearly admissible of order $lnq$. 
\item For each $\rho = (\rho_1, \cdots, \rho_n)$ with $\rho_i \in \{0,1\}$, $i=1, \cdots n$,  
$\Gamma_{\rho,1}$ is linearly admissible of order $nq$.
\end{enumerate}
 \end{claim} 

\begin{proof}
Let $l\geq 1$. Let $\rho= (\rho_1, \cdots, \rho_{n})$, with $\rho_i \in \{0, \cdots, k\}$ if $l\geq 2$ and $\rho_i \in \{0, 1\}$ if $l=1$. 
Consider the paths $\gamma|_{[0,l(pn +1)]}$, $p \in \N$. Since $\Gamma$ is linearly admissible of order $q$, there is a sequence $s_p= s_p(n,l)$, $p \in \N$, with $s_p \to + \infty$, $\limsup_{p \to \infty} (lpn)/s_p \geq 1/q$, and such that 
$\gamma|_{[0,l(pn +1)]}$ is admissible of order $\leq s_p$. 
Fix now $p \in \N$ and let $\hat{\gamma} := \gamma|_{[0,l]}(\gamma_{\rho,l})^p$. 
We will show that $\hat{\gamma}$ is admissible of order $\leq s_p$.
From this it follows that the same holds for $(\gamma_{\rho,l})^p$. So, since $p$ was arbitrary, the claim follows. 
 
Let $\hat{\rho} = (\hat{\rho}_1, \cdots , \hat{\rho}_{pn}) = (\rho_1,\cdots, \rho_n, \rho_1, \cdots, \rho_n, \rho_1, \cdots, \cdots, \rho_n)$ the $pn$-tuple that is build from $p$ repetitions of $\rho$. 
Consider the transverse paths $\hat{\gamma}_j$, $j \in \{0, \cdots, pn\}$ given by  
$\hat{\gamma}_0 :=  \gamma|_{[0,l]}\gamma|_{[0,lpn]}$,  $\hat{\gamma}_j := \gamma|_{[0,l]}\gamma_{\hat{\rho}_1,l} \cdots \gamma_{\hat{\rho}_j,l}\gamma|_{[lj,lpn]}$,  for $j\in\{1,\cdots, pn-1\}$, and $\hat{\gamma}_{pn} := \hat{\gamma}$.
If  $j\in \{0,\cdots,  np-1\}$ we will say that $\hat{\gamma}_j$
is \textit{reducible} if  it has an $\mathcal{F}$-transverse self-intersection at $\gamma|_{[lj,lpn]}(lj + t_{\hat{\rho}_{j+1}}) = \gamma|_{[lj,lpn]}(lj + t'_{\hat{\rho}_{j+1}})$. 
For $j\in\{0,\cdots, pn\}$ consider the statement 
\begin{align*}
(R_j): \text{ The path } \hat{\gamma}^j_{\rho} \text{ is admissible of order }\leq s_p, \text{and it is reducible if } j< pn.
\end{align*}
We want to show $(R_{pn})$, and we prove it by induction. 
Note that $(R_0)$ holds. Indeed, $\hat{\gamma}_0 = \gamma|_{[0,l]}\gamma|_{[0,lpn]} =  \gamma|_{[0,l(pn+1)]}$ is admissible or order $\leq s_p$ and $\gamma|_{[0,l(pn+1)]}$ has an $\mathcal{F}$-transverse self-intersection at $\gamma(l+ t_{\hat{\rho}_1}) = \gamma(l+ t'_{\hat{\rho}_1})$ by assumption. 
Assume now that $(R_j)$ holds for some  $j\in\{0,\cdots, pn-1\}$. 
Applying Proposition \ref{Proposition23} at $\gamma|_{[lj,lpn]}(lj + t_{\hat{\rho}_{j+1}}) = \gamma|_{[lj,lpn]}(lj + t'_{\hat{\rho}_{j+1}})$ yields that $\hat{\gamma}_{j+1}$ is admissible of order $\leq s_p$. 
We are left to show that $\hat{\gamma}_{j+1}$ is reducible if $j \in\{0,\cdots, pn-2\}$.
So let $s_1 := l(j+1) + t_{\hat{\rho}_{j+2}}$ and $s_2:= l(j+1) + t'_{\rho_{j+2}}$. 
Take two lifts $\tau_1$ and $\tau_2$ of $\hat{\gamma}_{j+1}$ that intersect at a lift $\widetilde{y}$ of $y:=\gamma|_{[l(j+1), lpn]}(s_1) = \gamma|_{[l(j+1), lpn]}(s_2)$. 

For $l\geq 2$, consider the subpaths
$\overline{\tau}_1 = \tau_1|_{[l(j+1)-1,l(j+1)+2]}$ and $\overline{\tau}_2 = \tau_2|_{[l(j+1)-1, l(j+1)+2]}$ of $\tau_1$ and $\tau_2$. It is enough to show that these subpaths intersect $\widetilde{\mathcal{F}}$-transversally at $\widetilde{y}$. 
Note that since $l\geq 2$, $\overline{\tau}_1$ and $\overline{\tau}_2$ are itself subpaths of lifts $\widetilde\gamma_1: \R \to \widetilde{\dom(\mathcal{F})}$ and $\widetilde\gamma_2: \R \to \widetilde{\dom(\mathcal{F})}$ of $\gamma$ that intersect $\widetilde{\mathcal{F}}$-transversally at $\widetilde{y}$. 
Therefore, since $|s_i - (l(j+1)-1)|\geq 1$ and $|s_i - (l(j+1) + 2)|\geq 1$ for $i=1,2$ , we conclude by Lemma \ref{Lemma7} $(1)$ that already $\overline{\tau}_1$ and $\overline{\tau}_2$ intersect $\widetilde{\mathcal{F}}$-transversally. 

For $l=1$, and $\rho_i \in \{0, 1\}$ for all $i \in \{1, \cdots, n\}$, consider the subpaths $\overline{\tau}_1 = \tau_1|_{[j + t'_1, j+1 + t'_1]}$ of $\tau_1$ and $\overline{\tau}_2 = \tau_2|_{[j+1 +t_1, j+2+t_1]}$, which are also subpaths of lifts $\widetilde\gamma_1$, $\widetilde\gamma_2$ of $\gamma$ that intersect $\widetilde{\mathcal{F}}$-transversally at $\widetilde{y}$. 
It follows by Lemma \ref{Lemma6} that already the subpaths $\overline{\tau}_1$ and $\overline{\tau}_2$ intersect $\widetilde{\mathcal{F}}$-transversally at $\tau_1|_{[j + t'_1, j+1 + t'_1]}(j+1+t_1) = \tau_2|_{[j+1 +t_1, j+2+t_1]}(j+1 +t'_1)$, positively if $t'_1= \overline{t}$ and negatively if $t'_1= \underline{t}$. 

In both cases considered above, $(R_{pn})$ follows now by induction, i.e. $\hat{\gamma}$ is admissible of order $\leq s_p$.
\end{proof}

Let $z = \Gamma(0) \in \dom(\mathcal{F})$. Choose a lift $\widetilde{z}$ of $z$ in $\widetilde{\dom(\mathcal{F})}$. 
Let $l\geq 1$. For $J\subset \{0, \cdots, k\}$ consider the family of paths $\gamma_{i,l}$, $i \in J$, where $\gamma_{i,l}$ is defined as above. Set $t_{\min}  := \min\{t_i, i \in J \setminus \{0\} \, \}$ and $t'_{\max} := \max\{t'_i, i \in J \setminus \{0\}\, \}$. 
For all $i\in J$, and suitable $a_i< b_i$,  let  $\widetilde\gamma_{i,l}:[a_{i}, b_{i}] \to \widetilde{\dom(\mathcal{F})}$ be the lift of the path $\gamma_{i,l}\gamma_{[0,t_{\min}]}$ that starts at $\widetilde{z}$. 
Furthermore denote $\gamma'_{i,l} := \gamma|_{[0,l-1 + t_i]}\gamma|_{[l-1 +t'_i, l]}$ and consider the lifts $\widetilde\gamma'_{i,l}:[a'_{i}, b'_{i}] \to \widetilde{\dom(\mathcal{F})}$ of the paths $\gamma_{[t'_{max},1]}\gamma'_{i,l}$, $i \in J$, that end at $\widetilde{z}$. 
We say that the paths $\gamma_{i,l}$, $i\in J$, \textit{spread}  
if for all $i\neq j \in J$, 
$\phi_{\widetilde\gamma_{i,l}}(b_{i})$ is above or below $\phi_{\widetilde\gamma_{j,l}}(b_{j})$ relative to $\phi_{\widetilde{z}}$, and 
$\phi_{\widetilde\gamma'_{i,l}}(a'_{i})$ is above or below $\phi_{\widetilde\gamma'_{j,l}}(a'_{j})$ relative to  $\phi_{\widetilde{z}}$.
We get two total orders $\prec$ and $\prec^*$ on $J$, by saying that $i\prec j$ if $\phi_{\widetilde\gamma_{i,l}}(b_i)$ is below $\phi_{\widetilde\gamma_{j,l}}(b_j)$ relative to $\phi_{\widetilde{z}}$, and by saying that $i \prec^* j$ if  $\phi_{\widetilde\gamma'_{i,l}}(a'_{i})$ is below $\phi_{\widetilde\gamma'_{j,l}}(a'_{j})$ relative to  $\phi_{\widetilde{z}}$.
For each $i\in J$ we define $i^*\in J$ by requiring that  $\#\{j\in J\, |\, j \prec i \} =  \#\{j\in J\, |\, i^*\prec^* j^* \}$. Note that if $j\prec i$, then $i^* \prec^* j^*$. 
We say that a \textit{quasi-palindromic word  of length $2n$ for $J$} is a $2n$-tuple $\hat{\rho} = (\rho_1, \cdots, \rho_{2n})$ with $\rho_j \in J$ and 
$\rho_{n+j} = \rho^*_{n-j+1}$ for all $j \in \{1, \cdots, n\}$. In the following we will only consider quasi-palindromic words $\rho$ with $\rho_0= 1$. For such a quasi-palindromic word $\rho$ of length $2n$ for $J$ we consider the lift $\widetilde\gamma_{\rho} = \widetilde\gamma_{\rho}^-\widetilde\gamma_{\rho}^+$ of $\prod_{1\leq j\leq 2n} \gamma_{\rho_j,l} = \gamma_{[0,l]} \prod_{2\leq j\leq 2n} \gamma_{\rho_j,l}=$ \linebreak $\gamma_{[0,1]}\prod_{1\leq j \leq n} \gamma'_{\rho_j,l} \gamma_{[1,l]}\prod_{n+1 \leq j \leq 2n} \gamma_{\rho_j,l}$, where $\widetilde\gamma_{\rho}^-$ is the lift of $\gamma_{[0,1]}\prod_{1\leq j\leq n}\gamma'_{\rho_j,l}\gamma_{[1,l]}$ ending at $\widetilde{z}$ and $\widetilde\gamma_{\rho}^+$ is the lift of $\prod_{n+1\leq j\leq 2n} \gamma_{\rho_j,l}$ starting at $\widetilde{z}$. Let $T_{\rho}$ be the deck transformation that sends the starting point of $\widetilde\gamma_{\rho}$ to its ending point. Let  $\widetilde\gamma^{\infty}_{\rho} := \prod_{k\in \Z} T^k_{\rho}(\widetilde\gamma_{\rho})$. It is a lift of $\Gamma_{\rho,l}$. We will also consider the path  $\widetilde\gamma^2_{\rho} = \widetilde\gamma_{\rho} T_{\rho}( \widetilde\gamma_{\rho})$.

\begin{claim}\label{struwwelpeter}
Let $l\geq 1$ and let $\gamma_{i,l}, i \in J$, be a family of paths as above that spread. If $\rho$ and $\rho'$ are two distinct quasi-palindromic words of the same length with $\rho_1 = \rho'_1 = 0$, then the paths $\widetilde\gamma_{\rho}$ and $\widetilde\gamma_{\rho'}$  intersect $\widetilde{\mathcal{F}}$-transversally at $\widetilde{z}$. 
\end{claim}

\begin{proof}
If $\rho \neq \rho'$, there is $j$, $1\leq j \leq n$ with $\rho_{n+i} = \rho'_{n+i}$ for $0 < i < j$ and $\rho_{n+j} \neq \rho'_{n+j}$. W.l.o.g. $\rho_{n+j} \prec \rho'_{n+j}$, which implies that $\rho'_{n-j+1}\prec^* \rho_{n-j+1}$. 
The leaf of $\widetilde{\mathcal{F}}$ through the endpoint of the subpath of  $\widetilde\gamma_{\rho}^+$ that lifts $\prod_{n+1 \leq i \leq n+j} \gamma_{\rho_i,l}$ is below the leaf through the endpoint of the subpath of  $\widetilde\gamma_{\rho'}^+$ that lifts $\prod_{n+1 \leq i \leq n+j} \gamma_{\rho'_i,l}$ relative to $\phi_{\widetilde{z}}$. Dually, the 
leaf of $\widetilde{\mathcal{F}}$ through the starting point of the subpath of  $\widetilde\gamma_{\rho}^-$ that lifts $\prod_{n-j +1 \leq i \leq n} \gamma'_{\rho_i,l}\gamma_{[1,l]}$ is above the leaf through the starting point of the subpath of $\widetilde\gamma^{-}_{\rho'}$ that lifts $\prod_{n-j +1 \leq i \leq n} \gamma'_{\rho'_i,l}\gamma_{[1,l]}$ relative to $\phi_{\widetilde{z}}$. 
The claim follows.
\end{proof}

\begin{claim}\label{few_equivalent}
Let $\gamma_{i,l}, i \in J$, be a family of paths that spread. Then there is a constant $L > 0$  such that for a given a quasi-palindromic word $\rho$  
of length $2n$ with $\rho_1 =0$, there are at most $Ln^{2}$ different such quasi-palindromic words $\rho'$  such that $\Gamma_{\rho}$ and $\Gamma_{\rho'}$ are equivalent.  
\end{claim}

\begin{proof}
The proof is the same as the proof of Lemma 35 in \cite{LeCalvezTal1}. 
There is a constant $L'$ such that the number of deck transformations $S$ such that $\widetilde\gamma_{i,l}$ and $S(\widetilde\gamma_{j,l})$ intersect is $\leq L'$ for all $i, j\in J$. 
Hence there are at most $8 L'n^2$ deck transformations $S$ such that $\widetilde\gamma_{\rho}$ and $S(\widetilde\gamma^2_{\rho})$ have an $\widetilde{\mathcal{F}}$-transverse intersection.
If $\rho$ and $\rho'$ are two quasi-palindromic words of length $2n$ with $\rho_1 = \rho'_1 = 0$  such that $\Gamma_{\rho}$ and $\Gamma_{\rho'}$ are equivalent, then there is a deck transformation $S_{\rho'}$ such that $\widetilde\gamma_{\rho'}$ is equivalent to a subpath of $S_{\rho'}(\widetilde\gamma^2_{\rho})$. Moreover, if $\rho' \neq \rho''$, it follows that $S_{\rho'} \neq S_{\rho''}$ by Claim \ref{struwwelpeter}. 
Also by Claim \ref{struwwelpeter}, $\widetilde\gamma_{\rho}$ and $S_{\rho'}(\widetilde\gamma^2_{\rho})$ intersect $\widetilde{\mathcal{F}}$- transversally. This proves the claim. 
\end{proof}

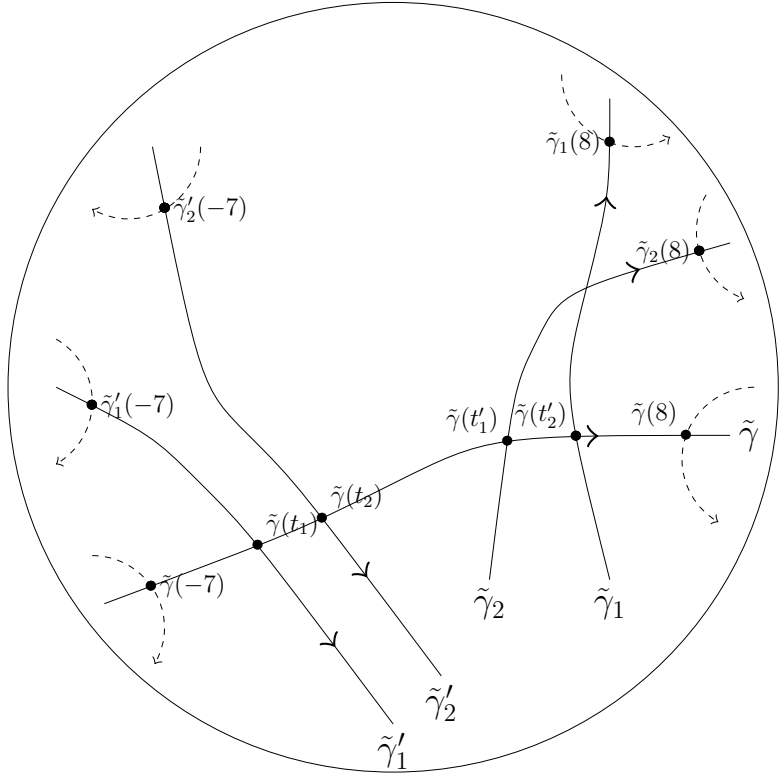
\begin{figure}
    \centering
		\scalebox{0.8}{
    \begin{tikzpicture}[scale=0.8]
        \coordinate (O) at (0,0);
        
        \draw (O) circle (8);
        
        \path[name path=lift1]  (-6,-4.5) .. controls (-2,-3) .. (0,-2) .. controls (2,-1) .. (7,-1);
        \path[name path=lift2]  (-7,0) .. controls (-5,-1) .. (-4,-2) .. controls (-3,-3) .. (0,-7);
        \path[name path=lift3]  (-5,5)  .. controls (-4,0) .. (-3,-1) .. controls (-2,-2) .. (1,-6);
        \path[name path=lift4]  (2,-4)  .. controls (2.5,0) .. (3,1) .. controls (3.5,2) .. (7,3);
        \path[name path=lift5]  (4.5,-4)  .. controls (3.5,0) .. (4,2) .. controls (4.5,4) .. (4.5,6);

\foreach \i in {1,2,3,4,5} {
            \draw[use path=lift\i, lift];
        }        
        
  \path[name path=leaf1] (-4,5) arc (0:-120:1.5);
       \path[name path=leaf2]  (-7,1) arc (60:-60:1.5);
       \path[name path=leaf3]  (-6.25,-3.5) arc (90:-30:1.5);
        \path[name path=leaf4]  (3.5,6.5) arc (180:300:1.5);
        \path[name path=leaf5]  (6.5,4) arc (150:250:1.5);
        \path[name path=leaf6]  (7.5,0) arc (90:240:1.5); 
        
\foreach \i in {1,2,3,4,5,6}   {
\draw[use path=leaf\i, leaf];
}   
        
\foreach \i in {2,3,4,5} {            
                    \path[name intersections={of={lift1} and {lift\i}, by={I1\i}}];
                    \node at (I1\i) {\textbullet};
                      
        }        

\foreach \i in {2,3,4,5} {            
                    \path[name intersections={of={lift1} and {lift\i}, by={I1\i}}];
                    \node at (I1\i) {\textbullet};
                      
        }        
        
\foreach\i in {1,2,3,4,5} {
\foreach \j in {1,2,3,4,5,6}
{  \path[name intersections={of={lift\i} and {leaf\j}, by={J\i\j}}];
        \node at (J\i\j) {\textbullet};
        }
        }
        

         \draw (I12)[anchor=south west] node {$\tilde\gamma(t_1)$};   
         \draw (I13)[anchor=south west] node {$\tilde\gamma(t_2)$};        
         \draw (I14)[anchor=south east] node {$\tilde\gamma(t'_1)$};        
         \draw (I15)[anchor=south east] node {$\tilde\gamma(t'_2)$};        
 
         \draw (J13)[anchor=west] node {$\tilde\gamma(-7)$};
         \draw (J22)[anchor=west] node {$\tilde\gamma'_1(-7)$};
         \draw (J31)[anchor=west] node {$\tilde\gamma'_2(-7)$};
         \draw (J16)[anchor=south east] node {$\tilde\gamma(8)$};
         \draw (J54)[anchor=east] node {$\tilde\gamma_1(8)$};
         \draw (J45)[anchor=east] node {$\tilde\gamma_2(8)$};

         \draw (7,-1)[anchor=west] node {\Large $\tilde\gamma$};
         \draw (0,-7)[anchor=north] node {\Large $\tilde\gamma'_1$};
         \draw (1,-6)[anchor=north] node {\Large $\tilde\gamma'_2$};
         \draw (2,-4)[anchor=north] node {\Large $\tilde\gamma_2$};
         \draw (4.5,-4)[anchor=north] node {\Large $\tilde\gamma_1$};

    \end{tikzpicture}
		}
    \caption{Some lifts of the loop $\Gamma$ to $\widetilde{\mathrm{dom}(\mathcal{F})}$, with $t_1 < t'_1, t_2 < t'_2 \in [0,1)$, and $\Gamma(t_1) = \Gamma(t'_1)$, $\Gamma(t_2)= \Gamma(t'_2)$ in the situation that the family $\gamma_{0,8}, \gamma_{1,8}, \gamma_{2,8}$ spreads, with $1\prec 2\prec 0$ and $1^* = 0$, $2^* = 1$, $0^*=2$.}
    \label{fig:spread}
\end{figure}

\begin{claim}\label{spread1}\begin{enumerate}
\item The family $\mathcal{P}_1$ of paths $\gamma_{i,8}$, $i \in \{0, 1, \cdots, k\}$, spread.
\item The pair $\mathcal{P}_2$ of paths $\gamma_{0,1}$, $\gamma_{1,1}$ spread.
\end{enumerate} 
\end{claim}

\begin{proof}
We first show $(2)$. Let  $\widetilde\gamma_{0,1}= \widetilde\gamma_0|_{[0,1+t_1]}$ and $\widetilde\gamma_{1,1} = \widetilde\gamma_0|_{[0,t_1]}\widetilde\gamma_1|_{[t'_1,1+t_1]}$ be the lifts of $\gamma_{0,1}\gamma|_{[0,t_1]}$ and $\gamma_{1,1}\gamma|_{[0,t_1]}$, respectively  that start at the point $\widetilde{z}$.  Here $\widetilde\gamma_0$ and $\widetilde\gamma_1$ are suitable lifts of $\gamma$. 
Note that $\widetilde\gamma_0$ and $\widetilde\gamma_1$ have an $\widetilde{\mathcal{F}}$-transverse intersection. So by Lemma \ref{Lemma6} and since $\{t_1, t'_1\} = \{\overline{t}, \underline{t}\}$ we have that
$\phi_{\widetilde\gamma_0(t'_1)}$ does not intersect $\widetilde\gamma_{1,1}$ and  $\phi_{\widetilde\gamma_1(1+t_1)}$ does not intersect $\widetilde\gamma_{0,1}$. It is then easy to see that $\phi_{\widetilde\gamma_0(1+ t_1)}$ is above or below $\phi_{\widetilde\gamma_1(1+t_1)}$ relative to $\phi_{\widetilde{z}}$. Similarly one sees that the leaf through the starting point of $\widetilde\gamma'_{0,1}$ is above or below the leaf through the starting point of $\widetilde\gamma'_{1,1}$ relative to $\phi_{\widetilde{z}}$,  where these paths lift $\gamma|_{[t_1',1]}\gamma'_{0,1}$ and $\gamma|_{[t_1',1]}\gamma'_{1,1}$ with a common endpoint. 

We now turn to (1). Consider first a pair $\gamma_{0,8} = \gamma_{[0,8]}$ and $\gamma_{i,8}= \gamma_{[0,t_i]}\gamma|_{[t'_i,8]}$, for some $i\in \{1,\cdots, k\}$. The same argument as in the proof of (2) shows that such a pair of paths spread. 
Let now $i, j \in \{1, \cdots, k\}$ with $i\neq j$, and consider $\gamma_{i,8}=\gamma_{[0,t_i]}\gamma|_{[t'_i,8]}$ and $\gamma_{j,8}= \gamma_{[0,t_j]}\gamma|_{[t'_j,8]}$, and their lifts $\widetilde\gamma_{i,8}= \widetilde\gamma|_{[0,t_i]}\widetilde\gamma_i|_{[t_i',8]}$ and 
$\widetilde\gamma_{j,8}= \widetilde\gamma|_{[0,t_j]}\widetilde\gamma_j|_{[t_j',8]}$ that start at $\widetilde{z}$. Here $\widetilde\gamma$, $\widetilde\gamma_i$ and $\widetilde\gamma_j$ are suitable lifts of $\gamma$. 
We will show that $\phi_{\widetilde\gamma_i(8)}$ is below or above $\phi_{\widetilde\gamma_j(8)}$ relative to $\phi_{\widetilde{z}}$. 
  Since the lifts $\widetilde\gamma_{i,8}$ and $\widetilde\gamma_{j,8}$ start at the same point it is sufficient to show that 
neither $\widetilde\gamma_{i,8}$ is equivalent to a subpath of $\widetilde\gamma_{j,8}$ nor $\widetilde\gamma_{j,8}$ is equivalent to a subpath of $\widetilde\gamma_{i,8}$. Assume that $\widetilde\gamma_{i,8}$ is equivalent to a subpath of $\widetilde\gamma_{j,8}$.
Then in particular $\widetilde\gamma_i|_{[t'_i+1,8]}$ is equivalent to a subpath of $\widetilde\gamma_{j,8}$. 
Since $\widetilde\gamma_i$ and $\widetilde\gamma$ intersect $\mathcal{F}$-transversally in $\widetilde\gamma_i(t'_i) = \widetilde\gamma(t_i)$ it follows by Lemma \ref{Lemma7} (1) that no subpath of $\widetilde\gamma_i|_{[t'_i+1,8]}$ is equivalent to a subpath of $\widetilde\gamma|_{[0,t_j]}$. 
Hence $\widetilde\gamma_i|_{[t'_i+1,8]}$ is equivalent to a subpath of $\widetilde\gamma_j|_{[t'_j,8]}$.  
Since $\widetilde\gamma_i$ and $\widetilde\gamma_j$ are no translates of each other, and $8-(t'_i +1) > 6$ we obtain a contradiction by Lemma \ref{Lemma7} (3). Analogously $\widetilde\gamma_{j,8}$ is not equivalent to a subpath of $\widetilde\gamma_{i,8}$. 
Similarly one can show, for $i, j\in \{0,\cdots, k\}$ with $i\neq j$, that the leaf through the starting point of $\widetilde\gamma'_{i,8}$ is above or below the leaf through the starting point of $\widetilde\gamma'_{j,8}$ relative to $\phi_{\widetilde{z}}$, where these paths are lifts of $\gamma'_{i,8}$ and $\gamma'_{j,8}$, respectively, with common endpoint at $\widetilde{z}$. 
\end{proof}

By Claim \ref{few_equivalent} applied to the family $\mathcal{P}_1$ and by Claim \ref{claim:admissible} $(1)$, there is a constant $L>0$ such that for every $n$ there are at least $(k +1)^{n-1}/ Ln^2$ many different equivalence classes of linearly admissible loops of order $16nq$, hence, by Proposition \ref{Proposition26}, there are at least  
${(k +1)^{n-1}/ Ln^2}$ many fixed points of $f^{16nq}$, and so 
$$\Per^{\infty}(f) \geq \limsup_{n\to +\infty} \frac{1}{16nq}\log(k +1)^{n-1}/ Ln^2 = \log(k +1)/16q.$$ 
Similarly, by Claim \ref{few_equivalent} applied to the pair of paths $\mathcal{P}_2$, we obtain that 
${\Per^{\infty}(f) \geq (\log 2 )/2q}$. 

The proof is easy to adapt to the general case $m>1$. As noted above, we can assume that $\Gamma = (\Gamma')^m$ and $\Gamma'$ satisfies the properties of Lemma \ref{Lemma10}. We can carry out the proof with $\Gamma'$ instead of $\Gamma$, where we consider $n$ that are multiple of $m$. Along the proof of Claim \ref{claim:admissible} one obtains that the loops $\Gamma'_{\rho,l}$ are even linearly admissible of order $\leq lnq/m$.
So one obtains at least $(k +1)^{n-1}/ Ln^2$ many different equivalence classes of linearly admissible loops of order $16nq/m$, and also at least $2^{n-1}/ Ln^2$ many different equivalence classes of linearly admissible loops of order $2nq/m$, and so the lower bounds on $\Per^{\infty}(f)$ follow as above.  
\end{proof}

\begin{proof}[Proof of Proposition \ref{prop:forcing2}]
As before we first assume that $\Gamma$ is primitive in $\dom(I)$, i.e. $m=1$.  
The proof in \cite{LeCalvezTal1} carries over with slight, technical changes. Hence we will indicate these changes and refer for the full proof to \cite{LeCalvezTal1}.
The outline of the proof in \cite{LeCalvezTal1} can be described as follows. One considers a pair of paths that spread with some suitable $l\geq 1$ and such that for every quasi-palindromic word $\rho$ of length $2n$ (in the situation in \cite{LeCalvezTal1} it is also a palindromic word), there is a linearly admissible loop $\Gamma_{\rho}$ of order $2lnq$. Every such loop gives rise to a fixed point of $f^{2lnq}$.  Furthermore there are at least $2^n / Ln^2$ equivalence classes of such loops $\Gamma_{\rho}$. 
Consider the one-point compactification $\dom(I)\cup \{\infty\}$ of $\dom(I)$ and the extension $\hat{f}$  of $f|_{\dom(I)}$ that fixes $\{\infty\}$. For every $p\in \N$ a family of suitable coverings $\mathcal{V}^p$ of $\dom(I)\cup \{\infty\}$ is constructed (the coverings differ by the size of their neighbourhoods at $\infty$) such that for each element $V= \cap_{0\leq k \leq 2lnq}\hat{f}^{-k}(V^k)$, $V^k \in \mathcal{V}^p$, of the covering $\bigvee_{0\leq k \leq 2lnq} \hat{f}^{-k}(\mathcal{V}^p)$
there are at most $M^{lnq/p}$ equivalence classes of those loops $\Gamma_{\rho}$ that are associated to some fixed point of $\hat{f}^{2lnq}$ that lies in $V$. Here $M$ is a suitable constant independent of $p$.  Hence the minimal cardinality $N_{2lnq}(\hat{f}, \mathcal{V}^p)$ of a subcover of $\bigvee_{0\leq k \leq 2lnq} f^{-k}(\mathcal{V}^p)$ is at least $2^n/(Ln^2M^{lnq/p})$. A lower bound on the entropy can then be directly obtained, using the classical definition of topological entropy. 
The main idea in the proof is to show that orbits that stay close to $\infty$ for "many" iterates, only contribute to rather "few" non-equivalent paths.  See \cite{LeCalvezTal1} for the proof. 

One can adapt the proof to other families of paths that spread.
We consider the two families $\mathcal{P}_1 = \{ \gamma_{0,8}, \cdots ,\gamma_{k,8}\}$ and $\mathcal{P}_2 = \{\gamma_{0,1}, \gamma_{1,1}\}$. We use the same notation as in the proof of Proposition \ref{prop:forcing1} and let for a given $n \in \N$,  $\Gamma_{\rho,l}=\prod_{i=1}^{2n}\gamma_{\rho_i,l}$,  where $\rho = (\rho_1, \cdots, \rho_{2n})$ with $\rho_i \in \{0,\cdots, l\}$,  $l=8$ in the case of family $\mathcal{P}_1$, and $l=1$ in the case of family $\mathcal{P}_2$. 
In order to find the suitable coverings that allows us to make the conclusions above, it is clear from the proof in \cite{LeCalvezTal1} that it is sufficient to check the following condition:
There is a finite family of paths $\tau_1, \cdots, \tau_N$, and a transverse loop $\Gamma^*$ with natural lift $\gamma^*$ such that 
\begin{itemize} 
\item  $\tau_i$, for any $i\in \{1,\cdots,  N\}$, has a leaf on its right and a leaf on its left, and
\item if a path $\sigma$ in $\dom(I)$ is equivalent to a subpath of $\prod_{i=1}^{2n}\gamma_{\rho_i,l}$, for some $n \in \N$, and there is no subpath of $\sigma$ that is equivalent to $\tau_i$, for some $i\in \{1, \cdots, N\}$, then $\sigma$ is equivalent to a subpath of $\gamma^*$. 
\end{itemize}

For $\mathcal{P}_1$ this condition holds true, as it is easy to check for the choices
$\tau_i:= \gamma_{i,8}$, $i \in \{0, \cdots k\}$, and $\Gamma^* :=  \prod_{1\leq i\leq k}\prod_{1\leq j \leq k} \tau_i \tau_j$.
One then can conclude from Proposition \ref{prop:forcing1}, considering quasi-palindromic words $\rho$ with $\rho_1=0$ and arguing as in the proof of \cite[Proposition 38]{LeCalvezTal1}, that there is a constant $M_1 > 0$ and for every $p \in \N$ a covering $\mathcal{V}_1^p$ such that 
the cardinality of $\bigvee_{0\leq k \leq 16nq} f^{-k}(\mathcal{V}_1^p)$ is at least $(k+1)^{n-1}/Ln^2M_1^{8nq/p}$.
We conclude that 
\begin{align}
\begin{split}
h_{\topo}(f) \geq &\sup_{p \in \N} h_{\topo}(\hat{f},\mathcal{V}_1^p) \geq \sup_{p \in \N} \lim_{n\to \infty} \frac{1}{16nq} \log((k+1)^{n-1}/Ln^2 M_1^{8nq/p}) = \\ &\sup_{p \in \N} \{(\log (k+1))/16q - \log M_1/ 2p\} = (\log(k+1))/16q, 
\end{split}
\end{align}
where $h_{\topo}(\hat{f},\mathcal{U}) := \limsup_{n \to \infty} \log(N_n(\hat{f}, \mathcal{U}))/n$ for a covering $\mathcal{U}$ of $\dom(I) \cup \{\infty\}$. 

In the case of $\mathcal{P}_2$ we choose some large $u\in \N$ and consider only quasi-palindromic words $\rho$ of length $2n$ such that $\rho_{\nu u} = 0$ for all $\nu \in \N$. 
Let $\tau_1 := \gamma^2_{0,1}$, $\tau_2:= \gamma_{0,1}\gamma_{1,1}, \tau_3:=\gamma_{1,1}\gamma_{0,1}$, and $\Gamma^*=\gamma^3_{0,1}\gamma^{u}_{1,1}$. $\tau_1, \tau_2$, and $\tau_3$ have all a leaf on its right and a leaf on its left, as can be easily checked. 
Any path $\sigma$ that is equivalent to a subpath of $\prod_{1\leq j \leq 2n}\gamma_{\rho_j,1}$ and does not have a subpath equivalent to some $\tau_i, 1\leq i\leq 3$, must be equivalent to a subpath of $\gamma_{0,1}$, $\gamma^2_{0,1}$, $\gamma^3_{0,1}$, $\gamma^2_{0,1}\gamma_{1,1}$, $\gamma_{1,1}\gamma^2_{0,1}$,  $\gamma_{0,1}\gamma^s_{1,1}$, $\gamma^s_{1,1}\gamma_{0,1}$, $\gamma_{1,1}^{s}$, $ 1\leq s \leq u-1$, hence is a subpath of the natural lift $\gamma^*$ of $\Gamma^*$.  
One can conclude that there is a constant $M_2 > 0$ and for every $p \in \N$ a covering $\mathcal{V}_2^p$ such that 
the cardinality of $\bigvee_{0\leq k \leq 2nq} f^{-k}(\mathcal{V}_2^p)$ is at least $2^{\frac{u-1}{u}(n-1)}/Ln^2 M_2^{nq/p}$.
Hence
\begin{align}
\begin{split}
h_{\topo}(f) \geq &\sup_{p \in \N} h_{\topo}(\hat{f},\mathcal{V}_2^p) \geq \sup_{p \in \N} \lim_{n\to \infty} \frac{1}{2nq} \log(2^{\frac{u-1}{u}(n-1)}/Ln^2 M_2^{nq/p}) = \\ &\sup_{p \in \N} \left\{\frac{u-1}{u}(\log 2)/2q - \log M_2 / 2p\right\} = \frac{u-1}{u}\log 2/2q. 
\end{split}
\end{align}
But since $u$ can be chosen arbitrarily large, we get that 
$h_{\topo}(f) \geq \log 2/2q$. 

The proof for the case $m>1$ is again almost identical. From the proof of Proposition \ref{prop:forcing1} for the families $\mathcal{P}_1$ or $\mathcal{P}_2$ we get, for every quasi-palindromic word $\rho$ of order $2n$ with $\rho_1 = 0$ and $n\in \N$ that is a multiple of $m$, transverse loops $\Gamma_{\rho,l}$ that are linearly admissible of order $2lnq/m$. One argues then as above to get the claimed lower bounds of the Proposition.

\end{proof}

\section{Turaev's cobracket and orbit growth}\label{sec:Turaev}
Goldman's bracket \cite{goldman} and Turaev's cobracket \cite{Turaev1991} define a Lie bialgebra structure on the free $\Z$-module of non-trivial free homotopy classes of loops $\widehat{\pi}(M)^*:= \widehat{\pi}(M)\setminus\{[*]\}$ on a surface $M$.\footnote{Similar mappings were investigated earlier in \cite{Turaev1978}.}  Turaev's cobracket $v:\Z[\widehat{\pi}(M)^*] \to \Z[\widehat{\pi}(M)^*] \otimes \Z[\widehat{\pi}(M)^*]$ applied to a free homotopy class $\alpha$ gives some information about the free homotopy classes of the loops that split at intersection points of any representative $\Gamma$ of $\alpha$. We adopt the construction of $v$ and refine it to keep additional information about how these loops relate to $\Gamma$ in the fundamental group $\pi_1(M,\Gamma(0))$. Via this invariant we then  define an exponential growth rate $\Tu^{\infty}(\alpha)$ of $\alpha$, and prove Theorem \ref{HT}.

Recall that we denote by $\mathcal{S}(\Gamma)= \{y \in M \,|\, y = \Gamma(t) = \Gamma(t'), t\neq t'\}$ the set of self-intersection points of a loop $\Gamma:S^1 \to M$. We say that a smooth loop $\Gamma$ is in \textit{general position} if it is an immersion, it has only double intersection points which are moreover transverse, and $\Gamma(0) \notin \mathcal{S}(\Gamma)$.
We first recall the construction of Turaev's cobracket. Let $\alpha \in \widehat{\pi}(M)^*$ be a free homotopy class of loops in $M$. Choose a smooth representative $\Gamma:[0,1] \to M$, $\Gamma(0) = \Gamma(1)$ of $\alpha$ that is in general position.
For  $y \in \mathcal{S}(\Gamma)$ let $\vec{v}_1$ and $\vec{v}_2$ be the two tangent vectors of $\Gamma$ at $y$, labeled such that the pair $\vec{v}_1, \vec{v}_2$ is positively oriented. Consider two loops $u^y_1, u^y_2$ that start both at $y$ in the direction $\vec{v}_1$ and $\vec{v}_2$, respectively, follow along $\Gamma$ until their first return to $y$.
Set 
\[ v(\alpha) := \sum_{y \in \mathcal{S}(\Gamma)} [u^y_1]^* \otimes [u^y_2]^* -  [u^y_2]^* \otimes [u^y_1]^* \,\,  \in \,\,\Z[\widehat{\pi}(M)^{*}] \otimes \Z[\widehat{\pi}(M)^{*}],\]
where $[u]^*$ denotes the free homotopy class of $u$ if $u$ is non-contractible and $[u]^* := 0$ otherwise. 
One can show that the right-hand side is invariant under free homotopies and hence $v$ is well-defined, see \cite{Turaev1991}. Extending $v$ linearly to $\Z[\widehat{\pi}(M)^{*}]$, sending the trivial class to $0$, defines Turaev's cobracket. 

To refine this construction we make the following definitions. Choose a basepoint  $x_0 \in M$. For $g \in \pi_1(M,x_0)$, we say that two elements $a$ and $a'$ in $\pi_1(M,x_0)$ are \textit{$g$-equivalent} if there is $k \in \Z$ with $g^k a {g^{-k}} = a'$.
Denote by $\pi_1(M,x_0)^* = \pi_1(M,x_0) \setminus \{1\}$ the set of elements of $\pi_1(M,x_0)$ without the identity $1$, and by $\pi_1(M,x_0)^*_{g}$ the set of $g$-equivalence classes in $\pi_1(M,x_0)^*$. Let $\Z[\pi_1(M,x_0)^*_{g}]$ be the free $\Z$-module over $\pi_1(M,x_0)^{*}_g$. 
We view the disjoint union  ${\mathcal{H} := \cup_{g \in \pi_1(M,x_0)} \Z[\pi_1(M,x_0)^*_{g}] \otimes \Z[\pi_1(M,x_0)^*_g]}$ as a bundle  $p: \mathcal{H} \to \pi_1(M,x_0)$ over $\pi_1(M,x_0)$ with fibre ${p^{-1}(g) =\Z[\pi_1(M,x_0)^*_{g}] \otimes \Z[\pi_1(M,x_0)^*_g]}$. 
We define  a section $\mu:\pi_1(M,x_0) \to \mathcal{H}$ as follows. 
Let $\Gamma$ with $\Gamma(0) = \Gamma(1) = x_0$ be any smooth loop in general position that  represents  $g$. For each $y\in \mathcal{S}(\Gamma)$ let $v_1, v_2, u^y_1, u^y_2$ be as above, and additionally consider two paths $q^y_1, q^y_2$  that start both at $\Gamma(0)$, follow $\Gamma$ in positive direction, 
and end at $y$ such that the following holds:
$q^y_1$ ends at $y$ as soon as it reaches $y$ the first time for which its tangent vector coincides with $v_2$, and  $q^y_2$ ends at $y$ as soon as it reaches $y$ the first time for which its tangent vector coincides with $v_1$. Let $a^y_i = \langle q^y_i u^y_i \overbar{q^y_i}\rangle  \in \pi_1(M,x_0), \, i=1,2$, be the element represented by the loop $ q^y_i u^y_i \overbar{q^y_i}$, where $\overbar{q^y_i}$ is the reverse path of $q^y_i$, see figures~\ref{fig:a1},\ref{fig:a2}. One checks that 
\begin{equation}\label{ab=g}
[a^y_2]_g = [(a^y_1)^{-1}g]_g, [a^y_1]_g = [g(a^y_2)^{-1}]_g.
\end{equation} 

Define
\begin{equation}\label{defmu} 
\mu(g) := \sum_{y \in \mathcal{S}(\Gamma)} [a^y_1]^*_g \otimes [a^y_2]^*_g - [a^y_2]^*_g \otimes [a^y_1]^*_g\, \, \in\,  \,\mathcal{H},
\end{equation}
where $[a]_g^*$ denotes the $g$-equivalence class of $a$ if $a\neq 1$ and $[1]_g^* := 0$.

\begin{figure}[!ht]
	\centering
	\begin{minipage}{.5\textwidth}
		\centering
		\begin{tikzpicture}
			\draw [thick,->] (-1,1) arc (45:360-45:{sqrt(2)}) -- (0,0) edge (1,1) -- (1,1) arc (180-45:-180+45:{sqrt(2)}) -- (0,0) edge (-1,1) -- cycle;
			
			\node at (-1,-1)[circle, fill, inner sep=1pt]{};
			\draw (-1,-1) node[anchor=north west] {$\Gamma(0)$};
			
			\node at (0,0)[circle, fill, inner sep=1pt]{};
			\draw (0,0) node[anchor=west] {$y$};
		    
		    \draw (1,1) node[anchor=south] {$v_1$};
		    \draw (-1,1) node[anchor=south] {$v_2$};
		\end{tikzpicture}
		\caption{A self-intersecting loop $\Gamma$.}
		\label{fig:loop1}
	\end{minipage}%
	\begin{minipage}{.5\textwidth}
		\centering
		\begin{minipage}{.5\textwidth}
		    \centering
		    \begin{tikzpicture}
			    \draw (0,0) -- (1,1) arc (180-45:-180+45:{sqrt(2)});
			    \draw [thick,->] (1,-1) -- (0,0);
			    \draw (0,0) node[anchor=east] {$y$};
		    \end{tikzpicture}
		    \caption{$u_1^y$}
		    \label{fig:u1}
	    \end{minipage}%
	    \begin{minipage}{.5\textwidth}
		    \centering
		    \begin{tikzpicture}
		        \draw (0,0) -- (-1,1) arc (45:360-45:{sqrt(2)});
		        \draw [thick,->] (-1,-1) -- (0,0);
		        \draw (0,0) node[anchor=west] {$y$};
		    \end{tikzpicture}
		    \caption{$u_2^y$}
		    \label{fig:u2}
	    \end{minipage}
	\end{minipage}
	\\
    \begin{minipage}{.6\textwidth}
	    \centering
	    \hspace{-1.5cm}%
		\begin{minipage}{.5\textwidth}
		    \centering
		    \begin{tikzpicture}[scale=0.7]
			    \draw (-1,-1) -- (1,1) arc (180-45:-180+45:{sqrt(2)});
			    \draw [thick,->] (1,-1) -- (0,0);
			    \draw (0,0) node[anchor=east] {$y$};
			    \draw (-1,-1) node[anchor=north] {$\Gamma(0)$};
		    \end{tikzpicture}
		    \caption{$q_1^y$}
		    \label{fig:q1}
	    \end{minipage}%
	    \hspace{-1.5cm}%
	    \begin{minipage}{.5\textwidth}
		    \centering
		    \begin{tikzpicture}
		        \draw [thick,->] (-1,-1) -- (0,0);
		        \draw (0,0) node[anchor=west] {$y$};
			    \draw (-1,-1) node[anchor=north] {$\Gamma(0)$};
		    \end{tikzpicture}
		    \caption{$q_2^y$}
		    \label{fig:q2}
	    \end{minipage}
    \end{minipage}%
	\begin{minipage}{.6\textwidth}
		\centering
		\hspace{-5cm}%
		\begin{minipage}{.5\textwidth}
		    \centering
		    \begin{tikzpicture}[scale=0.7]
			    \draw (-1,-1) -- (1,1) arc (180-45:-180+45:{sqrt(2)}) -- (0.1,-0.1);
			    \draw [thick,->] (0.1,-0.1) -- (-0.9,-1.1);
			    \draw (0,0) node[anchor=east] {$y$};
			    \draw (-1,-1) node[anchor=north] {$\Gamma(0)$};
	    			\node at (-1,-1)[circle, fill, inner sep=1pt]{};
		    \end{tikzpicture}
		    \caption{A loop representing $a_1^y$}
		    \label{fig:a1}
	    \end{minipage}%
	    \hspace{-1.5cm}%
	    \begin{minipage}{.5\textwidth}
		    \centering
		    \begin{tikzpicture}
		        \draw (-1,-1) -- (0,0);
		        \draw [thick,->] (0,0) -- (-1,1);
		        \draw (-1,1) arc (45:360-45:{sqrt(2)});
		        \draw (0,0) node[anchor=west] {$y$};
			    \draw (-1,-1) node[anchor=north west] {$\Gamma(0)$};
	    			\node at (-1,-1)[circle, fill, inner sep=1pt]{};
		    \end{tikzpicture}
			\hspace{-1cm}
		    \caption{A loop representing $a_2^y$}
		    \label{fig:a2}
	    \end{minipage}
	\end{minipage}
	
\end{figure}

We now show that $\mu$ is well-defined, and respects conjugation in the following sense. Let $g\in \pi_1(M,x_0)$. Conjugation by $h\in \pi_1(M,x_0)$ defines a map  $\phi^g_h: \pi_1(M,x_0)^*_g \to \pi_1(M,x_0)^*_{hgh^{-1}}$, $[a]_g \mapsto [hah^{-1}]_{hgh^{-1}}$. By extending linearly to $\Z[\pi_1(M,x_0)_g]^*$ and taking tensor products, we obtain mappings $\overline{\phi_h}= \bigcup_{g\in \pi_1(M,x_0)} \phi^g_h \otimes \phi^g_h :\mathcal{H} \to \mathcal{H}$.

\begin{lem}\label{lem:mu_welldefined}
$\mu$ is well defined and $\mu(hgh^{-1}) = \overline{{\phi}_h}(\mu(g))$ for all $h,g\in \pi_1(M,x_0)$. 
\end{lem} 
\begin{proof}
In the space of $C^{\infty}$ loops $C^{\infty}(S^1, M)$ we denote the set of loops in general position by $\Omega$ and  denote by $\Omega^*$ the subspace of loops that satisfy the properties of being in general position except at one exceptional point $y = \Gamma(t_0)$, where exactly one of the following occurs:
(1) the differential $\frac{d}{d t} \Gamma(t_0)$ vanishes and $\frac{d^2}{dt^2}\Gamma(t_0) \neq 0$; (2) there is a self-tangency at $y$; (3) there is a triple intersection at $y$; (4) $y = \Gamma(0) \in \mathcal{S}(\Gamma)$.
 $\Omega^*$ has  codimension $1$ in $\Omega$. 

Let $g$ and $g'$ be two elements in $\pi_1(M,x_0)$, and $\Gamma, \Gamma'$ be free loops in $\Omega$ with $\Gamma(0) = \Gamma'(0) = x_0$, representing $g$ and $g'$, respectively. 
A generic path of loops $\Gamma_s, s\in [0,1]$, from $\Gamma$ to $\Gamma'$ lies in $\Omega \cup \Omega^*$ and intersects $\Omega^*$ transversally for finitely many parameter values  $A\subset (0,1)$ (See \cite[p.291-294]{goldman}).
According to which of the cases (1) - (4) occur for some $s_* \in A$, we have that near $s_*$ the homotopy $\Gamma_s$ behaves like one of the  following elementary moves: 
(1) birth-death of a monogon; (2) birth-death of a bigon; (3)  jumping over an intersection; (4)  jumping over $\Gamma_s(0)$. (For $(1) - (3)$ see \cite{goldman}, the forth move is illustrated in Figure \ref{fig:strand}). 
Write $A$ as a disjoint union $A= A_1 \cup A_2 \cup A_3 \cup A_4$ such that $s_*\in A_i$ if and only if situation $(i)$ occurs near $s_*$.

\begin{figure}
    \centering
    \scalebox{0.7}{
    \begin{tikzpicture}[scale=0.4]
        \coordinate (O) at (0,0);
        
          \draw[dashed] (O) circle (8);
       
        \path[name path=strand1] (320:8) .. controls (2.5,-1) and (2.5,2) .. (80:8)   ;
        \path[name path=strand2] (200:8) .. controls (-3,0) and (4,2) .. (40:8) node[midway]{\textbullet} node[midway, above]{\Large $\Gamma_s(0)$};

        \foreach \i in {1,2} {
            \draw[use path=strand\i,lift];
						}
\draw[ <->] (12,0) -- (18,0);
   \begin{scope}[shift={(30,0)}]
		\coordinate (1) at (0,0);

       \draw[dashed] (1) circle (8);
							\path[name path=strand3] (310:8) .. controls (-3,-3) and (-3,4) .. (90:8) ;
        \path[name path=strand4] (200:8) .. controls (-3,0) and (4,2) .. (40:8) node[midway]{\textbullet} node[midway, above]{\Large $\Gamma_s(0)$};

			\foreach \i in {3,4} {
            \draw[use path=strand\i,lift];
						}

		  \end{scope}

        


    \end{tikzpicture}
		}
    \caption{Jumping over $\Gamma_s(0)$ }
    \label{fig:strand}
\end{figure}
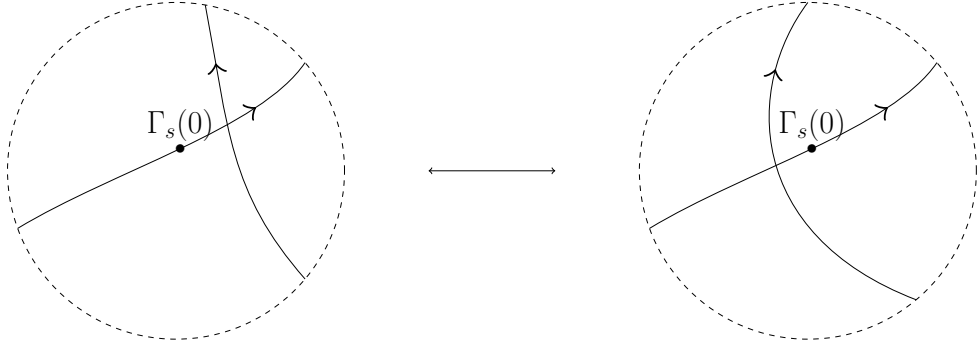

Consider the $s$-family of paths $b_s(t)$, $s\in [0,1]$,  with $b_0(t) \equiv x_0$ and $b_s(t) = \Gamma_{ts}(0)$, which all start at $x_0$ and end at $\Gamma_s(0)$. 
For each $s \in [0,1]$, and each $y_s \in \mathcal{S} (\Gamma_s)$ that is not an exceptional point, consider paths $q^{y_s}_1$ and $q^{y_s}_2$ as in the definition of $\mu$ for $\Gamma_s$, from $\Gamma_s(0)$ to $y_s$, and parametrized proportional to $\Gamma_s$. 
Similarly define $u^{y_s}_i$, $i=1,2$.    
Consider for any $s\in [0,1] \setminus A$  the loops $p^{y_s}_i = b_s q^{y_s}_i u^{y_s}_i  \overbar{q^{y_s}_i} \overbar{b_s}$ and the following element in $\Z[\pi_1(M,x_0)_g^*] \otimes \Z[\pi_1(M,x_0)_g^*]$:
\[ 
 \xi_s := \sum_{y_s \in \mathcal{S}(\Gamma_s)} [\langle p^{y_s}_1 \rangle ]_{g} \otimes 
[\langle p^{y_s}_2 \rangle ]_{g}- [\langle p^{y_s}_2 \rangle ]_{g}\otimes[\langle p^{y_s}_1 \rangle ]_{g}. \]

We claim that $\xi_s$ is constant along the homotopy $\Gamma_s$. 
To see this, one has to check how $\xi_s$ is affected when $s$ crosses some $s_* \in A$, since outside of $A$ one can choose all the data continuously in $s$.  
The argument that $\xi_s$ does not change near some $s_* \in A_1 \cup A_2 \cup A_3$ is similar to the argument of invariance of Turaev's cobracket \cite{Turaev1991}, this holds even without taking $g$-equivalence classes. 
We check invariance near some $s_* \in A_4$. Let, for $s\in [s_*-\epsilon, s_*+\epsilon]$, $\epsilon>0$ sufficiently small,  $y_s$ be those intersection point of $\Gamma_s$ such that $y_{s_*} = \Gamma_{s_*}(0)$ and $y_s$ varies continuously in $s\in [s_*-\epsilon, s_*+\epsilon]$. Let $q_i^{y_s}, u_i^{y_s}$, $i=1,2$, be defined as above. 
Note that either for $i=1$ or $i=2$ we have that $\lim_{s\searrow s_*} q_i^{y_s}\equiv\Gamma_{s_*}(0)$ is constant and $\lim_{s\nearrow s_*} q^{y_s}(t)= \Gamma_{s_*}(t)$, or vice versa, interchanging $s\searrow s_*$ with $s\nearrow s_*$. 
Assume that  $\lim_{s\searrow s_*} q_1^{y_s}\equiv\Gamma_{s_*}(0)$, the argument for the other cases is analogous. 
Clearly $\langle b_{s}  q^{y_s}_2 u^{y_s}_2  \overbar{q^{y_{s}}_2}\overbar{b_s} \rangle \in \pi_1(M,x_0)$ is identical for all $s\in [s_*-\epsilon,s_*+\epsilon]$, since all paths vary continuously in $s$. 
Also, $\lim_{s\searrow s_*}\langle b_s  q^{y_s}_1 u^{y_s}_1  \overbar{q^{y_{s_*}}_1} \overbar{b_s} \rangle = \langle b_{s_*}  u^{y_{s_*}}_1  \overbar{b_{s_*}} \rangle=:f\in \pi_1(M,x_0)$, and
 $\lim_{s\nearrow s_*}\langle b_s  q^{y_s}_1 u^{y_s}_1  \overbar{q^{y_{s_*}}_1} \overbar{b_s} \rangle = 
\langle b_{s_*}  \Gamma_{s_*}   u^{y_{s_*}}_1   \overbar{\Gamma_{s_*}}  \overbar{b_{s_*}}\rangle = gfg^{-1}$, 
since $g = \langle b_{s_*}  \Gamma_{s_*}   \overline{b_{s_*}} \rangle$. 
All other terms of $\xi_s$ for $s\in [s_*-\epsilon, s_*+\epsilon]$ correspond to non-exceptional intersection points if $\epsilon$ is sufficiently small and clearly do not change along $s\in [s_*-\epsilon, s_*+\epsilon]$. 

If we take in the above discussion $g'= g$ and a homotopy $\Gamma_s$ such that $\Gamma_s(0)$ stays close to $x_0$, then the loop $b_1(t)$ is contractible, and using $\xi_1 = \xi_2$ we see that the definition of $\mu(g)$ is independent whether choosing the loops $\Gamma$ or $\Gamma'$. 

If we take $g'= hgh^{-1}$ for some $h\in \pi_1(M,x_0)$, then 
\begin{align*}
\begin{split}
 \mu(g') &=\sum_{y \in \mathcal{S}(\Gamma_1)} [a^{y}_1]_{g'} \otimes [a^{y}_2]_{g'}- [a^y_2]_{g'}\otimes[a^y_1]_{g'} \\
 &=\sum_{y \in \mathcal{S}(\Gamma_1)} [h h^{-1} a^{y}_1 h h^{-1}]_{g'} \otimes [h h^{-1} a^{y}_2 h h^{-1}]_{g'}- [h h^{-1}  a^y_2 h h^{-1}]_{g'}\otimes[h h^{-1} a^y_1 h h^{-1}]_{g'} 
\\&= \overline{{\phi}_h}(\xi_1)  \\ &=\overline{{\phi}_h}(\mu(g)),
\end{split}
\end{align*}
where in the third equation we use that the loop $b_1$ represents $h^{-1}$ as well as the definition of $\xi_1$, and in the last equation 
we use that $\xi_s$ is constant along the homotopy $\Gamma_s$.   

\end{proof}

Note that $\widehat{\pi}(M)$ can be identified with the set of conjugacy classes of $\pi_1(M,x_0)$.
By Lemma \ref{lem:mu_welldefined}, $\mu$ induces a section 
$\widehat{\mu}:\widehat{\pi}(X) \to \widehat{\mathcal{H}}$, 
where $\widehat{\mathcal{H}}$ is obtained from $\mathcal{H}$ if we mod out by the action $\pi_1(M,x_0) \times \mathcal{H} \to \mathcal{H}$, $(h,\xi) \mapsto \overline{\phi_h}(\xi)$ and is hence actually a bundle over $\widehat{\pi}(M).$  Tuarev's cobracket can be recovered from $\widehat{\mu}$ by sending (tensors of) $g$-equivalence classes in $\widehat{\mathcal{H}}$ to (tensors of) conjugacy classes.

We will now define a growth rate $\Tu^{\infty}(\alpha)$ of $\alpha$ in terms of $\widehat{\mu}(\alpha)$. Roughly speaking, up to a modification due to some parametrization issues, it is the exponential growth rate in $k$ of the number of free homotopy classes of loops that can be obtained by following a loop representing $\alpha$ $k$-times and each time allowing a shortcut at some $\widehat{\mu}(\alpha)$-relevant self-intersection point and the turning at these points has to be always to the  right or always to the left.   
 
Let $S=\{s_1, \cdots, s_m\} \subset \pi_1(M,x_0)$. Denote by $B(n,S)$ the set of elements in $\pi_1(M,x_0)$ that can be written as a product of $\leq n$ factors that are elements in $S$.
Let $\widehat{B}(n,S)$ be the set of conjugacy classes of $B(n,S)$, $\widehat{N}(n,S) := \#\widehat{B}(n,S)$ the cardinality of  $\widehat{B}(n,S)$, and  $\Gamma(S):= \limsup_{n\to \infty} \frac{\log(\widehat{N}(n,S))}{n}< \infty$.
Given $g\in \pi_1(M,x_0)$ and a set $\mathfrak{S}$ of $g$-equivalence classes of elements in $\pi_1(M,x_0)$, we define 
$\Gamma(\mathfrak{S},g) = \inf\{\Gamma(S)\, | \, S \subset \pi_1(M,x_0), \,[S]_g = \mathfrak{S}\}$, where by $[S]_g$ we denote the set of $g$-equivalence classes of elements of $S \subset \pi_1(M,x_0)$. 

For a set $\mathfrak{S} = \{[a_1]_g, \ldots, [a_k]_g\}$ of $g$-equivalence classes, and $h\in \pi_1(M,x_0)$ denote  $h\cdot \mathfrak{S} \cdot h^ {-1} := \{[ha_1h^{-1}]_{hgh^{-1}}, \ldots, [ha_kh^{-1}]_{hgh^{-1}}\}$, and any $h$ induces in this way a bijection from the class of finite sets of $g$-equivalence classes to the class of finite sets of $hgh^{-1}$-equivalence classes. Moreover, one checks directly that 
\begin{equation}\label{conjugationS}
\Gamma(\mathfrak{S},g) = \Gamma(h \cdot \mathfrak{S} \cdot h^{-1},h g h^{-1}), 
\end{equation} 
for all $g,h\in \pi_1(M,x_0)$. 


Let now $\alpha \in \widehat{\pi}(M)$. Choose $g\in \pi_1(M,x_0)$, representing the class $[g] = \alpha$. 
$\mu(g)\in \mathcal{H}$ can be written as 
\begin{equation} 
\mu(g) = \sum_{\mathfrak{a}, \mathfrak{b} \in \pi_1(M,x_0)^*_g} k_{\mathfrak{a}, \mathfrak{b}}\left( \mathfrak{a} \otimes \mathfrak{b}\right),
\end{equation}
 where $k_{\mathfrak{a},\mathfrak{b}} \in \Z$, $\mathfrak{a}, \mathfrak{b}$ are $g$-equivalence classes, and $k_{\mathfrak{a},\mathfrak{b}} = - k_{\mathfrak{b},\mathfrak{a}}$. Let $\Comp(\mu(g))$ be the collection of terms $\mathfrak{a} \otimes \mathfrak{b}$ with $k_{\mathfrak{a},\mathfrak{b}} >0$. By \eqref{ab=g},  $\mathfrak{a}\otimes \mathfrak{b} \neq \mathfrak{a'}\otimes \mathfrak{b'} \in\Comp(\mu(g))$ implies  $\mathfrak{a} \neq \mathfrak{a'}$ and $\mathfrak{b} \neq \mathfrak{b'}$.  
 Note that by Lemma \ref{lem:mu_welldefined}, for any $h\in \pi_1(M,x_0)$,
 $\phi_h^g:\pi_1(M,x_0)_g \to \pi_1(M,x_0)_{hgh^{-1}}$ defines a bijection 
 \begin{equation}\label{eq:Comp}
  \Comp(\mu(g)) \cong \Comp(\mu(hgh^{-1})). 
 \end{equation}
 
 Let $\mathcal{T} \subset \Comp(\mu(g))$, $\mathcal{T}_+ = \{\mathfrak{a} \, | \, \mathfrak{a} \otimes \mathfrak{b} \in \mathcal{T}\}$, and   $\mathcal{T}_- = \{\mathfrak{b} \, | \, \mathfrak{a} \otimes \mathfrak{b} \in \mathcal{T}\}$. Note that $\# \mathcal{T} = \# \mathcal{T}_+ = \# \mathcal{T}_-$. 
Consider 
\[
\Gamma_{\mathcal{T}}^{g} := \min_{\pm} \min_{\mathfrak{S}}\left\{\Gamma(\mathfrak{S} \cup [g]_g, g)\, \middle| \, \mathfrak{S} \subset \mathcal{T}_{\pm}, \, \#\mathfrak{S}= \left\lceil \frac{1}{2}\#\mathcal{T}\right\rceil \right\}. \]
By \eqref{conjugationS}, $\Gamma_{\mathcal{T}}^{g} = \Gamma_{h\cdot \mathcal{T}\cdot h^{-1}}^{hgh^{-1}}$, for any $h\in \pi_1(M,x_0)$.
In particular, by \eqref{eq:Comp},
$\max \{\Gamma_{\mathcal{T}}^{g}\, | \, \mathcal{T} \subset \Comp(\mu(g))\} = \max \{ \Gamma_{\mathcal{T}'}^{hgh^{-1}} \, | \mathcal{T}' \subset \Comp(\mu(hgh^ {-1}))\}$. 
Hence this expression is independent of the choice of $g$ with $[g]=\alpha$  and we denote it by \begin{equation*}
\Tu^{\infty}(\alpha) :=\max \{\Gamma_{\mathcal{T}}^{g}\, | \, \mathcal{T} \subset \Comp(\mu(g))\}.
\end{equation*}   
 


We now give a proof of Theorem \ref{HT}. 

\begin{proof}[Proof of Theorem \ref{HT}]
We may assume $T^\infty(\alpha) \neq 0$, in particular $\selfi_{M}(\alpha) \neq 0$. We consider a maximal identity isotopy $\widehat{I}$ of $f$, and a foliation $\mathcal{F}$ transverse to $\widehat{I}$, which is possible by Theorem \ref{Theorem_foliation}. Let $\Gamma$ be a $q$-admissible $\mathcal{F}$-transverse path associated to $x$. Since being $\mathcal{F}$-transverse is an $C^0$-open condition, we can choose $\Gamma$ to be smooth. Moreover, we can choose $\Gamma$ to be in general position. 
 We have that $[\Gamma]_{\widehat{\pi}(M)} = \alpha$ is primitive. Note that also $[\Gamma]_{\widehat{\pi}(\dom(\widehat{I}))}$ is primitive. 
Let $x_0 = \Gamma(0)$ and let $g=\langle \Gamma \rangle \in \pi_1(M,x_0)$ be the element that is represented by $\Gamma$.
Let $\mathfrak{a}\otimes \mathfrak{b}$ be any element in $\Comp(\mu(g))$.   
By the construction of the map $\mu(g)$, there is an intersection point $x= \Gamma(t) = \Gamma(t')$, $t<t'$, and an element $a\in \pi_1(M,x_0)$ with $[a]_g = \mathfrak{a}$ such that $\gamma|_{[0,t']}\gamma|_{[t,t']}\overbar{\gamma|_{[0,t']}}$ is a representative of $a$ and  $\gamma|_{[0,t]}\gamma|_{[t',1]}\gamma|_{[0,t]}\overbar{\gamma|_{[0,t]}} = \gamma|_{[0,t]}\gamma|_{[t',1]}$ is a representative of $b$, or vice versa. 
\begin{claim}\label{claim:nontransverse}
One can choose $t$ and $t'$ above that any two lifts $\widetilde\gamma$ and $\widetilde\gamma'$ of $\Gamma$ that intersect in $\widetilde\gamma(t) = \widetilde\gamma'(t')$, intersect $\widetilde{\mathcal{F}}$-transversally.
\end{claim}

We will prove the claim below. Let now $\mathcal{T}\subset \Comp(\mu(g))$.
By the construction of $\mu(g)$, given any term $\mathfrak{a} \otimes \mathfrak{b} \in \mathcal{T}$, then $\mathfrak{a}$ or $\mathfrak{b}$ has a representative in $\pi_1(M,x_0)$ given trough a loop $\gamma|_{[0,t]}\gamma|_{[t',1]}$ for  $t<t'$ as above. It follows that for $m=\left\lceil \frac{1}{2}\#\mathcal{T}\right\rceil$  there is 
$\mathfrak{S}=\{\mathfrak{a}_1, \cdots, \mathfrak{a}_m\}  \subset \mathcal{T}_+$ or $\mathfrak{S} = \{\mathfrak{a}_1, \cdots, \mathfrak{a}_m\} \subset \mathcal{T}_-$ such that, with the above choices 
 of corresponding parameters $t_1, \cdots, t_m$, $t'_1, \cdots, t'_m\in [0,1)$, $t_i < t'_i$, and a set $S=\{a_1, \cdots, a_m\}\subset \pi_1(M,x_0)$ with $[S]_g = \mathfrak{S}$, we have that $\gamma_i := \gamma|_{[0,t_i]}\gamma|_{[t_i',1]}$ represents the element $a_i$, and any two lifts $\widetilde\gamma$ and $\widetilde\gamma'$ of $\Gamma$ that intersect in $\widetilde\gamma(t_i) = \widetilde\gamma'(t'_i)$, intersect $\mathcal{F}$-transversally, for all $i=1, \cdots, m$. 

Set $\gamma_0:= \gamma|_{[0,1]}$. 
With similar notation as in the proof of Proposition \ref{prop:forcing1} for 
$\rho= (\rho_1, \cdots, \rho_n)$ with $\rho_{i} \in \{0,1, \cdots, k\}$, 
consider the path $\gamma_{\rho} := \gamma_{\rho_1}\gamma_{\rho_2} \cdots \gamma_{\rho_n}$. 
The path defines a transverse loop $\Gamma_{\rho}$.
As in in the proof of Claim \ref{claim:admissible} of the proof of Proposition \ref{prop:forcing1}, one shows that 
for any  $\rho= (\rho_1, \cdots, \rho_n)$ with $\rho_i \in \{0, \cdots,k\}$, $i = 1, \cdots, n$,
 $\Gamma_{\rho}$ is linearly admissible of order $nq$: The fact that, for any $p\in \N$,  $\widehat{\gamma} := \gamma_{[0,1]}\gamma_{\rho}^p$ is admissible of order $pnq$ follows from Proposition \ref{Proposition23_1}, since by the assumption $\mathfrak{S} \subset \mathcal{T}_+$ or $\mathfrak{S} \subset \mathcal{T}_-$  and so all relevant $\mathcal{F}$-transverse self-intersections have the same sign.

Consider the set of free homotopy classes that are defined by $\Gamma_{\rho}$ for all $\rho = (\rho_1, \cdots, \rho_n)$ with $\rho_1 = \rho_2 = \rho_3 = 0$.  
With the notation above, this set coincides with $\widehat{N}(n-3,S)$. It follows from Lemma \ref{Lemma7} $(1)$ that $\Gamma_{\rho}$ has a $\mathcal{F}$-transverse self-intersection and hence by Proposition \ref{Proposition26}, there is for each $\rho$ of the above form a periodic point $z\in \dom(\mathcal{F})$ of period $nq$ such that $\Gamma_{\rho}$ is associated to $z$, so there are at least $\widehat{N}(n-3,S)$ periodic points of period $nq$ that are of pairwise different free homotopy classes. Hence,
\begin{align*}
\HP^{\infty}(f) = \limsup_{n\to\infty} \frac{\log(N_{\hp}(f,nq))}{nq} \geq \limsup_{n\to \infty} \frac{\log(\widehat{N}(n-3,S))}{nq} \geq \frac{1}{q}\Gamma_{\mathcal{T}}^{g}.
\end{align*}
Taking the maximum over $\mathcal{T}\subset \Comp(\mu(g))$ we obtain $\HP^{\infty}(f) \geq \frac{1}{q}\Tu^{\infty}(\alpha)$.  
\end{proof}

\begin{proof}[Proof of claim \ref{claim:nontransverse}]
By Lemma \ref{Lemma7} any two lifts $\widetilde\gamma$ and $\widetilde\gamma'= S\widetilde\gamma$ , where $S \neq T^k$ and $T$ shift of $\widetilde\gamma$, can only be $\widetilde{\mathcal{F}}$-equivalent along a finite interval. In particular, if in addition they do not intersect $\widetilde{\mathcal{F}}$-transversally, then $\widetilde\gamma|_{(B, \infty)}$, and $\widetilde\gamma|_{(-\infty,-B)}$ will be on the same side of $\widetilde\gamma'$ for $B$ sufficiently large. 
It is now sufficient to show that for any such lifts and $a,b, a', b'\in \R$, $a < b$, $a'< b'$ such that  $\widetilde\gamma(a) = \widetilde\gamma'(a')$, $\widetilde\gamma(b) = \widetilde\gamma'(b')$, and $\widetilde\gamma|_{(a,b)}$ does not intersect $\widetilde\gamma'$, the contributions to $\mu(\alpha)$ at the intersection points $y=\Gamma(a)$ and $y'=\Gamma(b)$ cancel each other.
So take such $a,b,a',b'$ and $\widetilde\gamma, \widetilde\gamma'$,
We may, by interchanging the lifts and translating if necessary assume that $0<a<a'<1$. 
We note that $b<b'$. Otherwise, each of the iterates $S^{k}\widetilde\gamma|_{[a,b]}$, $k\in \N$ is equivalent to a subpath of $\widetilde\gamma|_{[a,b]}$, and moreover one checks that all intersect $\widetilde\gamma$, a contradiction. Similarly one sees that $b'< b+1$. 
Let  $k,k'\in \Z$ with $k<b<k+1$ and $k'<b'<k'+1$. We have that 
$\widetilde\gamma|_{[0,a]}\widetilde\gamma'|_{[a',1]}\widetilde\gamma'|_{[1,k'+1]}$ is homotopic relative to endpoints to $\widetilde\gamma|_{[0,k]}\widetilde\gamma|_{[k,b]}\widetilde\gamma'|_{[b',k'+1]}$. 
By the construction of $\mu(g)$ and by \eqref{ab=g} it is enough to show that 
$[f]_g = [f']_g$, where 
$f \in \pi_1(M,x_0)$  is represented by the loop in $M$ that we obtain by  projecting $\widetilde\gamma|_{[0,a]}\widetilde\gamma'|_{[a',1]}$ to $\dom(\widehat{I})\subset M$, and 
$f'$ is,  if $(i)$  $b<b'< k+1 =k'+1$,  represented by the loop in $M$ that we obtain by  projecting $\widetilde\gamma|_{[k,b]}\widetilde\gamma'|_{[b',k'+1]}$ to $\dom(\widehat{I})\subset M$, and, if $(ii)$  $b<k'=k+1< b'<b+1$, represented by the loop in $M$ that we obtain by  projecting $\widetilde\gamma|_{[k,b]}\overline{\widetilde\gamma'|_{[k',b']}}$ to $\dom(\widehat{I})\subset M$. 
In case $(i)$ we get $fg^{k'} = g^kf$, and in case $(ii)$ we get $fg^{k'} = g^kf'g$. Hence in both cases $[f]_g = [f']_g$.

\end{proof}
  
\section{Floer theory and persistence modules prerequisites}
\label{sec:floer}
    \subsection{A bit of Floer homology}
    \label{sec:floer_hom}
        Let $(M^{2n},\omega)$ be a symplectic manifold, and let $H: S^1 \times M \to \R$ be a smooth function, called a Hamiltonian function. The function $H$ induces a time-dependent vector field $X_H: S^1 \times M \to TM$ on $M$ that satisfies:
        \[
            \forall t \in S^1: \omega(\cdot, X_{H_t}) = d{H_t}(\cdot) .
        \]
        where $X_{H_t}: M \to TM$ is the vector field at time $t \in S^1$ and $H_t: M \to \R$ is $H_t(\cdot) = H(t,\cdot)$. The flow induced by $H$, or by $X_H$, is the family of maps $\varphi_H^t: S^1 \times M \to M$ satisfying
        \[
            \forall t \in S^1, x \in M: \frac{d}{dt} \varphi_H^t(x) = X_H(t,x) .
        \]
        The time-one-map induced by $H$ is the map $\varphi_H: M \to M$ defined by $\varphi_H = \varphi_H^1$. We say also $H$ generates $\varphi_H$. The set of Hamiltonian diffeomorphisms $\Ham(M,\omega) := \{\varphi_H \, | \, H: S^1 \times M \to \R \text{ smooth} \}$ is given a group structure by composition.
        
        In~\cite{hofer1990topological}, H. Hofer defined a remarkable metric $d_{\hofer}: \Ham(M,\omega) \times \Ham(M,\omega) \to \R$. This metric is induced by the Hofer norm $||\cdot||_{\hofer}: \Ham(M,\omega) \to \R$, which is defined as follows:
        \[
            ||\psi||_{\hofer} = \inf \left\{ \int_0^1 \max_M H_t - \min_M H_t \ dt \, \middle|\,  H \text{ smooth, } \varphi_H = \psi \right\} , 
        \] 
        where the minimum runs over all Hamiltonians $H$ which have $\psi$ as the time-1-map of their Hamiltonian flow.

        Recall the setting of Morse theory: a manifold $X$ and a Morse function $f: X \to \R$ are given. Morse theory reveals a connection between the homology of sub-level sets $\{x \in X \, | \, f(x) < t\}_{t \in \R}$ and critical points of $f$. Floer theory is an an analogous theory, set in an infinite-dimensional manifold of loops on a manifold, which is equipped with the action functional, see~\cite{salamon1999lectures}.
        
        Let $(M^{2n},\omega)$ be a closed, symplectically aspherical symplectic manifold and let $\alpha \in \widehat\pi(M)$ be a free homotopy class of loops in $M$. Denote $\mathcal{L}_\alpha(M) = \{x: S^1 \to M \text{ smooth } \, | \, [x]_{\widehat\pi(M)} = \alpha\}$ the set of all smooth loops in $M$ which represent the class $\alpha$. Assume that $\alpha$ is symplectically atoroidal; that is, that for any loop $\rho$ in $\mathcal{L}_\alpha(M)$, which is to be thought of as a function $\rho: \mathbb{T}^2 \to M$, the following holds:
        \[
            \int_{\mathbb{T}^2} \rho^* \omega = \int_{\mathbb{T}^2} \rho^* c_1 = 0 ,
        \]
        where $c_1$ is the first Chern class of $M$ and $\mathbb{T}^2$ is the 2-torus.
				
				Note for future reference that if $M$ is a surface of genus $\geq 2$ this condition holds trivially, since $[\mathbb{T}^2,M] = 0$. Fix a reference loop $\eta_\alpha \in \mathcal{L}_\alpha(M)$ and let $x \in \mathcal{L}_\alpha(M)$. The above condition means that the quantity $\int_{\bar{x}} \omega := \int_{S^1 \times [0,1]} \bar{x}^* \omega$ is well-defined and independent of $\bar{x}$, for an map $\bar{x}: S^1 \times [0,1] \to M$ with $\bar{x}|_{S^1 \times \{0\}} = \eta_\alpha$ and $\bar{x}|_{S^1 \times \{1\}} = x$.
        
        Let $H: S^1 \times M \to \R$ be a Hamiltonian function on $M$. The action functional associated to $H$, $\mathcal{A}_H: \mathcal{L}_\alpha(M) \to \R$, is defined as follows:
        \[
            \mathcal{A}_H(x) = \int_0^1 H(t,x(t)) dt - \int_{\bar{x}} \omega .
        \]
        
        A fixed point $z$ of $\varphi_H$ shall be called \textit{nondegenerate} if the differential $(\varphi_H)_*: T_z M \to T_z M$ does not have 1 as an eigenvalue. A Hamiltonian function $H: S^1 \times M \to \R$ and its generated time-one-map $\varphi_H$ shall be called \textit{nondegenerate} if all fixed points $z$ of $\varphi_H$ are nondegenerate. A nondegenerate Hamiltonian diffeomorphism $\phi$ has isolated fixed points. The property of nondegeneracy is analogous, in the analogy between Morse and Floer theories, to the function $f: X \to \R$ being a Morse function.
        
        Let $x \in \mathcal{L}_\alpha(M)$. The tangent space to $\mathcal{L}_\alpha(M)$ at the point $x$, $T_x \mathcal{L}_\alpha(M)$, is identified with the space of vector fields $\xi$ along $x$, i.e. with the space of maps $\xi: S^1 \to TM$ which are compositions of $x$ with a section of $TM \to M$. One shows that the differential of $\mathcal{A}_H$ is
        \[
            (d\mathcal{A}_H)_x(\xi) = \int_0^1 \omega(\xi, X_H - \dot{x}(t)) dt .
        \]
        
        This formula for the differential of the action implies the following characterization of critical points of the action functional:
        
        \begin{prop}[The Least Action Principle]
        \label{prop:least_action}
            The critical points of $\mathcal{A}_H$ are exactly the periodic orbits of the flow generated by the Hamiltonian $H$.
        \end{prop}
        
        Endow $\mathcal{L}_\alpha(M)$ with an auxiliary Riemannian metric as follows. Choose a loop of $\omega$-compatible almost complex structures $J(t)$. The inner product at the point $x \in \mathcal{L}_\alpha(M)$, $\langle \cdot,\cdot \rangle_{x}: T_x \mathcal{L}_\alpha(M) \times T_x \mathcal{L}_\alpha(M) \to \R$, is defined to be
        \[
            \langle \xi,\zeta \rangle_x = \int_0^1 \omega(\xi(t),J(t) \zeta(t)) dt .
        \]
        
        Denote
        \[
            P(M,H)_\alpha = \{x \in \mathcal{L}_\alpha(M) \, |\,  x \text{ is a 1-periodic orbit of } \varphi_H^t\} .
        \]
        
        We wish to grade $P(M,H)_\alpha$; this is done using the Conley-Zehnder index as follows (for a definition of Conley-Zehnder index, see~\cite{gutt2012conley}). First, if $\Phi$ is a path of symplectic matrices where $\Phi(1)$ does not have 1 as an eigenvalue, denote by $\mu_{CZ}(\Phi)$ the Conley-Zehnder index of the path $\Phi$. Fix a trivialization $\eta_\alpha^* TM \simeq S^1 \times (\R^{2n},\omega_0)$ of the symplectic vector bundle $\eta_\alpha^* TM$. Let $x \in P(M,H)_\alpha$. For any annulus $w: [0,1] \times S^1 \to M$ connecting $x$ to $\eta_\alpha$, $w$ defines a trivialization $x^* TM \simeq S^1 \times (\R^{2n},\omega_0)$. Note that the symplectic a-toroidality condition implies that this trivialization does not depend on $w$. Using the trivialization of $x^* TM$, the differential $d(\varphi_H^t)_{x(0)}$ for $t \in [0,1]$ is a symplectic map of $(\R^{2n},\omega_0)$. Denote $Ind(x) = n - \mu_{CZ}(\{t \mapsto d(\varphi_H^t)_{x(0)}\})$. Note that indeed by nondegeneracy of $H$, $d(\varphi_H^1)_{x(0)}$ does not have 1 as an eigenvalue.
        
        Let $x,y \in P(M,H)_\alpha$ with $Ind(x) = Ind(y)+1$; denote by $\tilde{\mathcal{M}}(x,y)$ the space of solutions $u(s,t): \R \times S^1 \to M$ to the Floer equation:
        \[
           \bar{\partial}_{H,J}(u) =  \partial_s u + J(t) \left( \partial_t u - X_{H_t} \right) = 0, 
        \]
        which have boundary conditions $\lim_{s \to -\infty} u(s,t) = x(t), \ \lim_{s \to \infty} u(s,t) = y(t)$. Note that the loops $\{u(s, \cdot)\,  | \, s \in \R \}$ for $u \in \tilde{\mathcal{M}}(x,y)$ are gradient descent trajectories on the space $\mathcal{L}_\alpha(M)$ with respect to the auxiliary metric defined above and the action functional $\mathcal{A}_H$. The space $\tilde{\mathcal{M}}(x,y)$ has an obvious $\R$-action, and one can show that for a generic choice of $J$, $\tilde{\mathcal{M}}(x,y) / \R$ is a compact 0-dimensional manifold, i.e. a finite set of points. Denote in that case $n(x,y) = \#(\tilde{\mathcal{M}} / \R) \mod 2$.
        
        Denote the Floer complex over $\Z_2$, filtered by action less than $r \in \R$:
        \[
            CF_k^r(M,H)_\alpha = Span_{\Z_2}\{x \in P(M,H)_\alpha \, | \, Ind(x) = k,\,  \mathcal{A}_H(x) < r\} ,
        \]
        
        and consider the linear map $\partial_k: CF_k^r(M,H)_\alpha \to CF_{k-1}^r(M,H)_\alpha$ which linearly extends the following map on $\{x \in P(M,H)_\alpha \, | \, Ind(x) = k, \mathcal{A}_H(x) < r\}$:
        \[
            x \mapsto \sum_{\substack{y \in P(M,H)_\alpha \\ Ind(y) = k-1}} n(x,y) y .
        \]
        
        It can be shown that $\partial: CF_*^r(M,H)_\alpha \to CF_*^r(M,H)_\alpha$ is a differential, i.e. that $\partial^2 = 0$. Here $CF_*^r(M,H)_\alpha = \bigoplus_k CF_k^r(M,H)_\alpha$ denotes the Floer complex in all indices. Since $(CF_*^r(M,H)_\alpha, \partial)$ defines a chain complex, it has a well-defined homology, the filtered Hamiltonian Floer homology of action below $r$ in class $\alpha$, which is denoted $HF_*^r(M,H)_\alpha$. One shows that this homology does not depend on the choice of almost complex structure $J(t)$, and in addition, since $M$ is symplectically aspherical, that it does not in fact depend on the Hamiltonian $H$, but rather on the time-one-map of its flow. Thus for $\phi \in \Ham(M,\omega)$, denote its Hamiltonian Floer homology in free homotopy class $\alpha$, filtered with action less than $t$ and over $\Z_2$, by $HF_*^r(\phi)_\alpha$.
    
    \subsection{Persistence modules and barcodes}
    \label{sec:persistence}
        We shall begin this subsection by defining a persistence module. For more background on all the definitions and theorems appearing in this subsection see~\cite{polterovich2019topological}.
        
        \begin{defn}
        \label{defn:per_mod}
            Let $\mathbb{F}$ be a field. A \textit{persistence module} is a pair $(V,\pi)$ where $V = (V_t)_{t \in \R}$ is a family of finite-dimensional $\mathbb{F}$-vector spaces and $\pi = (\pi_{s,r})_{s,r \in \R, s \leq r}$ is a family of linear maps $\pi_{s,r}: V_s \to V_r$ such that the following conditions hold:
            \begin{itemize}
                \item For all $r \leq s \leq t$, $\pi_{s,t} \circ \pi_{r,s} = \pi_{r,t}$.
                \item There exists a finite set $Spec(V,\pi) \subset \R$, such that for any $t \not\in Spec(V,\pi)$ there exists a neighborhood $U$ of $t$ such that for all $r,s \in U$ with $r \leq s$: $\pi_{r,s}$ is an isomorphism.
                \item For any $t \in \R$ there exists $\epsilon$ such that for all $t - \epsilon < s \leq t$, $\pi_{s,t}$ is an isomorphism.
                \item There exists $s_- \in \R$ such that for all $s < s_-$, $V_s = 0$.
            \end{itemize}
        \end{defn}
        
        We present a few examples of persistence modules.
        \begin{itemize}
            \item Let $\mathbb{F}$ be a field, and let $I \subset \R$ be an interval of the form $(a,b]$ with $a \in \R$, $a < b \in \R \cup \{\infty\}$, where the interval $(a,\infty] \subset \R$ is to be interpreted as equal to $(a,\infty) \subset \R$. The persistence module $\mathbb{F}(I)$ consists of the following data: the vector spaces $\mathbb{F}(I)_t$, for $t \in \R$, are
            \[
                \mathbb{F}(I)_t = \left\{\begin{array}{ll}
                                        \mathbb{F}  & t \in I \\
                                        0           & t \not\in I \end{array}\right. ,
            \]
            and the linear maps $\pi_{s,r}: \mathbb{F}(I)_s \to \mathbb{F}(I)_r$ are
            \[
                \pi_{s,r} = \left\{
                                    \begin{array}{ll}
                                        id_\mathbb{F} & s,r \in I \\
                                        0             & else
                                    \end{array}
                            \right. .
            \]
            
            \item Let $\mathbb{F}$ be a field, let $X$ be a closed manifold and let $f: X \to \R$ be a Morse function. The Morse homology $H_*(\{x \in X | f(x) < t\};\mathbb{F})$ for $t \in \R$ induces an $\mathbb{F}$-persistence module: the vector fields are $V_t = H_*(\{x \in X\,  |\,  f(x) < t\};\mathbb{F})$, and the linear maps $\pi_{r,s}: V_r \to V_s$ are induced by the inclusion maps $i_{r,s}: \{x \in X \, | \, f(x) < r\} \hookrightarrow \{x \in X \, |\,  f(x) < s\}$.
            
            \item Let $(M,\omega)$ be a symplectic manifold, let $\alpha \in \widehat\pi(M)$ and let $\phi \in \Ham(M,\omega)$ be a nondegenerate Hamiltonian diffeomorphism. Denote by $HF_*^\bullet(\phi)_\alpha = (HF_*^r(\phi)_\alpha)_{r \in \R}$ the persistence module whose vector spaces are the filtered Floer homology vector spaces in class $\alpha$, $HF_*^r(\phi)_\alpha$ for $r \in \R$, and whose linear maps $\pi_{r,s}: HF_*^r(\phi)_\alpha \to HF_*^s(\phi)_\alpha$ are induced by the inclusion maps $CF_*^r(M,H)_\alpha \to CF_*^s(M,H)_\alpha$, where $H$ is a Hamiltonian that generates $\phi$.
        \end{itemize}
        
        The direct sum of two persistence modules is defined as follows.
        
        \begin{defn}
            Let $(V,\pi)$ and $(W,\tau)$ be two persistence modules over the same field. Their direct sum, denoted $V \oplus W$, is the persistence module whose vector spaces $(V \oplus W)_t$ are $(V \oplus W)_t = V_t \oplus W_t$ and whose linear maps are $(\pi \oplus \tau)_{r,s}: (V \oplus W)_r \to (V \oplus W)_s$ are $(\pi \oplus \tau)_{r,s} = \pi_{r,s} \oplus \tau_{r,s}$.
        \end{defn}
        
        Consider the notion of a barcode, defined below.
        
        \begin{defn}
        \label{defn:barcodes}
            A \textit{barcode} is a finite multiset of intervals. Explicitly, a barcode is a finite set of pairs of intervals of $\R$ and their multiplicities, $\{(I_i,m_i)\}_{i=1}^N$ for some $N \in \N$, where $I_i \subset \R$ is an interval of the form $(a,b]$ for some $a \in \R,a < b \in \R \cup \{\infty\}$, and $m_i \in \N$ is the multiplicity of $I_i$. The intervals which make up a barcode are called its \textit{bars}.
        \end{defn}
        
        \begin{figure}[!ht]
            \centering
            \begin{tikzpicture}
                \node (A) at (-8,0) {};
                \node (B) at (8,0) {};
                
                \node (C1) at (-6,1) {};
                \node at (C1.east) {(};
                
                \node (C2) at (7,1) {};
                \node at (C2.east) {...};
                
                \node (D1) at (-6,2) {};
                \node at (D1.east) {(};
                
                \node (D2) at (-4,2) {};
                \node at (D2.west) {]};
                
                \node (E1) at (-6,3) {};
                \node at (E1.east) {(};
                
                \node (E2) at (-4,3) {};
                \node at (E2.west) {]};
                
                \node (F1) at (0,2) {};
                \node at (F1.east) {(};
                
                \node (F2) at (7,2) {};
                \node at (F2.east) {...};
                
                \node (G1) at (1,3) {};
                \node at (G1.east) {(};
                
                \node (G2) at (4,3) {};
                \node at (G2.west) {]};
                
                \draw[->] (A) -- (B);
                \draw (C1) -- (C2);
                \draw (D1) -- (D2);
                \draw (E1) -- (E2);
                \draw (F1) -- (F2);
                \draw (G1) -- (G2);
            \end{tikzpicture}
            \caption{An example barcode.}
            \label{fig:barcode_ex}
        \end{figure}
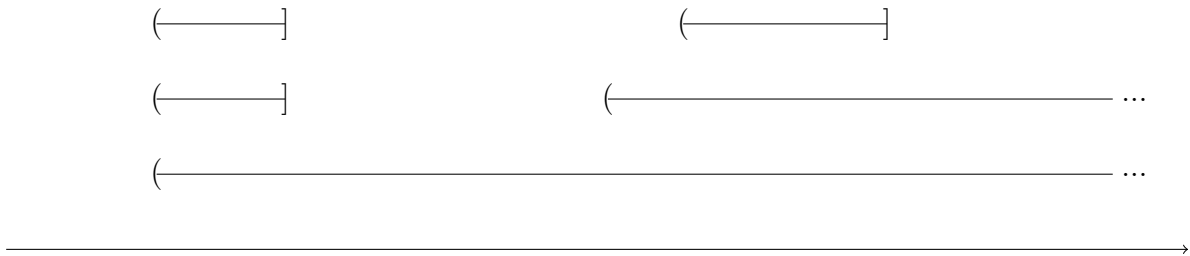
        
        Persistence modules and barcodes are matched in a 1-to-1 manner, as stated in Theorem~\ref{thm:normal_form}.
        
        \begin{thm}[Normal Form Theorem]
        \label{thm:normal_form}
            Let $(V,\pi)$ be a persistence module over a field $\mathbb{F}$. Then there exists a unique barcode $\mathcal{B}(V) = \{(I_i,m_i)\}_{i=1}^N$ such that
            \[
                V = \bigoplus_{i=1}^N \mathbb{F}(I_i)^{m_i} ,
            \]
            where here equality is to be understood as persistence modules isomorphism, and uniqueness of the barcode is up to permutation of the order in which its bars appear.
        \end{thm}
        
        We next recall a metric on the space of all barcodes, called the bottleneck distance, and denoted $d_{bot}$.
        
        \begin{defn}
        \label{defn:bottleneck_dist}
            Let $I = (a,b]$ be an interval with $a \in \R, a < b \in \R \cup \{\infty\}$, and let $\delta > 0$. Denote by $I^{-\delta}$ the interval $(a-\delta, b+\delta]$. Let $\mathcal{B}$ be a barcode and $\delta > 0$. Denote $\mathcal{B}_\delta = \{(I,m) \in \mathcal{B}\,  |\,  I = (a,b] \text{ with } b-a > \delta \}$, i.e. $\mathcal{B}_\delta$ is the set of bars of $\mathcal{B}$ which have length greater than $\delta$.
            
            Let $X,Y$ be multisets. A matching $\mu: X \to Y$ is a bijection $\mu: X^\prime \to Y^\prime$, where $coim \ \mu = X^\prime \subseteq X$, $im \ \mu = Y^\prime \subset Y$ are sub-multisets.
            
            Let $\mathcal{A},\mathcal{B}$ be barcodes. A \textit{$\delta$-matching} between $\mathcal{A},\mathcal{B}$ is a matching $\mu: \mathcal{A} \to \mathcal{B}$ such that $\mathcal{A}_{2\delta} \subseteq coim \ \mu$, $\mathcal{B}_{2\delta} \subseteq im \ \mu$, and for any $I \in coim \ \mu$: $I \subset (\mu(I))^{-\delta}$ and $\mu(I) \subset I^{-\delta}$.
            
            The \textit{bottleneck distance} between two barcodes $\mathcal{A},\mathcal{B}$ is
            \[
                d_{bot}(\mathcal{A},\mathcal{B}) = \inf \{ \delta \, |\,  \exists \mu: \mathcal{A} \to \mathcal{B} \text{ $\delta$-matching} \} .
            \]
        \end{defn}
        
        It can be easily shown that this is a genuine metric and not a pseudo-metric, i.e. that if $\mathcal{A},\mathcal{B}$ are two distinct barcodes, then $d_{bot}(\mathcal{A},\mathcal{B}) > 0$. Theorem~\ref{thm:dynamic_stability} relates the Hofer distance between two Hamiltonian diffeomorphisms and the bottleneck distance of the barcodes associated to their Floer homology persistence modules (see~\cite{polterovich2019topological}).
        
        \begin{thm}[Dynamical Stability Theorem]
        \label{thm:dynamic_stability}
            Let $(M,\omega)$ be a symplectic manifold with $\pi_2(M) = 0$, $\alpha \in \widehat\pi(M)$, and $\phi,\psi \in \Ham(M,\omega)$ nondegenerate. Then
            \[
                d_{bot}(\mathcal{B}(HF_*^\bullet(\phi)_\alpha), \mathcal{B}(HF_*^\bullet(\psi)_\alpha)) \leq d_{\hofer}(\phi,\psi) .
            \]
        \end{thm}

\section{Eggbeater maps}\label{sec:egg_beaters}

We first recall in subsection~\ref{sec:egg_beater_defs} the definitions of eggbeater surfaces $(C_g, \omega_0)$ and eggbeater maps on $C_g$. The eggbeater maps on surfaces $(\Sigma_g, \sigma_g)$ of genus $g$ are the images under the pushforward $(i_{\Sigma_g})_*$ of eggbeaters on $C_g$ for specific embeddings $i_{\Sigma_g}:(C_g, \omega_0) \to (\Sigma_g, \omega)$. For some eggbeater maps $\phi$ and well-chosen free homotopy classes $\alpha$, one obtains lower bounds on the length of some bars in $HF_*^\bullet(\phi)_\alpha$, which is discussed in  subsection~\ref{sec:action_gaps}, as well as computations for $\Tu^{\infty}(\alpha)$ resp. $\selfi(\alpha)$, which is discussed in subsection \ref{sec:topology}. We then  prove Theorems~\ref{thm:stable} and~\ref{thm:generic}, using tools of Section \ref{sec:floer}.
 
\color{black}

    \subsection{Definition of the eggbeater surfaces and maps}
    \label{sec:egg_beater_defs}
        The eggbeater surface $C_g$ is constructed as follows (see Figure~\ref{fig:eggbeater_def}). Fix $L \geq 4$ and denote by $C^\prime$ the cylinder of width 2 and length $L$, $[-1,1] \times \R / L\Z$, equipped with the standard symplectic form $dx \wedge dy$. Let $g \geq 2$ be an integer, which will later be the genus of the surface of interest. Denote by $C_V^\prime, C_H^\prime, C_1^\prime, C_2^\prime, C_3^\prime$ five copies of $C^\prime$, with $c_V: C^\prime \to C_V^\prime, c_H: C^\prime \to C_H^\prime, c_1: C^\prime \to C_1^\prime, c_2: C^\prime \to C_2^\prime, c_3: C^\prime \to C_3^\prime$ identity maps, and consider the squares $S_0 = [-1,1] \times [-1,1] / L\Z \subset C^\prime$, $S_1 = [-1,1] \times [\frac{L}{2}-1,\frac{L}{2}+1] / L\Z \subset C^\prime$. Define the symplectomorphism $VH: c_V(S_0) \bigsqcup c_V(S_1) \to c_H(S_0) \bigsqcup c_H(S_1)$ by $VH = VH_0 \bigsqcup VH_1$, where $VH_0: c_V(S_0) \to c_H(S_0), VH_1: c_V(S_1) \to c_H(S_1)$ are defined as follows:
        \begin{myequation}
        	VH_0(x,[y]) = (-y,[x]) , \\
        	VH_1(x,[y]) = \left(y-\frac{L}{2}, \left[-x+\frac{L}{2}\right]\right) .
        \end{myequation}
        
        The eggbeater surface $C_g$, which depends on the genus $g$ of the surface of interest, is the following disjoint union between a surface $C_V^\prime \bigcup_{VH} C_H^\prime$ and possibly other annuli:
        \begin{myequation}
        	C_2 = C_V^\prime \bigcup_{VH} C_H^\prime \bigsqcup C_1^\prime \bigsqcup C_2^\prime ,\\
        	C_3 = C_V^\prime \bigcup_{VH} C_H^\prime \bigsqcup C_1^\prime \bigsqcup C_2^\prime \bigsqcup C_3^\prime ,\\
        	\text{for $g \geq 4$: } C_g = C_V^\prime \bigcup_{VH} C_H^\prime .
        \end{myequation}

        The eggbeater surface $C_3$, for example, is depicted in Figure~\ref{fig:eggbeater_def}.
        
        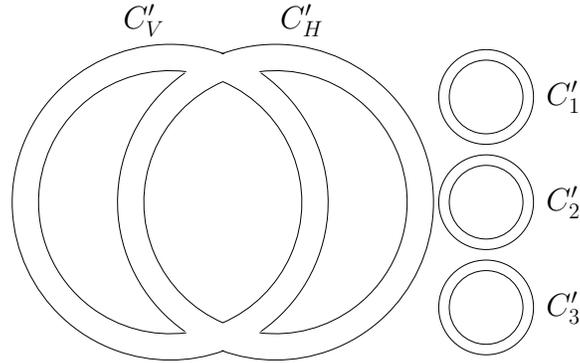
\begin{figure}
            \centering
            \scalebox{0.7}{
            \begin{tikzpicture}
                \def\a{70};
                \def\b{55};
                \def\c{83};
                \def\d{66.5};
                
                \draw (-1,0) ++(\a:3) arc (\a:360-\a:3);
                \draw (-1,0) ++(\b:3) arc (\b:-\b:3);
                \draw (-1,0) ++(\c:2.5) arc (\c:360-\c:2.5);
                \draw (-1,0) ++(\d:2.5) arc (\d:-\d:2.5);
                
                \draw (1,0) ++(180-\a:3) arc (540-\a:180+\a:3);
                \draw (1,0) ++(180-\b:3) arc (180-\b:180+\b:3);
                \draw (1,0) ++(180-\c:2.5) arc (540-\c:180+\c:2.5);
                \draw (1,0) ++(180-\d:2.5) arc (180-\d:180+\d:2.5);
                
                \draw (5,2) circle (0.9);
                \draw (5,2) circle (0.7);
                \draw (5,0) circle (0.9);
                \draw (5,0) circle (0.7);
                \draw (5,-2) circle (0.9);
                \draw (5,-2) circle (0.7);
                
                \draw (-1.5,3) node[anchor=south] {\Large $C_V^\prime$};
                \draw (1.5,3) node[anchor=south] {\Large $C_H^\prime$};
                \draw (6,2) node[anchor=west] {\Large $C_1^\prime$};
                \draw (6,0) node[anchor=west] {\Large $C_2^\prime$};
                \draw (6,-2) node[anchor=west] {\Large $C_3^\prime$};
            \end{tikzpicture}}
            \caption{The eggbeater surface $C_3$.}
            \label{fig:eggbeater_def}
        \end{figure}
        
        Eggbeater surfaces are symplectic manifolds with the standard symplectic form $\omega_0 = dx \wedge dy$ on each copy of $C^\prime$, see Figure~\ref{fig:eggbeater_def}. Denote the natural injections of individual cylinders into $C_g$ by $i_V: C_V^\prime \hookrightarrow C_g, i_H: C_H^\prime \hookrightarrow C_g$, and in the cases $g = 2,3$, also denote the natural injections $i_1: C_1^\prime \hookrightarrow C_g, i_2: C_2^\prime \hookrightarrow C_g$. If $g = 3$ denote $i_3: C_3^\prime \hookrightarrow C_g$. We will identify $C_V^\prime, C_H^\prime, C_1^\prime, C_2^\prime, C_3^\prime$ with their images in $C_g$ ($g \geq 2$); this will be clear from context. \\
        
        These eggbeater surfaces $C_g$ will, later in this section, be embedded in a closed surface of genus $g \geq 2$. A Hamiltonian function will be defined on $C_g$ which will induce some dynamics on it. We want to push forward these dynamics to the closed surface of genus $g$, by extending them with the identity map. In order for this map to be a Hamiltonian diffeomorphism, some condition on the embedding and on the Hamiltonian function on $C_g$ must apply. This condition is given in the following definition.
        
        \begin{defn}
        \label{defn:pushforward}
            Let $X,Y$ be compact topological spaces, and $i: X \hookrightarrow Y$ a continuous embedding. Let $f: X \to \R$ be a continuous map on $X$, and assume the following condition holds:
            \begin{itemize}
                \item For any path-component $C$ of $Y \setminus i(X)$, $f \restriction_{i^{-1}(\partial C)}$ is constant.
            \end{itemize}
            
            Let $C_y$ be the path-component of $Y$ that contains $y \in Y$, and denote $D_i = \bigcup_{y \in \Ima(i)} C_y \subseteq Y$. For all $y \in D_i$, denote by $\beta_{i,y}: [0,1] \to C_y$ a continuous path with $\beta_{i,y}(0) = y$, $\beta_{i,y}(1) \in \Ima(i)$, and such that if $\beta_{i,y}(t) \in \Ima(i)$ for some $t \in [0,1]$, then $\beta_{i,y} \restriction_{[t,1]}$ is constant. Note that if $y \in \Ima(i)$, then $\beta_{i,y} \equiv y$.
            
            Denote the following, not necessarily continuous, map:
            \begin{myequation}
                b_i: D_i \to \Ima(i) ,\\
                y \mapsto \beta_{i,y}(1) .
            \end{myequation}
            Define the following map, the \textit{pushforward} of $f$ through $i$:
            \begin{myequation}
                i_* f: D_i \to \R ,\\
                y \mapsto f \circ i^{-1} \circ b_i(y) .
            \end{myequation}
            By the condition on $f, i$, this is a continuous map $D_i \to \R$, and doesn't depend on the choice of the $\beta_{i,y}$s.
        \end{defn}
        
        We turn to describing the dynamics on the eggbeater surface. Consider the function $u_0: [-1,1] \to \R$ given by $u_0(s) = 1-|s|$. Take an even, non-negative, sufficiently $C^0$-close smoothing $u$ to $u_0$ such that $u$ is supported away from $\{\pm 1\}$, both $u-u_0$ and $\int_{-1}^r (u(s)-u_0(s)) ds$ are supported in a sufficiently small neighborhood of $\{\pm 1, 0\}$, and $\int_{-1}^1 (u(s)-u_0(s)) ds = 0$. For $k \in \N$, define the following autonomous Hamiltonian function on $C^\prime$:
        \begin{myequation}
        	h_k: C^\prime \to \R , \\
        	h_k(x,[y]) = -\frac{1}{2}k + k \int_{-1}^x u(s) ds .
        \end{myequation}
                
        The Hamiltonian $h_k$ induces the following five autonomous Hamiltonian functions:
        \begin{myequation}
            (i_V \circ c_V)_* h_k, \ (i_H \circ c_H)_* h_k : C_V^\prime \bigcup_{VH} C_H^\prime \to \R ,\\
            (i_1 \circ c_1)_* h_k: C_1^\prime \to \R ,\\
            (i_2 \circ c_2)_* h_k: C_2^\prime \to \R ,\\
            (i_3 \circ c_3)_* h_k: C_3^\prime \to \R .
        \end{myequation}
        
        Consider the following two autonomous Hamiltonian functions $C_g \to \R$. Abusing the notation, we will denote the two Hamiltonian functions with the same notation, for all three options $g=2,g=3,g \geq 4$, even though the definitions differ. For $g = 2$, denote:
        \begin{myequation}
            h_{V,k} = (i_V \circ c_V)_* h_k \bigsqcup - (i_1 \circ c_1)_* h_k \bigsqcup (i_2 \circ c_2)_* h_k ,\\
            h_{H,k} = (i_H \circ c_H)_* h_k \bigsqcup (i_1 \circ c_1)_* h_k \bigsqcup (i_2 \circ c_2)_* h_k .
        \end{myequation}
        
        For $g = 3$, denote:
        \begin{myequation}
            h_{V,k} = (i_V \circ c_V)_* h_k \bigsqcup (i_1 \circ c_1)_* h_k \bigsqcup (i_2 \circ c_2)_* h_k \bigsqcup 0 \restriction_{C_3^\prime} ,\\
            h_{H,k} = (i_H \circ c_H)_* h_k \bigsqcup 0 \restriction_{C_1^\prime} \bigsqcup (i_2 \circ c_2)_* h_k \bigsqcup (i_3 \circ c_3)_* h_k .
        \end{myequation}
        
        For $g \geq 4$, denote:
        \begin{myequation}
            h_{V,k} = (i_V \circ c_V)_* h_k ,\\
            h_{H,k} = (i_H \circ c_H)_* h_k .
        \end{myequation}
        
        These Hamiltonian functions generate Hamiltonian diffeomorphisms denoted $f_{V,k}, f_{H,k} \in \Ham_c(C_g, \omega_0)$ respectively.
        
        Define a homomorphism
        \begin{myequation}
        	\Psi_k: F_2 \to \Ham_c(C_g,\omega_0) ,\\
        	V \mapsto f_{V,k}, H \mapsto f_{H,k} .
        \end{myequation}
        Note that the image of a word $w = V^{N_1} H^{M_1} ... V^{N_r} H^{M_r} \in F_2$ is $f_{H,k}^{M_r} \circ f_{V,k}^{N_r} \circ ... \circ f_{H,k}^{M_1} \circ f_{V,k}^{N_1}$. We call these images \textit{eggbeater maps} 
				in $C_g$ \color{black}. \\
        
       Denote by $S_0,S_1 \subset C$ the identification of the squares $S_{V,0},S_{H,0}$ and $S_{V,1},S_{H,1}$.  In fact, $S_0 \cup S_1 = C_V^\prime \cap C_H^\prime$ (recall: $C_V^\prime$ and $C_H^\prime$ are identified with their images in $C_g$). Fix two points $s_0 \in S_0, s_1 \in S_1$. Define 4 paths: two paths $q_1,q_3$ from $s_0$ to $s_1$, and two paths $q_2,q_4$ from $s_1$ to $s_0$ as shown in Figure~\ref{fig:qs}; $q_1, q_2$ are paths on $C_V^\prime$, and $q_3,q_4$ are paths on $C_H^\prime$.

        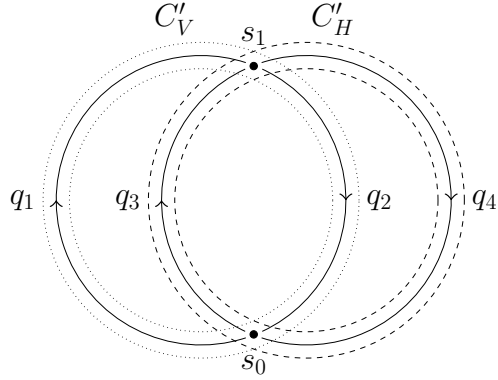
\begin{figure}
            \centering
            \scalebox{0.7}{
            \begin{tikzpicture}
                \begin{scope}[decoration={markings,mark=at position 0.5 with {\arrow[scale=2]{>}}}]
                    \draw[dotted] (-1,0) circle (3);
                    \draw[dotted] (-1,0) circle (2.5);
                    \draw (-1.5,3) node[anchor=south] {\Large $C_V^\prime$};
                    
                    \draw[dashed] (1,0) circle (3);
                    \draw[dashed] (1,0) circle (2.5);
                    \draw (1.5,3) node[anchor=south] {\Large $C_H^\prime$};
                    
                    \filldraw[black] (0,-2.55) circle (2pt);
                    \draw (0,-2.8) node[anchor=north] {\Large $s_0$};
                    \filldraw[black] (0,2.55) circle (2pt);
                    \draw (0,2.8) node[anchor=south] {\Large $s_1$};
                    
                    \draw[postaction={decorate}] (-1,0) ++(288:2.75) arc (288:72:2.75);
                    \draw[postaction={decorate}] (-1,0) ++(65:2.75) arc (65:-65:2.75);
                    \draw[postaction={decorate}] (1,0) ++(468:2.75) arc (468:252:2.75);
                    \draw[postaction={decorate}] (1,0) ++(245:2.75) arc (245:115:2.75);
                    
                    \draw (-4,0) node[anchor=east] {\Large $q_1$};
                    \draw (2,0) node[anchor=west] {\Large $q_2$};
                    \draw (-2,0) node[anchor=east] {\Large $q_3$};
                    \draw (4,0) node[anchor=west] {\Large $q_4$};
                \end{scope}
            \end{tikzpicture}}
            \caption{The paths $q_1, q_2, q_3, q_4$ in $C_g$.}
            \label{fig:qs}
        \end{figure}

        Note that $\pi_1(C_g,s_0) \simeq F_3$, the free group on 3 generators. The 3 generators $a,b,c$ of $\pi_1(C_g,s_0)$ are taken to be:
        \begin{myequation}
        a = \langle q_1  q_2 \rangle_{\pi_1(C_g,s_0)} , \\
        b = \langle q_3  q_4\rangle _{\pi_1(C_g,s_0)} , \\
        c = \langle q_3  q_2 \rangle_{\pi_1(C_g,s_0)} .
        \end{myequation}
        That is, $a$ is the class of a loop going around $C_V^\prime$, positively oriented (i.e. $a = [t\mapsto (0,Lt) \in C_V^\prime]$ as elements in $\pi_1(C_g)$), and $b$ is the class of a loop going around $C_H^\prime$, again positively oriented.
        
    \subsection{Action gaps in eggbeater maps}\label{sec:action_gaps}
        The next result is a consequence of the proof of Proposition 5.1 in~\cite{polterovich2016autonomous}. Fix $L > 4$.
        \begin{prop}
        \label{prop:action_gap}
            Let $w = VH \in F_2$. There exist certain $\nu,\mu \in (0,1)$, an unbounded subset $K \subset \N$ and a family of primitive free homotopy classes $\alpha_k = \left[ a^{\frac{\nu k}{L}} b^{\frac{\mu k}{L}} \right] \in \widehat\pi(C_g)$, for $k \in K$, such that for large enough $k \in K$, there are exactly $4$ nondegenerate fixed points of $\Psi_k(w)$ that are in class $\alpha_k$, and different such fixed points have action gaps that grow linearly with $k$: that is, for such fixed points $y,z$:
            \begin{myequation}
                |\mathcal{A}(y) - \mathcal{A}(z)| \geq c \cdot k + O(1) ,
            \end{myequation}
            as $k \to \infty$, for some global constant $c>0$.
        \end{prop}

        In order to get a similar result for eggbeater-like maps on surfaces of genus $\geq 2$, let $(\Sigma_g,\sigma_g)$ be a surface of genus $g$, equipped with symplectic form $\sigma_g$. In~\cite{polterovich2016autonomous} ($g\geq4$) and~\cite{ArnonThesis} ($g=2,3$), symplectic embeddings $i_{\Sigma_g}: (C_g, \omega_0) \hookrightarrow (\Sigma_g, \sigma_g)$ are constructed, with the property that the pairs $i_{\Sigma_g}, h_{V,k}$ and $i_{\Sigma_g}, h_{H,k}$ both satisfy the condition of Definition~\ref{defn:pushforward}.
				\footnote{To find such an embedding one in general has to multiply $\omega_0$ and the Hamiltonians in the construction with a sufficiently small $\delta >0$, and we implicitly assume this, since everything below will not depend on this rescaling.}  Thus Hamiltonian functions $(i_{\Sigma_g})_* h_{V,k}, (i_{\Sigma_g})_* h_{H,k}: \Sigma_g \to \R$ are defined. These Hamiltonian functions generate Hamiltonian diffeomorphisms on $\Sigma_g$ denoted $(i_{\Sigma_g})_* f_{V,k}, (i_{\Sigma_g})_* f_{H,k}$. Similarly to the construction of $\Psi_k$ in $C_g$, define a homomorphism
        \begin{myequation}
            (i_{\Sigma_g})_* \Psi_k: F_2 \to \Ham(\Sigma_g,\sigma_g) ,\\
            V \mapsto (i_{\Sigma_g})_* f_{V,k}, H \mapsto (i_{\Sigma_g})_* f_{H,k} .
        \end{myequation}
We say that the diffeomorphisms in $\Ham(\Sigma_g,\sigma_g)$ that are elements in the images of $(i_{\Sigma_g})_* \Psi_k$ are \textit{eggbeater maps in $\Sigma_g$}. The set of eggbeaters on $\Sigma_g$ is denoted by $\mathcal{E}_g$. \color{black}

Furthermore, we will use later the following properties of the embeddings $i_{\Sigma_g}$ given in \cite{ArnonThesis}: 
\begin{itemize}
\item $(i_{\Sigma_g})_*: \pi_1(C_g,s_0) \to \pi_1(\Sigma_g,i_{\Sigma_g}(s_0))$, $g\geq 2$, are injective;
\item In the case $g=2$, writing $\pi_1(C_2,s_0)= \langle a, b, c\rangle$, and, for a suitable choice of generators, $\pi_1(\Sigma_2,s_0)=  \langle g_1, \cdots, g_{4} \, |\, [g_1,g_2][g_3,g_4]\rangle$, the pushforward $(i_{\Sigma_2})_*:\pi_1(C_2,s_0)  \to \pi_1(\Sigma_2,i_{\Sigma_g}(s_0))$ is given by 
\begin{equation}\label{definitioni*}
(i_{\Sigma_2})_*(a)= g_1g_3,\, (i_{\Sigma_2})_*(b) = g_3 g_2 g_1^{-1} g_2^{-1}   ,\,  (i_{\Sigma_2})_*(c) = g_3; 
\end{equation}
\item The induced maps $(i_{\Sigma_g})_*:\widehat{\pi}(C_g,s_0) \to \widehat{\pi}(\Sigma_g,i_{\Sigma_g}(s_0))$ on the set of conjugacy classes of the fundamental groups,  which can be identified with free homotopy classes of loops in $C_V^\prime \bigcup_{VH} C_H^\prime$ resp. $\Sigma_g$, are injective for $g\geq 3$, and $(i_{\Sigma_2})_*$ is injective up to the relations
\begin{equation}\label{relations}
\begin{aligned}
&(i_{\Sigma_2})_*[c^j] =  (i_{\Sigma_2})_*[(ac^{-1}b)^j], \, j\in \Z \setminus \{0\},  \text{ and } \\ &(i_{\Sigma_2})_*[(ac^{-1})^j] = (i_{\Sigma_2})_*[(b^{-1}c)^j],\,  j\in \Z \setminus \{0\}. 
\end{aligned}
\end{equation}
\end{itemize}
\color{black}
				
        In~\cite{polterovich2016autonomous} and~\cite{ArnonThesis} results analogous to Proposition~\ref{prop:action_gap} in surfaces of genus $\geq 2$ are shown. The case of genera $g \geq 4$ is a consequence of Proposition 5.1 in~\cite{polterovich2016autonomous}, and the cases $g = 2,3$ are consequences of Proposition 3.9 in~\cite{ArnonThesis}. The results state:
        \begin{prop}
        \label{prop:action_gap_surface}
            Let $\Sigma_g$ be a surface of genus $g \geq 2$, and let $w = VH \in F_2$. There exist certain $\nu,\mu \in (0,1)$, an unbounded subset $K \subset \N$ and a family of primitive free homotopy classes $\alpha_k = (i_{\Sigma_g})_* \left[ a^{\frac{\nu k}{L}} b^{\frac{\mu k}{L}} \right] \in \widehat\pi(\Sigma_g)$, for $k \in K$, such that for large enough $k \in K$, there are exactly $4$ nondegenerate fixed points of $(i_{\Sigma_g})_* \Psi_k(w)$ of class $\alpha_k$, and different such fixed points have action gaps that grow linearly with $k$: that is, for such fixed points $y,z$:
            \begin{myequation}
                |\mathcal{A}(y) - \mathcal{A}(z)| \geq c \cdot k + O(1) ,
            \end{myequation}
            as $k \to \infty$, for some global constant $c>0$.
        \end{prop}
        Note that the $K, \alpha_k, \nu, \mu, c$ given by this proposition may or may not be the same as those of Proposition~\ref{prop:action_gap}. Also, for future reference, note the following key observation: if a Hamiltonian diffeomorphism $\phi$ has an action gap of $A > 0$ for fixed points of class $\alpha$ (in the sense of Proposition~\ref{prop:action_gap_surface}), then all the bars in the barcode of its Floer persistence module, $\mathcal{B}(HF_*^\bullet(\phi)_\alpha)$, are of length $\geq A$.

\begin{rem}\label{rem:entropy_eggb}
The topological entropy of eggbeater maps $\phi = (i_{\Sigma_g})_*\Psi_k(a)$, $a \in F_2$   is positive as long as $a$ is not of the form $H^n$ or $V^n$. This can be seen e.g. directly by showing the existence of a horseshoe, as in \cite{Devaney1978}, or with the results in the present paper. 
We moreover note that there is a constant $c_0>0$ such that 
$h_{\topo}(\phi) \leq \log((c_0 k)^n)$ where $n$ is the length of $a \in F_2$, and $k$ as in the definition of $\phi$. To see this, we observe that $\max_{\Sigma} \|df_{V,k}\|$ and $\max_{\Sigma} \|df_{H,k}\|$ are bounded from above by $\max_{\Sigma} \|dh_k\|$, which is bounded from above by $c_0 k$ for a constant $c_0>0$. 
Hence the exponential growth rate of the lengths $\phi(\gamma)$ of any path $\gamma$ in $\Sigma_g$ is bounded from above by $\log((c_0 k)^n)$. By \cite{Newhouse} this yields an upper bound on the topological entropy of $\phi$.  
\end{rem}    
    
    \subsection{Self-intersection number and $\Tu^{\infty}$ for a family of free homotopy classes}\label{sec:topology}
    \label{sec:self_intersections}
        We now focus our attention on the family of free homotopy classes given by Proposition~\ref{prop:action_gap_surface}, namely $(i_{\Sigma_g})_* \left[ a^m b^n \right] \in \widehat\pi(\Sigma_g)$, for $m,n \in \N$. We first compute their geometric self-intersection number and below compute a lower bound on their $\Tu^{\infty}$-growth rate, defined in section \ref{sec:Turaev}. 
    				
Recall from the introduction, that for a compact surface $M$, and a loop $\Gamma:S^1 \to M$ in general position, we denote by $\selfi(\Gamma)$ the total number of self-intersections of $\Gamma$, and that the (geometric) self-intersection number $\selfi_M(\alpha)$ of a free homotopy class $\alpha \in \widehat{\pi}(M)$ is defined to be $\min \selfi(\Gamma)$, where the minimum is taken over all self-transverse loops $\Gamma: S^1 \to M$ that represent  $\alpha$. This is well defined, since loops in general position in $M$ have finitely many self-intersections.
  
      \begin{defn}
            Let $M$ be a closed surface, and let $\Gamma: S^1 \to M$ be a loop. The loop $\Gamma$ is said to be in \textit{minimal position} if $\selfi(\Gamma) = \selfi_{M}([\Gamma]_{\widehat{\pi}(M)})$.
        \end{defn}
        
        \begin{defn}
            Let $M$ be a closed surface, and $\Gamma: S^1 \to M$ be a loop. The loop $\Gamma$ is said to form \textit{a bigon} with itself if there exists an embedding of the disc $D \hookrightarrow M$ such that the boundary of its image is the union of two arcs of $\Gamma$. Similarly, $\Gamma$ is said to form \textit{a monogon} with itself if there exists an embedding of the disc $D \hookrightarrow M$ such that the boundary of its image is a single arc of $\Gamma$.
        \end{defn}
        
        In order to check whether a given loop $\Gamma: S^1 \to M$ is in minimal position, and so calculate the self-intersection number of $[\Gamma]$, one may use the following fact (Lemma 2.2 from~\cite{HassScott1985}).
        
        \begin{fact}
        \label{cl:bigon_crit}
            Let $M$ be a closed surface, and $\Gamma: S^1 \to M$ a loop. If $\Gamma$ does not form any bigons or monogons with itself then $\Gamma$ is in minimal position.
        \end{fact}

        Construct a loop $\Gamma=\Gamma_{m,n}: S^1 \to C_g$, whose image is a subset of $i_V(C_V^\prime) \bigcup_{VH} i_H(C_H^\prime)$ and whose free homotopy class is the primitive free homotopy class $\left[ a^m b^n \right]$, as follows: $\Gamma_{m,n}$ starts at $s_0$, and performs $m$ rounds of the annulus $C_V^\prime$, while regularly spiralling inwards. After finishing $m$ rounds of $C_V^\prime$, $\Gamma$ starts making $n$ rounds of the annulus $C_H^\prime$, while regularly spiralling outwards. As $\Gamma$ finishes these rounds, it reaches the square $c_V(S_0)$ again, and connects back to $s_0$, without any further self-intersections. An example loop $\Gamma_{m,n}$ can be seen in Figure~\ref{fig:minimal_pos}.
        
        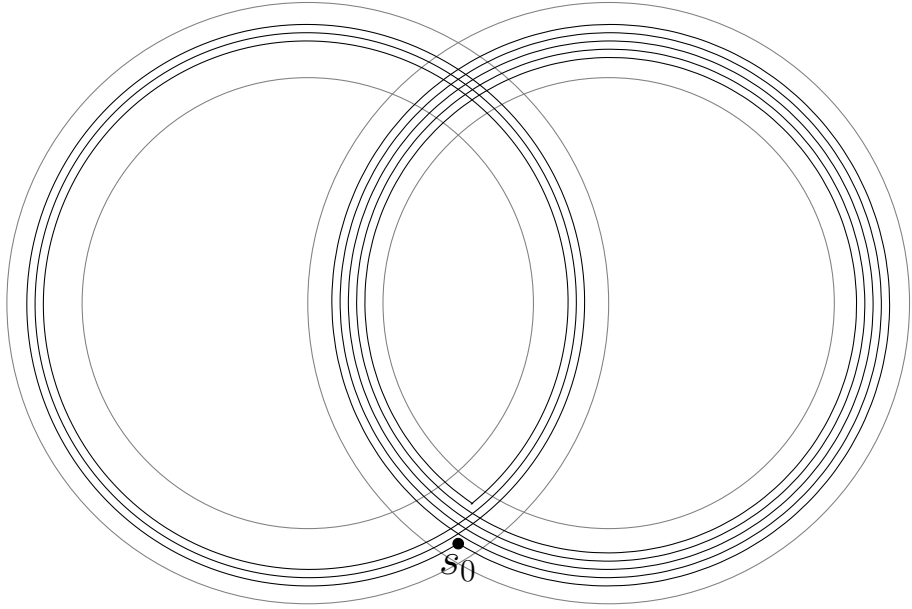
\begin{figure}
            \centering
            \begin{tikzpicture}
                \draw[gray] (-2,0) circle (4);
                \draw[gray] (-2,0) circle (3);
                \draw[gray] (2,0) circle (4);
                \draw[gray] (2,0) circle (3);
                \filldraw[black] (0,-3.2) circle (2pt) node[below] {\Large $s_0$};
                
                \draw[shift={(-2,0)}] [domain=11289:12362,variable=\t,smooth,samples=300] plot ({\t}: {-0.0003*\t});
                \draw[shift={(2,0)}] [domain=10564:12362,variable=\t,smooth,samples=300] plot ({-\t}:{0.0003*\t});
                    
            \end{tikzpicture}
            \caption{The loop $\Gamma_{3,5}: S^1 \to C_g$. This figure only shows the annuli $C_V^\prime$, $C_H^\prime$.}
            \label{fig:minimal_pos}
        \end{figure}
        
        By invoking Fact~\ref{cl:bigon_crit}, one deduces that $\Gamma_{m,n}$ is in minimal position. Counting self-intersections in $\Gamma_{m,n}$, one sees that $\selfi_{C_g}(\left[ a^m b^n \right]) = mn + (m-1)(n-1)$, where the $(m-1)(n-1)$ term comes from the self-intersections in the square $c_V(S_0)$, and the $mn$ term comes from the self-intersections in the other square $c_V(S_1)$.
        
        However, 
				we also want to calculate the self-intersection number $\selfi_{\Sigma_g}((i_{\Sigma_g})_* \left[ a^m b^n \right])$ of the pushforward of the free homotopy class considered up to this point under $i_{\Sigma_g}$, for $g \geq 2$. This is performed in the following lemma.
        
        \begin{lem}
        \label{cl:self_intersections_surface}
            Let $\Sigma_g$ be a closed surface of genus $g \geq 2$, and consider the embedding $i_{\Sigma_g}: C_g \to \Sigma_g$. The following holds:
            \[
                \selfi_{\Sigma_g}((i_{\Sigma_g})_* \left[ a^m b^n \right]) = mn + (m-1)(n-1) .
            \]
        \end{lem}
        
        \begin{proof}
            Recall that in~\cite{polterovich2016autonomous} and~\cite{ArnonThesis} the embeddings $i_{\Sigma_g}$ are shown to induce injections $(i_{\Sigma_g})_*: \pi_1(C_g,s_0) \to \pi_1(\Sigma_g,i_{\Sigma_g}(s_0))$. Assume by contradiction that $i_{\Sigma_g} \circ \Gamma_{m,n}: S^1 \to \Sigma_g$ is not in minimal position. By Fact~\ref{cl:bigon_crit}, $i_{\Sigma_g} \circ \Gamma_{m,n}$ forms a monogon or a bigon with itself.
            
            Let $\delta: S^1 \to \Sigma_g$ be the boundary of this monogon or bigon; $\delta$ is null-homotopic. Since $\Ima \delta \subset \Ima i_{\Sigma_g} \circ \Gamma_{m,n}$ and $i_{\Sigma_g}$ is injective, $i_{\Sigma_g}^{-1} \circ \delta : S^1 \to C_g$ is a well-defined loop in $C_g$, which is made up of one or two arcs of $\Gamma_{m,n}$. Since $(i_{\Sigma_g})_*$ is injective, $i_{\Sigma_g}^{-1} \circ \delta$ is also null-homotopic. This implies that $\Gamma_{m,n}$ forms a monogon or a bigon with itself, in contradiction to the above discussion.
        \end{proof}

We give now a lower bound on the growth rate $\Tu^{\infty}(\alpha)$, where $\alpha = {i_{\Sigma_g}}_*[a^mb^n] \in \widehat{\pi}(\Sigma_g)$. 
We show 
\begin{lem}\label{lem:T_comp}
$\Tu^{\infty}(\alpha) \geq \log(\lceil\frac{ mn }{2}\rceil+1)$.  
\end{lem}

\begin{proof}
We use the notation from section \ref{sec:Turaev}, and consider  
$$\mu:\pi_1(\Sigma_g,s_0) \to \mathcal{H} = \bigcup_{y\in \pi_1(\Sigma_g,s_0)} \Z[\pi_1(\Sigma_g,s_0)^*_y]\otimes  \Z[\pi_1(\Sigma_g,s_0)^*_y],$$
defined in section \ref{sec:Turaev}. Here we write $s_0$ instead of $i_{\Sigma_g}(s_0)$, also we will write $a, b, c, \Gamma_{m,n}$ instead of $(i_{\Sigma_g})_*a, (i_{\Sigma_g})_*b, (i_{\Sigma_g})_*c\in \pi_1(\Sigma,s_0)$, $i_{\Sigma_g} \circ \Gamma_{m,n}:[0,1] \to \Sigma_g$, etc. whenever the context is clear. 
We can read from the loop $\Gamma_{m,n}$ that $\mu(a^mb^n)$ has the form 
\begin{equation}\label{mu1}
\begin{split}
\mu&(a^m b^n)= \sum_{i=1}^{m-1}\sum_{k=1}^{n-1} [a^ib^k]_{a^m b^n} \, \otimes \, [a^m b^{n-k}a^{-i}]_{a^m b^n}\,   -\, [a^m b^{n-k}a^{-i}]_{a^m b^n}  \, \otimes \,   [a^ib^k]_{a^m b^n}  
+ \\ &\sum_{i=1}^{m}\sum_{k=0}^{n-1}  [a^m b^{k} c a^{-i}]_{a^m b^n} \, \otimes \, [a^i c^{-1} b^{n-k}]_{a^m b^n} \, - \, [a^i c^{-1} b^{n-k}]_{a^m b^n}  \, \otimes \, [a^m b^{k} c a^{-i}]_{a^m b^n}
\end{split}
\end{equation}
In particular, if $m=1$ and $n=1$,
\begin{equation}
\mu(ab) = [aca^{-1}]_{ab} \otimes [ac^{-1}b]_{ab} -[ac^{-1}b]_{ab}\otimes [aca^{-1}]_{ab}. 
\end{equation}

We first claim that no terms cancel in \eqref{mu1}. Note first that this is the case if the expression in \eqref{mu1} is considered as one in terms of the free group in $a,b,c$.   
Therefore, if the genus is $g \geq 3$, no cancellation of terms in \eqref{mu1}   follows from the fact that $(i_{\Sigma_g})_*:\widehat{\pi}(C_g,s_0) \to \widehat{\pi}(\Sigma_g,s_0)$ is injective.
$(i_{\Sigma_2})_*:\widehat{\pi}(C_2,s_0) \to \widehat{\pi}(\Sigma_2,s_0)$ is injective up to \eqref{relations}, so clearly one only has to check whether a cancellation occurs for the terms in the second row of \eqref{mu1} with ($i=1$, $k=n-1$) and ($i=m$, $k=0$).
If $m\geq 2$ or $n\geq 2$ one readily checks that this holds already on the level of conjugacy classes.
Consider now the remaining case $m=n=1$.
By \eqref{definitioni*}, 
$$(i_{\Sigma_2})_*(ab) =  g_1g_3^2g_2g_1^{-1}g_2^{-1},  (i_{\Sigma_2})_*(aca^{-1}) = g_1g_3g_1^{-1},   (i_{\Sigma_2})_*(ac^{-1}b) = g_1 g_3 g_2 g_1^{-1} g_2^{-1}.$$ 
So $\mu(ab) = 0$ if and only if there is $k \in \Z$ such that $f_k=1$, where 
\begin{equation}\label{rel}
f_k:= (g_1g_3^2g_2g_1^{-1}g_2^{-1})^{k}g_1g_3g_1^{-1}(g_1g_3^2g_2g_1^{-1}g_2^{-1})^{-k} g_2 g_1 g_2^{-1} g_3^{-1} g_1^{-1}.
\end{equation}  
If $k=0$ or $k=-1$, one checks that $f_k \neq 1$. If $k > 0$, after freely and cyclically reducing the word on the right-hand side of \eqref{rel} we obtain $$(g_1g_3^2g_2g_1^{-1}g_2^{-1})^{k-1} g_1 g_3 g_1^{-1} (g_1g_3^2g_2g_1^{-1}g_2^{-1})^{-k} g_2 g_1 g_2^{-1} g_3 g_2 g_1^{-1} g_2^{-1} ,$$
and if $k<-1$ after freely and cyclically reducing this word we obtain 
$$(g_1g_3^2g_2g_1^{-1}g_2^{-1})^{k+1} g_2 g_1 g_2^{-1} g_3 g_2g_1^{-1}g_2^{-1} (g_1g_3^2g_2g_1^{-1}g_2^{-1})^{-(k+2)} g_1 g_3 g_1^{-1}.$$
Since in neither case the reduced words contain a subword of more than $4$ elements which is also a subword of a cyclic permutation of a relator $[g_1,g_2][g_3,g_4]$ or its inverse, Dehn's algorithm shows $f_k \neq 1$ for all $k\in \Z$. 
Hence the terms in \eqref{mu1} do not cancel. 

To estimate $\Tu^{\infty}(\alpha)$, 
choose $$\mathcal{T} = \{ [a^m b^{k} c a^{-i}]_{a^m b^n} \, \otimes \, [a^i c^{-1} b^{n-k}]_{a^m b^n} \, | \, i=1, \cdots, m; \,  k=0, \cdots, n-1\}.$$ Then, $\mathcal{T}_+ = \{[a^m b^{k} c a^{-i}]_{a^m b^n} \, | \,  i=1, \cdots, m; \, k=0, \cdots, n-1\}$ and $\mathcal{T}_- = \{[a^i c^{-1} b^{n-k}]_{a^m b^n} \, | \,  i=1, \cdots, m;\,  k=0, \cdots, n-1\}$. 
We show that for any  $\mathfrak{S} \subset \mathcal{T}_-$ or $\mathfrak{S} \subset \mathcal{T}_+$, 
and any set $S \subset \pi_1(\Sigma_g,s_0)$ with $[S]_{a^mb^n} =\mathfrak{S}$, it holds that
$\Gamma(S\cup \{a^mb^n\}) \geq  \log(\#S + 1)$. This yields that $\Tu^{\infty}(\alpha) \geq \log(\lceil \#\frac{\mathcal{T}}{2} \rceil + 1) = \log(\lceil \#\frac{mn}{2} \rceil + 1)$. We argue for $\mathfrak{S} \subset \mathcal{T}_+$, the other case is analogous. 

Let  $P = q^{s_0} p_1 q^{s_1} p_2 \cdots p_{l'} q^{s_{l'}}$, where $p_1, \cdots, p_{l'} \in S$, $l'\in \N$, $q=a^mb^n$, and $s_0, \cdots, s_{l'} \in \N$, with $l' + s_0 + \cdots + s_{l'} =: l$. Consider $P$ as a word $w$ in the letters $a$, $b$, and $c$, and consider the corresponding reduced word $w'$.
It is clear, since $[S]_{a^mb^n} \subset \mathcal{T}_+$ and from the type of elements in $\mathcal{T}_+$, that the word $w'$ can be written as 
$$w'= w_0 c w_1 c w_2 \cdots c w_{l'},$$ for some reduced, possibly empty, words $w_0, \cdots, w_{l'}$ in $a$ and $b$. 
Moreover, if $[p_{j}] = [a^m b^{k} c a^{-i}]_{a^m b^n}$, then $w_{j-1}$ ends with $b^{k'}$ such that $k'\sim k \mod n$, and 
$w_j$ starts  with $a^{-i'}$ such that $i'\sim i \mod m$. 
Therefore, such a word $w'$ determines the original elements $p_1, \cdots p_{l'} \in S$. And so, $w'$ also determines the original choice of $s_0, \cdots, s_{l'}$. Indeed, if  $p_1 = u c v$, then $w_0 u^{-1} = (a^mb^n)^{s_0}$, etc. 
Moreover, if two such words $w'_1$ and $w_2'$, coming from products $P_1$ and $P_2$, are equal up to cyclic permutation and reduction as elements in $F_3 = \langle a, b, c\rangle$, then so are  $P_1$, $P_2$ as products of symbols in $S \cup \{q\}$.
 If the genus $g$ is $\geq 3$, then $(i_{\Sigma_g})_*: \widehat{\pi}(C_g,s_0)= \langle a, b, c\rangle/\text{conj} \to \widehat{\pi}(\Sigma_g)$ is injective and we obtain that $\widehat{N}(l,S \cup \{a^mb^n\}) \geq \sum_{i=1}^{l} \frac{(\# S+1)^l}{l}$, so $\Gamma(S\cup \{a^mb^n\}) \geq  \log(\#S + 1)$. $(i_{\Sigma_2})_*$ is injective up to relation \eqref{relations}, but for each $l$, only $4$ products could appear in those equations and hence also in this case $\Gamma(S\cup \{a^mb^n\}) \geq  \log(\#S + 1)$. 
Note that in the case $m=n=1$ clearly $\Tu^{\infty}([ab])\leq \log(2)$, and hence in fact $\Tu^{\infty}([ab])= \log(2)$. 
\end{proof}
 


\subsection{Proofs of Theorems~\ref{thm:stable} and~\ref{thm:generic}}
				 This subsection contains, as an application of Theorem~\ref{HT}, the proofs of Theorems~\ref{thm:stable} and~\ref{thm:generic}, using Proposition~\ref{prop:action_gap_surface} and the tools from Section~\ref{sec:floer}.
					
	In the following, for $R>0, \phi \in \Ham(\Sigma_g,\sigma_g)$, we denote by $B_{d_{\hofer}}(\phi,R) \subset \Ham(\Sigma_g, \sigma_g)$ the set of $\psi \in \Ham(\Sigma_g,\sigma_g)$ with $d_{\hofer}(\psi, \phi) < R$. 
	
        \begin{proof}[Proof of Theorem~\ref{thm:stable}]
				            Let $\nu,\mu \in (0,1), K \subset \N$ be those given by applying Proposition~\ref{prop:action_gap_surface}, and for $k \in K$, denote $m_k = \frac{\nu k}{L}, n_k = \frac{\mu k}{L} \in \N, \alpha_k = (i_{\Sigma_g})_* \left[ a^{m_k} b^{n_k} \right] \in \widehat\pi(\Sigma_g)$. Recall that by Proposition~\ref{prop:action_gap_surface}, $\alpha_k$ are primitive classes. There is $c>0$ such that the action gap for $(i_{\Sigma_g})_* \Psi_{k}(VH)$ guaranteed by Proposition~\ref{prop:action_gap_surface} is $\geq c k$. Note also that $m_k n_k \geq \frac{\mu\nu}{L^2} k^2$. 
Write the elements in $K$ as a sequence $(k_l)_{l\in \N}$, and let $\phi_l :=(i_{\Sigma_g})_* \Psi_{k_l}(VH)$. 
$M_l := \|\phi_l\|_{\hofer}$ grows linearly in $k_l$, since by construction of $\phi_l$   $M_l$ grows at most linearly, and by the statement on the action gaps and the Dynamical Stability theorem at least linearly in $k_l$.
By the construction $\phi_l$ the statement on the action gaps and the construction of $\phi_l$ it follows that  $M_l := \|\phi_l\|_{\hofer}$ grows linearly in $k_l$.  
Hence there are constants $\delta>0$ and $C>0$ such that the action gaps for $\phi_l$ are $\geq 2\delta M_l + 1$, and  $m_{k_l} n_{k_l} \geq 2CM_l^2$. 								
						By Lemma~\ref{lem:T_comp}, $\Tu^{\infty}(\alpha_{k_l})\geq \log(C M_l^2)$. 
						
						Let $\psi \in  B_{d_{\hofer}}(\phi_l,\delta M_l)$.
						We show that there is a fixed point $z\in \Sigma_g$ of $\psi$ in class $\alpha_{k_l}$. 
						To simplify notation write $\phi=\phi_l, \alpha = \alpha_{k_l}, M = M_l$. 
            
            Recall that Proposition~\ref{prop:action_gap_surface} guarantees us that $\phi$ has an action gap $\geq 2\delta M + 1$, and that $\phi$ has at least one fixed point in class $\alpha$. Thus, by definition, $\mathcal{B}(HF_*^\bullet(\phi)_{\alpha})$ has a bar of length $\geq 2\delta M + 1$.
            
            Let $H: S^1 \times \Sigma_g \to \R$ be a Hamitonian function that generates $\psi$. Recall that the set $\{F: S^1 \times \Sigma_g \to \R \, |\,  F \text{ is nondegenerate}\}$ is $C^\infty$-dense in the space of all functions $S^1 \times \Sigma_g \to \R$, with the $C^\infty$ topology. Thus one can take a sequence $H_j \xrightarrow{C^\infty} H$ of nondegenerate Hamiltonian functions. Denote by $\psi_j$ the Hamiltonian diffeomorphism generated by $H_j$. Note that $d_{\hofer}(\psi_j,\psi) \to 0$. Since $\psi \in B_{d_{\hofer}}(\phi,\delta M)$, there exists an integer $j_0 \in \N$ such that $\psi_j \in B_{d_{\hofer}}(\phi,\delta M)$ for $j > j_0$. By Theorem~\ref{thm:dynamic_stability}, $d_{bot}(\mathcal{B}(HF_*^\bullet(\phi)_{\alpha}), \mathcal{B}(HF_*^\bullet(\psi_j)_{\alpha})) \leq d_{\hofer}(\phi,\psi_j) < \delta M$. Thus $\mathcal{B}(HF_*^\bullet(\psi_j)_{\alpha})$ has a bar of length $\geq 1$. By definition of the barcode, this means that $\psi_j$ has a fixed point $z_j \in \Sigma_g$ in class $\alpha$ for $j > j_0$. The manifold $\Sigma_g$ is compact; take a convergent subsequence of $z_j$, $(z_j)_{j \in J}$ for some infinite $J \subset \N$, which converges to $z \in \Sigma_g$. Since $(H_j)_{j \in J} \xrightarrow{C^\infty} H$ and  by the  Arzel\`{a}-Ascoli theorem, $z$ is a fixed point of $\psi$ in class $\alpha$.
							
							Finally, by Theorem \ref{HT}, $\HP^{\infty}(\alpha) \geq \Tu^{\infty}(\alpha) \geq \log(C M_l^2)$. By a version of Ivanov's inequality, see \cite[Theorem 2.7]{Jiang1996}, we obtain that $h_{\topo}(\psi) \geq \HP^{\infty}(\alpha) \geq \log(C M_l^2)$. 
					        \end{proof}

    \begin{proof}[Proof of Theorem~\ref{thm:generic}]
	        Let $\psi \in \Ham(M,\omega)$, $\varepsilon > 0$. We find $\hat{\chi} \in \Ham(M,\omega)$ with $d_{\hofer}(\hat{\chi},\psi) < \varepsilon$ and which has some open neighborhood $V$  in $(\Ham(M,\omega), d_{\hofer})$ on which $h_{\topo} > C$; note that the union of these neighborhoods $V$ is an open and dense set with respect to $d_{\hofer}$. 

Choose $k$ sufficiently large that $\frac{\log(\selfi(\alpha_k)+1)}{16} > C$, and all fixed points of $\Psi_{k}(VH)$ in class $\alpha_k$ are nondegenerate. 
Let $\delta$ sufficiently small such that there is an embedding $i_{\Sigma_g}: (C_g, \delta\omega_0) \to (\Sigma_g, \sigma_g)$ as discussed in subsection \ref{sec:action_gaps}. Recall that when multiplying with $\delta$ we implicitly multiply the Hamiltonians $\Psi_{k}(VH)$ by $\delta$ in order not to change the dynamics. Hence we can additionally choose $\delta$ so small that  $\phi:=(i_{\Sigma_g})_*\Psi_{k}(VH)$ satisfies $\|\phi\|_{\hofer} < \varepsilon/3$. Choose a generating Hamiltonian path $\phi_t$ of $\phi$, $\phi_0 = \id$ and $\phi_1 = \phi$. 
Let $x \in \Sigma_g$ be a fixed point of $\phi$ that lies in class $\alpha_k$. 

Choose a path of generating Hamiltonian path $\psi_t$ of $\psi$. Furthermore, by multiplying the Hamiltonian that generates  $\psi_{-t}$ with a suitable cut-off function we can choose a path of Hamiltonian diffeomorphisms $\tau_t$, $\tau:=\tau_1$ with $\tau_t \circ \psi_{t}|_U = \id|_U$ for a small neighbourhood $U$ of $x$ and $\|\tau\|_{\hofer} < \varepsilon/3$. We write $\chi_{t}= \phi_t \circ \tau_t \circ \psi_t$, $\chi= \phi \circ \tau \circ \psi$. Let $H:S^1 \times \Sigma_g \to \R$ be a Hamiltonian that generates $\chi_t$. 
Then $x$ is a nondegenerate fixed point of $\chi$ in class $\alpha_k$. Let $\xi(t) = \chi_t(x)$. 
Choose tubular neighbourhoods $U_0, U$ ($\overline{U_0} \subset U$) of $\{(t,\xi(t)) \, | \, t\in  S^1 \} \subset S^1 \times \Sigma_g$ 
and $\kappa >0$ sufficiently small such that $d_0(\chi(\chi_{-t}(z)),\chi_{-t}(z)) \geq \kappa$, for all $(t,z) \in U\setminus U_0$, where $d_0$ is some fixed metric on $\Sigma_g$. Such neighbourhoods exists, since $x$ is an isolated fixed point. 
A standard perturbation argument shows, see e.g. \cite{FloerHoferSalamon}, that for any $\epsilon'>0$ there are perturbations $h:S^1 \times M \to \R$ with $\|h\|_{C^2} < \epsilon'$ and that are supported outside of $U_0$ such that all periodic orbits $\eta(t)$  of the Hamiltonian diffeomorphism $\hat{\chi}$ generated by $H+h$ for which $\{ (t,\eta(t)) \, |\, t\in S^1 \} \cap U = \emptyset$, are nondegenerate. If we choose $\epsilon'$ sufficiently small such that for any such $\hat{\chi}$ also $d_0(\hat{\chi}(\hat{\chi}_{-t}(z)), \hat{\chi}_{-t}(z)) \geq \kappa/2$, then actually all orbits of $\hat{\chi}$ are nondegenerate. If also $\epsilon'< \varepsilon/3$, one has   $d_{\hofer}(\hat{\chi}, \psi) \leq d_{\hofer}(\hat{\chi}, \chi) + d_{\hofer}(\chi, \tau \circ \psi) + d_{\hofer}(\tau \circ \psi, \psi) < \varepsilon/3 + \varepsilon/3 + \varepsilon/3 = \varepsilon$.  



			Since $\HF_*^\bullet(\hat{\chi})_{\alpha_k}$ is nonzero, its barcode has a non-zero bar, say of length  $\sigma > 0$. By the Dynamical Stability Theorem, nondegenerate diffeomorphisms in $B_{d_{\hofer}}(\hat{\chi},\sigma')$, for $\sigma'< \sigma/2$, have nonzero filtered Floer homology in class $\alpha_k$, and so have a fixed point in this class. This conclusion holds also for degenerate diffeomorphisms in $B_{d_{\hofer}}(\hat{\chi},\sigma')$ with an analogous argument as in the proof of Theorem~\ref{thm:stable}. By Theorem~\ref{intro:thm1}, this means that $h_{\topo}|_{B_{d_{\hofer}}(\hat{\chi},\sigma')} > C$, as desired.
    \end{proof}

\bibliographystyle{plain}

\end{document}